\documentclass[11pt]{amsart}
\usepackage{amsmath}
\usepackage{amsthm}
\usepackage{mathrsfs}
\usepackage{amssymb,graphics,amscd,amsfonts}
\usepackage{color}
\usepackage{comment}
\usepackage{enumerate}
\usepackage{booktabs}
\usepackage{float}
\usepackage{mathtools}

\makeatletter

\@addtoreset{equation}{section}
\makeatother

\def\Xint#1{\mathchoice
{\XXint\displaystyle\textstyle{#1}}%
{\XXint\textstyle\scriptstyle{#1}}%
{\XXint\scriptstyle\scriptscriptstyle{#1}}%
{\XXint\scriptscriptstyle\scriptscriptstyle{#1}}%
\!\int}
\def\XXint#1#2#3{{\setbox0=\hbox{$#1{#2#3}{\int}$}
\vcenter{\hbox{$#2#3$}}\kern-.5\wd0}}
\def\dashint{\Xint-}

\theoremstyle{plain}
\newtheorem{theorem}{Theorem}[subsection]
\newtheorem{proposition}[theorem]{Proposition}
\newtheorem{corollary}[theorem]{Corollary}
\newtheorem{lemma}[theorem]{Lemma}

\newtheorem{example}[theorem]{Example}
\newtheorem{definition}[theorem]{Definition}
\newtheorem{remark}[theorem]{Remark}
\newtheorem{fact}[theorem]{Fact}

\theoremstyle{definition}

\DeclareMathOperator{\Spec}{Spec}
\DeclareMathOperator{\Proj}{Proj}

\title[A uniform version of the Yau-Tian-Donaldson correspondence ]{A uniform version of the Yau-Tian-Donaldson correspondence for extremal K\"ahler metrics on polarized toric manifolds}
\author{Yasufumi Nitta}
\author{Shunsuke Saito}

\address{Department of Mathematics, Faculty of Science, Tokyo University of Science, 1-3 Kagurazaka, Shinjuku-ku, Tokyo 162-8601, Japan}
\email{nitta@rs.tus.ac.jp}

\address{Department of Mathematics, Faculty of Science, Tokyo University of Science, 1-3 Kagurazaka, Shinjuku-ku, Tokyo 162-8601, Japan}
\email{saito@rs.tus.ac.jp}


\begin{document}
\begin{abstract}
The aim of this paper is to solve a uniform version of the Yau-Tian-Donaldson conjecture for polarized toric manifolds. 
Also, we show a combinatorial sufficient condition for uniform relative K-polystability. 
\end{abstract}
\maketitle
\setcounter{tocdepth}{2}
\tableofcontents
\section{Introduction}
  Let $(X, L)$ be an $n$-dimensional polarized algebraic manifold. We denote $\mathcal{H}(X, L)$ by the set of all K\"ahler metrics of $X$ in $c_{1}(L)$. A K\"ahler metric $\omega\in \mathcal{H}(X, L)$ is called a \emph{Calabi's extremal K\"ahler metric} if it is a critical point of the Calabi functional 
\begin{align*}
\mathcal{H}(X, L)\ni\alpha\mapsto\int_{X}(s(\alpha)-\overline{s})^{2}\,\frac{\alpha^{n}}{n!}\in\mathbf{R}, 
\end{align*}
where $s(\alpha)$ is a scalar curvature of $\alpha$ and $\overline{s} = -n(K_{X}\cdot L^{n-1})/(L^{n})$ is its average. 
Calabi showed in \cite{Ca82} that this condition is equivalent to that $\operatorname{grad}_{\omega}^{(1, 0)}(s(\omega)-\overline{s})$ is a holomorphic vector field on $X$. Extremal K\"ahler metrics have been widely studied as canonical K\"ahler metrics containing constant scalar curvature K\"ahler metrics and in particular K\"ahler-Einstein metrics. Especially, the existence problem of extremal K\"ahler metrics is known as the \emph{Yau-Tian-Donaldson conjecture}, and is one of the central problem in K\"ahler geometry. 

 The Yau-Tian-Donaldson conjecture predicts that the existence of an extremal K\"ahler metric is equivalent to the stability of $(X, L)$ in some sense of geometric invariant theory. Sz\'ekelyhidi introduced in \cite{S07} the notion of \emph{relative K-polystability} as a candidate of such a stability condition. 
Roughly speaking, relative K-polystability requires the positivity of the \emph{relative Donaldson-Futaki invariant} $DF_{V}>0$ for any nontrivial test object. Here $V$ is the extremal K\"ahler vector field of $(X, L)$ (\cite{FM95}), and test object is a special kind of degeneration of $(X, L)$, which is called a \emph{test configuration}. 

  On the other hand, calculations by Apostolov-Calderbank-Gauduchon-T\o nnesen-Friedman \cite{ACGTF08} suggest that relative K-polystability might not be sufficient to ensure the existence of a extremal K\"ahler metric, and leads to an expectation that we need to strengthen the notion of relative K-polystability. 
For strengthenings, the following two approaches are known: 
\begin{enumerate}[\upshape(a)]
\item Requiring that $DF_{V}$ is bounded from below by a multiple of some `norm' of test configurations (\cite{S06}, \cite{Der16}, \cite{BHJ17, BHJ19}). 
\item Testing $DF_{V}>0$ for more general test objects than test configurations (\cite{S15}, \cite{M15}). 
\end{enumerate}
In this paper, we will show the equivalence of both approaches for polarized toric manifolds. As a consequence, we will solve a uniform version of the Yau-Tian-Donaldson conjecture for polarized toric manifolds. 

An $n$-dimensional polarized toric manifold $(X, L)$ corresponds to an $n$-dimensional integral Delzant polytope $P \subset \mathbf{R}^{n}$ which is written by an intersection of half spaces 
\begin{align}\label{Delzant_polytope}
P = \{x \in \mathbf{R}^{n} \mid \langle \lambda_{j}, x \rangle + d_{j} \geq 0\ (j=1, \ldots, r)\}, 
\end{align}
where $\langle \cdot, \cdot \rangle$ is the standard inner product on $\mathbf{R}^{n}$, $r$ is the number of facets of $P$, $\lambda_{j} \in \mathbf{Z}^{n}$, $d_{j} \in \mathbf{Z}$ and each $\lambda_{j}$ is a primitive vector. 
Then, any toric K\"ahler metric representing $c_{1}(L)$ corresponds to a smooth strictly convex function on $P^{\circ}$ which satisfies the \emph{Guillemin's boundary condition} (see Section \ref{sec:SympPot}). 
Also, any toric test configuration $(\mathcal{X}, \mathcal{L})$ for $(X, L)$ corresponds to a rational piecewise affine function $f$ on $P$, and the relative Donaldson-Futaki invariant $DF_{V}(\mathcal{X}, \mathcal{L})$ of $(\mathcal{X}, \mathcal{L})$ can be written as 
\begin{align*}
L_{V}(f) = \int_{\partial P}f\,d\sigma-\int_{P}(\overline{s}+V)f\,dx, 
\end{align*}
where $\sigma$ is the Borel measure on $\partial P$ which is the multiple of the $(n-1)$-dimensional Lebesgue measure by $|\lambda_{j}|^{-1}$ on each facet $F_{j} = \{x \in P \mid \langle \lambda_{j}, x \rangle + d_{j} = 0\}$ (see \cite{ZZ08}). 

Let $\mathcal{S}$ be the set of all strictly convex functions on $P^{\circ}$ which satisfy the Guillemin's boundary condition. 
Let $\mathcal{C}_{PL}^{\mathbf{Q}}$ be the set of all rational piecewise affine functions on $P$. As mentioned above, we can regard $\mathcal{C}_{PL}^{\mathbf{Q}}$ as the set of all toric test configurations for $(X, L)$. 
Also, we define $\mathcal{C}_{PL}$ to be the set of all piecewise affine functions on $P$. 
Let $\mathcal{C}_{\infty}$ be the set of all continuous convex functions on $P$ which is smooth in the interior. 
Further, let $P^{\ast}$ be the union of relative interiors of faces of $P$ up to codimension one, and $\mathcal{C}_{\ast}$ be the set of all continuous convex functions on $P^{\ast}$ which are integrable on $\partial P$ with respect to the measure $\sigma$. 
We fix a point $x_{0}\in P$ in the interior, and denote by $\widetilde{\mathcal{C}}_{\ast}$ the set of all $f \in \mathcal{C}_{\ast}$ satisfying the \emph{normalized condition} $\inf_{P^{\ast}}f=f(x_{0})=0$. 
Also, set $\widetilde{\mathcal{F}}= \mathcal{F}\cap \widetilde{\mathcal{C}}_{\ast}$ for any $\mathcal{F} \subset \mathcal{C}_{*}$. 
Finally, for each $f\in\mathcal{C}_{\ast}$ we denote the \emph{$J$-norm} by $\|f\|_{J}$ (see Section \ref{sec:RedJ}). 

Under the notations above, our first main result is the following equivalence. 
\begin{theorem}\label{equiv_strengthenings}
Let $\mathcal{F}$ be one of the spaces $\mathcal{E}_{1}$, $\mathcal{C}_{\ast}$, $\mathcal{C}_{PL}$, $\mathcal{C}_{PL}^{\mathbf{Q}}$, $\mathcal{C}_{\infty}$ or $\mathcal{S}$. 
Then the followings are equivalent. 
\begin{description}
\item[$(b)_{\mathcal{F}}$] There exists a $\delta > 0$ such that $L_{V}(f)\ge \delta \int_{\partial P}f\,d\sigma$ for any $f\in\widetilde{\mathcal{F}}$. 
\item[$(J)_{\mathcal{F}}$] There exists a $\delta > 0$ such that $L_{V}(f)\ge \delta \|f\|_{J}$ for any $f\in\mathcal{F}$. 
\item[$(\mathrm{K})_{\mathcal{C}_{\ast}}$] $L_{V}(f)\ge 0$ for any $f\in\mathcal{C}_{\ast}$ and $L_{V}(f) = 0$ if and only if $f$ is affine. 
\item[$(\mathrm{K})_{\mathcal{E}_{1}}$] $L_{V}(f)\ge 0$ for any $f \in \mathcal{E}_{1}$ and $L_{V}(f) = 0$ if and only if $f$ is affine. 
\end{description}
\end{theorem}
Note that the proof of Theorem \ref{equiv_strengthenings} does not require the Delzant condition and that this equivalence also holds for polarized toric \emph{varieties}. 

Let $(X, L)$ be a polarized toric manifold associated to an integral Delzant polytope $P \subset \mathbf{R}^{n}$. 
We say that $(X, L)$ is \emph{uniformly relatively K-polystable} if $P$ satisfies the condition $(J)_{\mathcal{C}_{PL}^{\mathbf{Q}}}$. 
$(X, L)$ is \emph{uniformly K-polystable} if in addition that $V = 0$. 
By combining Theorem \ref{equiv_strengthenings} with the works of \cite{CLS14}, \cite{ZZ08}, \cite{He18} and \cite{Leg19}, we can prove a uniform version of the Yau-Tian-Donaldson correspondence for polarized toric manifolds. 
\begin{theorem}[The toric Yau-Tian-Donaldson correspondence]\label{toric_YTD}
  Let $(X, L)$ be an $n$-dimensional polarized toric manifold. 
Let $T$ be a maximal algebraic torus of $\mathrm{Aut}^{0}(X, L)$, and $S$ be the maximal compact subgroup. 
Then the followings are equivalent. 
\begin{enumerate}[\upshape(1)]
\item $(X, L)$ admits an $S$-invariant extremal K\"ahler metric. 
\item $(X, L)$ is uniformly relatively K-polystable. 
\item The relative K-energy of $(X, L)$ is $T$-coercive. 
\end{enumerate}
\end{theorem}
The implications (1) $\Rightarrow$ $(b)_{\mathcal{C}_{\infty}}$ and $(b)_{\mathcal{C}_{\infty}}$ $\Rightarrow$ (3) are proved in \cite{CLS14} and \cite{ZZ08}, respectively. 
The implication (3) $\Rightarrow$ (1) can be obtained by the work of He \cite{He18} and its toric reduction by Legendre (see Section 3.2 of \cite{Leg19}). 
Finally, the equivalence $(J)_{\mathcal{C}_{PL}^{\mathbf{Q}}} \Leftrightarrow (b)_{\mathcal{C}_{\infty}}$ is proved in Theorem \ref{equiv_strengthenings}. 

If we consider the case that the extremal K\"ahler vector field of $(X, L)$ is zero, we obtain the following: 

\begin{corollary}\label{toric_YTD_cscK}
Let $(X, L)$, $T$, $S$ be as above. 
Then the followings are equivalent. 
\begin{enumerate}[\upshape(1)]
\item $(X, L)$ admits an $S$-invariant K\"ahler metric of constant scalar curvature. 
\item $(X, L)$ is uniformly K-polystable. 
\item The K-energy of $(X, L)$ is $T$-coercive. 
\end{enumerate}
\end{corollary}

We mention the very recent papers \cite{Jubert} and \cite{LLS}, which have some overlaps with the present paper. 

As a simple application of Theorem \ref{toric_YTD}, we obtain the following product theorem for uniform relative K-polystability. 

\begin{corollary}
Let $(X_{1}, L_{1})$ and $(X_{2}, L_{2})$ be polarized toric manifolds, and set $X = X_{1} \times X_{2}$ and $L = L_{1} \boxtimes L_{2}$. If $(X_{1}, L_{1})$ and $(X_{2}, L_{2})$ are uniformly relatively K-polystable, then so is $(X, L)$. 
\end{corollary}

In fact, if $(X_{1}, L_{1})$ and $(X_{2}, L_{2})$ are uniformly relatively K-polystable, then they admit torus invariant extremal K\"ahler metrics by Theorem \ref{toric_YTD}. 
Since the product metric of extremal K\"ahler metrics is also an extremal K\"ahler metric, $(X, L)$ is uniformly relatively K-polystable by Theorem \ref{toric_YTD} again. 

Next, we give a combinatorial sufficient condition for unirofm relative K-polystability. 
Let $(X, L)$ be a polarized toric manifold associated an integral Delzant polytope $P$ defined by \eqref{Delzant_polytope}. 
We fix a point $x_{0}\in P$ in the interior, and put $d_{x_{0}, j} = \langle \lambda_{j}, x_{0}\rangle + d_{j}$ for each $j=1, \ldots, r$. 
Note that $d_{x_{0}, j} > 0$ for any $j$. 
Then the following theorem strengthens \cite[Theorem 0.1]{ZZ08}, where relative K-polystability was deduced under the same assumptions. 
\begin{theorem}\label{suff}
Let $d_{x_{0}}= \max\{d_{x_{0}, 1}, \ldots, d_{x_{0}, r}\}$. 
Suppose that $P$ satisfies either 
\begin{align}\label{unif.stab}
V=0\quad \text{and} \quad\overline{s}< \frac{n+1}{d_{x_{0}}}, 
\end{align}
or 
\begin{align}\label{unif.rel.stab}
V \neq 0\quad \text{and}\quad \overline{s} + \max_{P}V \le \frac{n+1}{d_{x_{0}}}. 
\end{align}
Then $(X, L)$ is uniformly relatively K-polystable. 
\end{theorem}
In case that $(X, L) = (X, -K_{X})$ is a toric Fano manifold with anticanonical polarization, then $\overline{s} = n$ and $d_{j} = 1$ for all $j$. 
Moreover, we can choose $x_{0}$ to be the origin of $\mathbf{R}^{n}$. 
Hence the conditions \eqref{unif.stab} and \eqref{unif.rel.stab} reduce 
\begin{align}\label{unif.stab.Fano}
V=0 
\end{align}
and
\begin{align}\label{unif.rel.stab.Fano}
&V \neq 0\quad \text{and}\quad \max_{P}V \le 1, 
\end{align}
respectively. 
Combining this with the computations in \cite{ZZ08} and \cite{NSY19} yields the following: 
\begin{corollary}\label{classification_K}
\begin{enumerate}[(1)]
\item All toric del Pezzo surfaces are uniformly relatively K-polystable. 
\item In 3-dimensional case, at least 13 toric Fano manifolds out of 18 are uniformly relatively K-polystable. 
\end{enumerate}
\end{corollary}
It is known that all toric del Pezzo surfaces admit an extremal K\"ahler metric in its anticanonical class by the works of \cite{Ca82} (degree 8 case), \cite{Siu88} and \cite{TY87} (degree 6 case) and \cite{CLW08} (degree 7 case). 
Corollary \ref{classification_K} gives the stability counterpart of this fact. 

We finally concern uniform K-polystability of toric Fano manifolds. 
Note that the condition \eqref{unif.stab.Fano} is equivalent to that the Futaki invariant of $X$ vanishes (\cite{FM95}). 
Also, it is known that the Futaki invariant vanishes if and only if the barycenter of $P$ coincides with the origin of $\mathbf{R}^{n}$ (\cite{M87b}). 
Thus, 
\begin{corollary}\label{stability_Fano_cases_vanFtaki} 
Let $(X, -K_{X})$ be an $n$-dimensional toric Fano manifold with anticanonical polarization associated to an integral Delzant polytope $P$. 
Then the followings are equivalent: 
\begin{enumerate}[(1)]
\item $(X, -K_{X})$ is uniformly K-polystable. 
\item $(X, -K_{X})$ is K-polystable. 
\item $(X, -K_{X})$ is K-semistable. 
\item The Futaki invariant of $X$ vanishes. 
\item The barycenter of $P$ coincides with the origin of $\mathbf{R}^{n}$. 
\item $V = 0$. 
\end{enumerate}
\end{corollary}

\section{Background materials}
\subsection{Notation and convensions}
Throughout this paper, we use the following notation and convensions: 
In this paper, we mean an $n$-dimensional polarized manifold a pair $(X, L)$ of an $n$-dimensional compact complex manifold $X$ and an ample holomorphic line bundle $L$ over $X$. 
\begin{itemize}
\item $S^{1} = \mathbf{R}/\mathbf{Z}$, $\mathrm{Lie}(S^{1}) = 2\pi\sqrt{-1}\mathbf{R}$, $\mathrm{vol}(S^{1}) = 1$. 
\item $\mathbf{C} = (-1/2)\mathbf{R} \oplus 2\pi\sqrt{-1}\mathbf{R} = (-\sqrt{-1}/4\pi) \mathrm{Lie}(S^{1}) \oplus \mathrm{Lie}(S^{1})$. 
\item $\mathbf{C}^{\times} = \mathbf{C}\setminus\{0\} = \exp((-1/2)\mathbf{R} \oplus 2\pi\sqrt{-1}\mathbf{R})$. 
\item $\omega(\cdot, \cdot) = (1/2\pi)g(J\cdot, \cdot)$ for any K\"ahler metric $g$. 
Here $J$ is the integrable almost complex structure of the complex manifold $X$. 
\item $\Delta = \Delta_{d} = 2\Delta_{\bar \partial} = \nabla^{i}\nabla_{i} = \mathrm{tr}_{g}\mathrm{Hess}_{g}$ for any K\"ahler metric $g$. 
\end{itemize}

Let  $T$ be a maximal algebraic torus of $\mathrm{Aut}^{0}(X)$, with $\mathfrak{t} = \mathrm{Lie}(T)$. 
\begin{itemize}
\item $\mathbf{A}^{1} = \Spec \mathbf{C}[t]$, $\mathbf{G}_{m} = \Spec \mathbf{C}[t, t^{-1}]$. 
\item $N = \mathrm{Hom}(\mathbf{G}_{m}, T)$, which is a lattice of rank $\dim_{\mathbf{C}}T$. 
\item $N_{k} = N \otimes_{\mathbf{Z}}k$ for $k = \mathbf{Q}$ or $\mathbf{R}$. 
\item $M = \mathrm{Hom}(N, \mathbf{Z})$, $M_{k} = M \otimes_{\mathbf{Z}}k$ for $k = \mathbf{Q}$ or $\mathbf{R}$. 
\item $S = N \otimes_{\mathbf{Z}}S^{1}$, which is the maximal compact subgroup of $T$, with $\mathfrak{s} = \mathrm{Lie}(S) = 2\pi\sqrt{-1}N_{\mathbf{R}}$, which satisfies $\mathfrak{t} = (-\sqrt{-1}/4\pi)\mathfrak{s} \oplus \mathfrak{s} = (-1/2)N_{\mathbf{R}} \oplus 2\pi\sqrt{-1}N_{\mathbf{R}}$. 
Note that $M_{\mathbf{R}} = \mathfrak{s}^{\vee}$. 
\end{itemize}

\subsection{Extremal K\"ahler metrics}
In this section, we recall the definition of extremal K\"ahler metrics. 
Let $(X, L)$ be an $n$-dimensional polarized manifold, and fix a maximal algebraic torus $T$ of $\mathrm{Aut}^{0}(X)$. 
We denote by $\mathcal{H}(X, L)^{S}$ the set of all $S$-invariant K\"ahler metrics $\omega$ such that $\omega \in c_{1}(L)$. 
By setting $\overline{s} = -n(K_{X}\cdot L^{n-1})/(L^{n})$, we define $\Phi \colon \mathcal{H}(X, L)^{S} \to \mathbf{R}$ by 
\begin{align*}
\Phi(\omega) = \dashint_{X}(S(\omega) - \overline{s})^{2}\omega^{n} = \frac{1}{\mathrm{vol}_{\omega}(X)}\int_{X}(s(\omega) - \overline{s})^{2}\omega^{n}, 
\end{align*}
where $\mathrm{vol}_{\omega}(X) = \int_{X}\omega^{n}$ and $\overline{s} = \dashint_{X}s(\omega)\omega^{n}$. 
Note that $\mathrm{vol}_{\omega}(X)$ and $\overline{s}$ is independent of choice of $\omega \in \mathcal{H}(X, L)^{S}$ since $\mathrm{vol}_{\omega}(X) = (L^{n}) = \langle c_{1}(L)^{n}, [X]\rangle$ and $\overline{s} = -n(K_{X}\cdot L^{n-1})/(L^{n})$. 
We call $\Phi$ the Calabi functional. 
\begin{definition}[\cite{Ca82}]
A K\"ahler metric $\omega \in \mathcal{H}(X, L)^{S}$ is called an extremal  K\"ahler metric if $\omega$ is a critical point of $\Phi$. 
\end{definition}
K\"ahler-Einstein metrics are extremal metrics, and (more generally) constant scalar curvature K\"ahler metrics in $\mathcal{H}(X, L)^{S}$ are typical examples of extremal K\"ahler metrics. 
The following fact is well-known. 
\begin{proposition}[{\cite[Theorem 2.1]{Ca82}}]
A K\"ahler metric $\omega \in \mathcal{H}(X, L)^{S}$ is an extremal  K\"ahler metric if and only if $\mathrm{grad}_{g}(s(\omega) - \overline{s})$ is a real holomorphic vector field on $X$. 
\end{proposition}
We fix $\omega \in \mathcal{H}(X, L)^{S}$. 
Set
\begin{align*}
&C^{\infty}(X, \omega, \mathbf{R})_{0}^{S} = \{f \in C^{\infty}(X, \mathbf{R}) \mid \int_{X}f\omega^{n} = 0\}, \\
&\mathfrak{s}(X, \omega)_{0} = \{u \in C^{\infty}(X, \omega, \mathbf{R})_{0}^{S} \mid 4\pi J \mathrm{grad}_{g}u \in \mathfrak{s}\}. 
\end{align*}
Conversely, for each $W \in \mathfrak{s}$ there exists unique $u_{\omega}^{W} \in C^{\infty}(X, \mathbf{R})^{S}$ such that 
\begin{align*}
i_{W}\omega = -du_{\omega}^{W},\quad \int_{X}u_{\omega}^{W}\omega^{n} = 0. 
\end{align*}
Then we have 
\begin{align*}
-du_{\omega}^{W} = g(JJ\mathrm{grad}_{g}u_{\omega}^{W}, \cdot) = i_{2\pi J\mathrm{grad}_{g}u_{\omega}^{W}}\omega
\end{align*}
and $2\pi J\mathrm{grad}_{g}u_{\omega}^{W} = W \in \mathfrak{s}$, which implies $u_{\omega}^{W} \in \mathfrak{s}(X, \omega)_{0}$. 
Hence the correspondence $\mathfrak{s} \ni W \mapsto u_{\omega}^{W} \in \mathfrak{s}(X, \omega)_{0}$ gives a Lie algebra isomorphism. 
The following is also ovbious. 
\begin{proposition}
A K\"ahler metric $\omega \in \mathcal{H}(X, L)^{S}$ is extremal if and only if $s(\omega) - \overline{s} \in \mathfrak{s}(X, \omega)_{0}$. 
\end{proposition}

\subsection{The Futaki invariant and the extremal K\"ahler vector field}
Let $(X, L)$ be a polarized manifold and $\omega \in \mathcal{H}(X, L)^{S}$. 
Since $\mathfrak{t} = (-1/4\pi)J\mathfrak{s} \oplus \mathfrak{s}$, every $W \in \mathfrak{t}$ can be written as 
\begin{align*}
W = -\frac{1}{4\pi}JW_{1} + W_{2}
\end{align*}
for some $W_{1}, W_{2} \in \mathfrak{s}$. 
Then we define $u_{\omega}^{W} \in C^{\infty}(X, \omega, \mathbf{C})_{0}^{S}$ by 
\begin{align*}
u_{\omega}^{W} = -\frac{1}{4\pi}\sqrt{-1}u_{\omega}^{W_{1}} + u_{\omega}^{W_{2}}. 
\end{align*}
\begin{definition}[\cite{Fut83a}, \cite{Fut83b}, \cite{Ca85}, \cite{Ba06}]
We call the Lie algebra character $F \colon \mathfrak{t} \to \mathbf{C}$ defined by 
\begin{align*}
F(W) 
= \dashint_{X}u_{\omega}^{W}(s(\omega)-\overline{s})\omega^{n} 
\end{align*}
the Futaki invariant of $(X, L)$. 
\end{definition}
We remark that (i) the definition of $F$ is independent of choice of $\omega \in \mathcal{H}(X, L)^{S}$, (ii) if $\mathcal{H}(X, L)^{S}$ contains a K\"ahler metric of constant scalar curvature, then $F$ vanishues identically, and (iii) an extremal K\"ahler metric $\omega \in \mathcal{H}(X, L)^{S}$ has constant scalar curvature if and inly if $F = 0$ (see \cite{Fut83a}, \cite{Fut83b}, \cite{Ca85}, \cite{Ba06}). 

Let $\mathrm{pr}_{\omega} \colon C^{\infty}(X, \omega, \mathbf{R})_{0}^{S} \to \mathfrak{s}(X, \omega)_{0}$ be the $L^{2}$-orthogonal projection with respect to the volume form $\omega^{n}$. 
\begin{definition}[\cite{FM95}]
Let $\theta_{\omega} = \mathrm{pr}_{\omega}(s(\omega) - \overline{s}) \in \mathfrak{s}(X, \omega)_{0}$. 
We call the real holomorphic vector field defined by 
\begin{align*}
V \coloneqq \mathrm{grad}_{g}\theta_{\omega} \in N_{\mathbf{R}}
\end{align*}
the extremal K\"ahler vector field of $(X, L)$ with respect to $T$. 
\end{definition}

It is known that (i) the definition of $V$ is independent of choice of $\omega \in \mathcal{H}(X, L)^{S}$, (ii) a K\"ahler metric $\omega \in \mathcal{H}(X, L)^{S}$ is extremal if and only if $s(\omega) - \overline{s} = \theta_{\omega}$, and (iii) $\mathrm{grad}_{g}\theta_{\omega} \in N_{\mathbf{Q}}$, that is, $\exp(mV) =\mathrm{id}_{X}$ for some $m \in \mathbf{Z}_{>0}$ (see \cite{FM95} and \cite{Nak99}). 
\begin{remark}
We remark that the definition of $V$ depends only on the choice of $T$ and independent of neither the choice of $K$ nor $K$-invariant K\"ahler metric in $c_{1}(L)$. 
Indeed, by \cite{FM95} $V$ is independent of the choice of $K$-invariant K\"ahler metric in $c_{1}(L)$ for fixed $K$. 
If $K'$ is another choice, then $K$ and $K'$ have a common maximal torus $S$. 
Hence there exists $\tau \in T$ so that $K' = \tau K \tau^{-1}$ by \cite[Proposition A.2]{Li19}. 
Finally, by \cite[Corollary E]{FM95} we obtain $V' = \mathrm{Ad}(\tau)V = V$. 
\end{remark}

\begin{definition}[\cite{FM95}]
We define $B \colon \mathfrak{t} \times \mathfrak{t} \to \mathbf{C}$ by 
\begin{align*}
B(W, W') = \dashint_{X}u_{\omega}^{W}u_{\omega}^{W'}\omega^{n}
\end{align*}
and call the Futaki-Mabuchi bilinear form. 
\end{definition}
It is known that (i) the definition of $B$ is independent of choice of $\omega \in \mathcal{H}(X, L)^{S}$, (ii) $B$ is a nondegenerate and symmetric bilinear form on $\mathfrak{t}$, and (iii) for any $W \in \mathfrak{t}$
\begin{align*}
B(W, V) = -F(W), 
\end{align*}
namely, $F$ is dual to $-V$ (see \cite{FM95}). 
In paricular, the Futaki invariant vanishes identically if and only if $V = 0$. 

\subsection{Riemannian geometry of the space of K\"ahler potentials} 
In \cite{M87a}, Mabuchi studied a natural Riemannian structure on the space of K\"ahler potentials. 
Let $(X, L)$ be an $n$-dimensional polarized manifold and fix $\omega_{0} \in \mathcal{H}(X, L)^{S}$. 
By the well-known $dd^{c}$-lemma, $\mathcal{H}(X, L)^{S}$ can be identified the space 
\begin{align*}
\mathcal{H}^{S} = \mathcal{H}(X, \omega_{0})^{S} \coloneqq \{\varphi \in C^{\infty}(X, \mathbf{R})^{S} \mid \omega_{\varphi} \coloneqq \omega_{0} + dd^{c}\varphi > 0\}. 
\end{align*}
If $\varphi \in \mathcal{H}^{S}$, the tangent space $T_{\varphi}\mathcal{H}^{S}$ at $\varphi$ can be identified with $C^{\infty}(X, \mathbf{R})^{S}$. 
Then the natural Riemannian metric on $T_{\varphi}\mathcal{H}^{S}$ is given by 
\begin{align*}
\langle u, v \rangle_{\varphi} \coloneqq \dashint_{X}uv\omega_{\varphi}^{n}. 
\end{align*}
Let $\{\varphi^{t}\}$ be a smooth path in $\mathcal{H}^{S}$, and $\{u^{t}\}$ be a smooth vector filed of $\mathcal{H}^{S}$ along the path $\{\varphi^{t}\}$, namely, $u^{t} \in T_{\varphi^{t}}\mathcal{H}^{S} = C^{\infty}(X, \mathbf{R})^{S}$ for each $t$. 
Then the Levi-Civita connection is given by
\begin{align*}
\nabla_{\dot{\varphi}^{t}}u^{t} = \dot{u}^{t}-\frac{1}{2}\langle d\dot{\varphi}^{t}, du^{t}\rangle_{g_{\varphi^{t}}}. 
\end{align*}
This immidiately implies that $t \mapsto \varphi^{t}$ is a geodesic if and only if $\nabla_{\dot{\varphi}^{t}}\dot{\varphi}^{t} = 0$, or equivalently 
\begin{align}\label{geod_eq}
\ddot{\varphi}^{t}-\frac{1}{2}|d\dot{\varphi}^{t}|_{g_{\varphi^{t}}}^{2} = 0. 
\end{align}
As discovered by Semmes \cite{Sem92} and Donaldson \cite{Don99}, the above equation can be understood as a complex Monge-Amp\`ere equation as follows. 
For each $T \in (0, \infty]$, let $A = \{w \in \mathbf{C} \mid e^{-T/2} \leq |w| \leq 1\}$. 
Here we propose $A = \{w \in \mathbf{C} \mid 0 < |w| \leq 1\}$ if $T = \infty$. 
By setting $w = \exp(-t/2 + 2\pi\sqrt{-1}\theta)$, we have $t = -\log|w|^{2}$ and 
\begin{align*}
e^{-T/2} \leq |w| \leq 1 \iff 0 \leq t \leq T. 
\end{align*}
Let $\overline{X} = X \times A$. 
We regard $\overline{X}$ as a complex manifold with boundary. 
Let $\pi \colon \overline{X} \to X$ be the natural projection. 
For each smooth path $\{\varphi^{t}\}$ in $\mathcal{H}^{S}$, we define $\Phi \colon \overline{X} \to \mathbf{R}$ by 
\begin{align*}
\Phi(x, w) \coloneqq \varphi^{-\log|w|^{2}}(x). 
\end{align*}
\begin{theorem}[\cite{Sem92}]
We denote the $d$ and $d^{c}$-operator on $\overline{X}$ by $\overline{d}$ and $\overline{d}^{c}$, respectively. 
Then 
\begin{align*}
(\pi^{*}\omega_{0}+\overline{d}\,\overline{d}^{c}\Phi)^{n+1} = (n+1)\left(\ddot{\varphi}^{t}-\frac{1}{2}|d\dot{\varphi}^{t}|_{g_{\varphi^{t}}}^{2}\right)\pi^{*}\omega_{\varphi^{t}}^{n}+ \frac{\sqrt{-1}}{2\pi}\frac{dw \wedge d\overline{w}}{|w|^{2}}. 
\end{align*}
In particular, $\{\varphi^{t}\}$ is a geodesic if and only if $\Phi$ satisfies 
\begin{align}\label{geod_MA}
(\pi^{*}\omega_{0}+\overline{d}\,\overline{d}^{c}\Phi)^{n+1} = 0. 
\end{align}
\end{theorem}
Unfortunately, the Dirichlet problem of the equation \eqref{geod_MA} does not usually have a smooth solution (see \cite{LV13}). 
This means that in general one cannot find a smooth geodesic segment connecting two points in $\mathcal{H}_{0}^{S}$. 
On the other hand, Chen \cite{Ch00a} proved that the Dirichlet problem of \eqref{geod_MA} always has a unique solution in a weak sense (\cite{BT76}). 
After successive refinement in \cite{Bl12} and \cite{CTW17}, it is proved that the weak solution has $C^{1,1}$ regularity. 

\subsection{Energy functionals} 
In this subsection, we recall functionals on the space of K\"ahler potentials. 
\begin{definition}
\begin{enumerate}[(1)]
\item We define the Monge-Amp\`ere energy of $\varphi \in \mathcal{H}^{S}$ by 
\begin{align*}
E(\varphi) = \frac{1}{n+1}\sum_{i=0}^{n}\dashint_{X}\varphi \omega_{\varphi}^{i} \wedge \omega_{0}^{n-i}. 
\end{align*}
\item We define the $J$-functional of $\varphi \in \mathcal{H}^{S}$ by 
\begin{align*}
J(\varphi) = \dashint_{X}\varphi \omega_{0}^{n} -E(\varphi). 
\end{align*}
\end{enumerate}
\end{definition}
\begin{fact}
\begin{enumerate}[(1)]
\item $E(\varphi + c) = E(\varphi) + c$ for any $\varphi \in \mathcal{H}^{S}$ and $c \in \mathbf{R}$. 
\item $J(\varphi) \geq 0$ and $J(\varphi + c) = J(\varphi)$ for any $\varphi \in \mathcal{H}^{S}$ and $c \in \mathbf{R}$. 
\item For any smooth path $\{\varphi^{t}\}$ in $\mathcal{H}^{S}$ we have 
\begin{align*}
\frac{d}{dt}E(\varphi^{t}) &= \dashint_{X}\dot{\varphi}^{t}\omega_{\varphi^{t}}^{n}, \\
\frac{d}{dt}J(\varphi^{t}) &= \dashint_{X}\dot{\varphi}^{t}(\omega_{0}^{n} - \omega_{\varphi^{t}}^{n}). 
\end{align*}
Here we mean $\dot{\varphi}^{t}$ the derivative of $\varphi^{t}$ with respective to $t$. 
\end{enumerate}
\end{fact}
By using $E$, we define 
\begin{align*}
\mathcal{H}_{0}^{S} = \mathcal{H}(X, \omega)_{0}^{S} \coloneqq \{\varphi \in \mathcal{H}^{S} \mid E(\varphi) = 0\}. 
\end{align*}
Then we have the identification $\mathcal{H}^{S} \ni \varphi \mapsto (\varphi - E(\varphi), E(\varphi)) \in \mathcal{H}_{0}^{S} \mathbf{R}$. 

\begin{definition}
We define the functionals on $\mathcal{H}_{0}^{S}$ as follows: 
\begin{enumerate}[(1)]
\item We define th Ricci energy of $\varphi \in \mathcal{H}^{S}$ by 
\begin{align*}
R(\varphi) = -\sum_{i=0}^{n-1}\dashint_{X}\varphi \mathrm{Ric}(\omega_{0}) \wedge \omega_{\varphi}^{i} \wedge \omega_{0}^{n-i-1}. 
\end{align*}
\item We define the entropy of $\varphi \in \mathcal{H}^{S}$ by 
\begin{align*}
H(\varphi) = \dashint_{X}\log\left(\frac{\omega_{\varphi}^{n}}{\omega_{0}^{n}}\right)\omega_{\varphi}^{n}. 
\end{align*}
\item We define the K-energy of $\varphi \in \mathcal{H}^{S}$ by 
\begin{align*}
M(\varphi) = -\int_{0}^{1}dt\dashint_{X}\dot{\varphi}^{t}(s(\omega_{\varphi^{t}})-\overline{s})\omega_{\varphi^{t}}^{n},  
\end{align*}
where $\{\varphi^{t}\}_{t \in [0, 1]}$ is a smooth path in $\mathcal{H}^{S}$ so that $\varphi^{0} = 0$ and $\varphi^{1} = \varphi$. 
\end{enumerate}
\end{definition}

\begin{fact}
\begin{enumerate}[(1)]
\item For any smooth path $\{\varphi^{t}\}$ in $\mathcal{H}^{S}$ we have 
\begin{align*}
\frac{d}{dt}R(\varphi^{t}) &= -n\dashint_{X}\dot{\varphi}^{t}\mathrm{Ric}(\varphi_{0}) \wedge \omega_{\varphi^{t}}^{n-1}. 
\end{align*}
\item $M(\varphi) = H(\varphi) + R(\varphi) + \overline{s}E(\varphi)$ for any $\varphi \in \mathcal{H}_{0}^{S}$, and in particular the definition of $M(\varphi)$ is independent of choice of the path $\{\varphi^{t}\}_{t \in [0, 1]}$ in $\varphi \in \mathcal{H}_{0}^{S}$ (\cite{Ch00b}, \cite[Proposition 3.2]{BB17}). 
\item For any smooth path $\{\varphi^{t}\}$ in $\mathcal{H}^{S}$ we have 
\begin{align*}
\frac{d}{dt}M(\varphi^{t}) &= -\dashint_{X}\dot{\varphi}^{t}(s(\omega_{\varphi^{t}})-\overline{s})\omega_{\varphi^{t}}^{n}. 
\end{align*}
In particlar, $\varphi \in \mathcal{H}_{0}^{S}$ is a critical point of $M$ if and only if the scalar curvature of $\omega_{\varphi}$ is constant. 
\item For each $\varphi \in \mathcal{H}_{0}^{S}$ and $W \in \mathfrak{s}$, put  
\begin{align*}
\varphi^{t} = \exp\left(-\frac{t}{4\pi}JW\right)^{*}\varphi. 
\end{align*}
Then $\dot{\varphi}^{t} = u_{\omega_{\varphi^{t}}}^{W}$ and 
\begin{align*}
\frac{d}{dt}M(\varphi^{t}) = \dashint_{X}u_{\omega_{\varphi^{t}}}^{W}(s(\omega_{\varphi^{t}})-\overline{s})\omega_{\varphi^{t}}^{n} = F(W). 
\end{align*}
In other word, the K-energy functional is the integration of the Futaki invariant. 
\item For a smooth geodesic $\{\varphi^{t}\}$ in $\mathcal{H}_{0}^{S}$ we have 
\begin{align*}
\frac{d^{2}}{dt^{2}}M(\varphi^{t}) = \dashint_{X}|\overline{\partial}\mathrm{grad}_{g_{\varphi^{t}}}\dot{\varphi^{t}}|_{g_{\varphi^{t}}}^{2}\omega_{\varphi^{t}}^{n} \geq 0. 
\end{align*}
Hence $M$ is convex along the geodesic $\{\varphi^{t}\}$. 
\end{enumerate}
\end{fact}

\subsection{Vector field energies and the relative K-energy}
In this subsection, we introduce the vector field energy functionals defined by Mabuchi \cite{M01} and the relative K-energy functional defined by \cite{M01}, \cite{CG00}, \cite{Gua00}, and \cite{Sim00} independently. 

Let $(X, L)$ be an $n$-dimensional polarized manifold. 
Fix $\omega_{0} \in \mathcal{H}(X, L)^{S}$, and define 
\begin{align*}
H_{V}(\varphi) = \int_{0}^{1}dt\dashint_{X}\dot{\varphi}^{t}\theta_{\varphi^{t}}\omega_{\varphi^{t}}^{n} 
\end{align*}
for any $\varphi \in \mathcal{H}^{S}$, where $\{\varphi^{t}\}_{t \in [0, 1]}$ is a smooth path in $\mathcal{H}^{S}$ so that $\varphi^{0} = 0$ and $\varphi^{1} = \varphi$. 
\begin{fact}[\cite{M01}]
\begin{enumerate}[(1)]
\item For any $\varphi \in \mathcal{H}_{0}^{S}$ 
\begin{align*}
H_{V}(\varphi) 
&= \frac{1}{n+1}\sum_{i=0}^{n}\dashint_{X}\varphi\theta_{\omega_{\varphi}}\omega_{\varphi}^{i} \wedge \omega_{0}^{n-i} \\
&\quad - \frac{1}{(n+1)(n+2)}\sum_{i=0}^{n}(n-i+1)\dashint_{X}\varphi V\varphi \omega_{\varphi}^{i}\wedge \omega_{0}^{n-i}. 
\end{align*}
In particular, the definition of $H_{V}(\varphi)$ is independent of choice of the path $\{\varphi^{t}\}_{t \in [0, 1]}$ in $\varphi \in \mathcal{H}_{0}^{S}$. 
\item For any smooth path $\{\varphi^{t}\}$ in $\mathcal{H}_{0}^{S}$ we have 
\begin{align*}
\frac{d}{dt}H_{V}(\varphi^{t}) = \dashint_{X}\dot{\varphi}^{t}\theta_{\varphi^{t}}\omega_{\varphi^{t}}^{n}. 
\end{align*}
\item For each $\varphi \in \mathcal{H}_{0}^{S}$ and $W \in \mathfrak{s}$, put $\varphi^{t} = \exp(-tJW/4\pi)^{*}\varphi$. 
Then $\dot{\varphi}^{t} = u_{\omega_{\varphi^{t}}}^{W}$ and 
\begin{align*}
\frac{d}{dt}H_{V}(\varphi^{t}) = \dashint_{X}u_{\omega_{\varphi^{t}}}^{W}\theta_{\omega_{\varphi}}\omega_{\varphi^{t}}^{n} = B(W, V). 
\end{align*}
\item For a smooth geodesic $\{\varphi^{t}\}$ in $\mathcal{H}_{0}^{S}$, we have 
\begin{align*}
\frac{d^{2}}{dt^{2}}H_{V}(\varphi^{t}) = \dashint_{X}\theta_{\omega_{\varphi^{t}}}\left(\ddot{\varphi}^{t}- \frac{1}{2}|d\dot{\varphi}^{t}|_{g_{\varphi^{t}}}^{2}\right) \omega_{\varphi^{t}}^{n} = 0. 
\end{align*}
Hence $H_{V}$ is affine along the geodesic $\{\varphi^{t}\}$. 
\end{enumerate}
\end{fact}

\begin{definition}[\cite{M01}, \cite{CG00}, \cite{Gua00}, \cite{Sim00}]
We define the relative K-energy functional $M_{V} \colon \mathcal{H}^{S} \to \mathbf{R}$ by 
\begin{align*}
M_{V}(\varphi) \coloneqq M(\varphi) + H_{V}(\varphi) = -\int_{0}^{1}dt\dashint_{X}\dot{\varphi}^{t}(s(\omega_{\varphi^{t}})-\overline{s}-\theta_{\varphi^{t}})\omega_{\varphi^{t}}^{n}. 
\end{align*}
\end{definition}

\begin{fact}
\begin{enumerate}[(1)]
\item For any smooth path $\{\varphi^{t}\}$ in $\mathcal{H}_{0}^{S}$ we have 
\begin{align*}
\frac{d}{dt}M_{V}(\varphi^{t}) = -\dashint_{X}\dot{\varphi}^{t}(s(\omega_{\varphi^{t}})-\overline{s}-\theta_{\varphi^{t}})\omega_{\varphi^{t}}^{n}. 
\end{align*}
Hence $\omega_{\varphi}$ is an extremal K\"ahler metric if and only if $\varphi \in \mathcal{H}_{0}^{S}$ is a critical point of $M_{V}$. 
\item For each $\varphi \in \mathcal{H}_{0}^{S}$ and $W \in \mathfrak{s}$, put $\varphi^{t} = \exp(-tJW/4\pi)^{*}\varphi$. 
Then $\dot{\varphi}^{t} = u_{\omega_{\varphi^{t}}}^{W}$ and 
\begin{align*}
\frac{d}{dt}M_{V}(\varphi^{t}) = F(W) + B(W, V) = 0. 
\end{align*}
\item For a smooth geodesic $\{\varphi^{t}\}$ in $\mathcal{H}_{0}^{S}$ we have 
\begin{align*}
\frac{d^{2}}{dt^{2}}M_{V}(\varphi^{t}) = \frac{d^{2}}{dt^{2}}M(\varphi^{t}) \geq 0. 
\end{align*}
Hence $M_{V}$ is also convex along the geodesic $\{\varphi^{t}\}$. 
\end{enumerate}
\end{fact}

\subsection{$d_{1}$-distances, reduced $d_{1}$-distances and the reduced J-functionals}
Let $(X, L)$ be an $n$-dimensional polarized manifold. 
For any smooth path $\{\varphi^{t}\}_{t \in [0, 1]} \in \mathcal{H}_{0}^{S}$, we define the $L^{1}$-length of $\{\varphi^{t}\}_{t \in [0, 1]}$ by 
\begin{align*}
\ell_{1}(\{\varphi^{t}\}_{t \in [0, 1]}) \coloneqq \int_{0}^{1}dt\dashint_{X}|\dot{\varphi^{t}}|\omega_{\varphi^{t}}^{n}.  
\end{align*}
\begin{definition}
For each $\varphi_{0}, \varphi_{1} \in \mathcal{H}_{0}^{S}$, we define the $L^{1}$-distance of $\varphi_{0}$ and $\varphi_{1}$ by 
\begin{align*}
d_{1}(\varphi_{0}, \varphi_{1}) \coloneqq \inf_{\{\varphi^{t}\}_{t \in [0, 1]}}\ell_{1}(\{\varphi^{t}\}_{t \in [0, 1]}), 
\end{align*}
where $\{\varphi^{t}\}_{t \in [0, 1]}$ runs through all smooth paths in $\mathcal{H}_{0}^{S}$ so that $\varphi^{0} = \varphi_{0}$ and $\varphi^{1} = \varphi_{1}$. 
\end{definition}

\begin{fact}[\cite{Dar15}]
The $L^{1}$-distance $d_{1}$ is in fact a distance function on $\mathcal{H}_{0}^{S}$. 
However, the metric space $(\mathcal{H}_{0}^{S}, d_{1})$ is not complete. 
\end{fact}

Next we consider the reduced $L^{1}$-distance. 
Let us consider the action of $T$ on $\mathcal{H}_{0}^{S}$. 
For each $\tau \in T$ and $\varphi \in \mathcal{H}_{0}^{S}$, define $\varphi_{\tau} \in \mathcal{H}_{0}^{S}$ by 
\begin{align*}
\tau^{*}\omega_{\varphi} = \omega_{0} + dd^{c}\varphi_{\tau} = \omega_{\varphi_{\tau}}. 
\end{align*}
Since 
\begin{align*}
\tau^{*}\omega_{\varphi} = \tau^{*}(\omega_{0} + dd^{c}\varphi) = \omega_{0} + dd^{c}(0_{\tau} + \tau^{*}\varphi)
\end{align*}
and 
\begin{align*}
E(0_{\tau} + \tau^{*}\varphi) 
&= E(0_{\tau} + \tau^{*}\varphi) - E(0_{\tau}) \\
&= \frac{1}{n+1}\sum_{i=0}^{n}\dashint_{X}(\tau^{*}\varphi)(\tau^{*}\omega_{\varphi})^{i} \wedge (\tau^{*}\omega_{0})^{n-i} \\
&= E(\varphi) = 0, 
\end{align*}
we have $0_{\tau} + \tau^{*}\varphi \in \mathcal{H}_{0}^{S}$ and $\varphi_{\tau} = 0_{\tau} + \tau^{*}\varphi$. 

\begin{proposition}\label{T-action_isom}
The $T$-action on $(\mathcal{H}_{0}^{S}, d_{1})$ is isometric, that is, for each $\tau \in T$ and $\varphi_{0}, \varphi_{1} \in \mathcal{H}_{0}^{S}$, we have
\begin{align*}
d_{1}((\varphi_{0})_{\tau}, (\varphi_{1})_{\tau}) = d_{1}(\varphi_{0}, \varphi_{1}). 
\end{align*}
\end{proposition}
\begin{proof}
Let $\{\varphi^{t}\}_{t \in [0, 1]}$ be a smooth path in $\mathcal{H}_{0}^{S}$ connecting $\varphi_{0}$ and $\varphi_{1}$. 
Then $\{\varphi_{\tau}^{t}\}_{t \in [0, 1]}$ is a smooth path in $\mathcal{H}_{0}^{S}$ connecting $(\varphi_{0})_{\tau}$ and $(\varphi_{1})_{\tau}$, and 
\begin{align*}
\dot{\varphi}_{\tau}^{t} = \frac{d}{dt}(0_{\tau} + \tau^{*}\varphi^{t}) = \tau^{*}\dot{\varphi}^{t}. 
\end{align*}
Hence we have
\begin{align*}
\ell_{1}(\{\varphi_{\tau}^{t}\}_{t \in [0, 1]}) 
&= \int_{0}^{1}dt\dashint_{X}|\dot{\varphi}_{\tau}^{t}|\omega_{\varphi_{\tau}^{t}}^{n} \\
&= \int_{0}^{1}dt\dashint_{X}|\tau^{*}\dot{\varphi}^{t}|\tau^{*}\omega_{\varphi^{t}}^{n} 
= \ell_{1}(\{\varphi^{t}\}_{t \in [0, 1]}),  
\end{align*}
as required. 
\end{proof}

\begin{definition}
For each $\varphi_{0}, \varphi_{1} \in \mathcal{H}_{0}^{S}$, we define the reduced $L^{1}$-distance of $\varphi_{0}$ and $\varphi_{1}$ by 
\begin{align*}
d_{1, T}(\varphi_{0}, \varphi_{1}) 
\coloneqq \inf_{\tau_{0}, \tau_{1} \in T}d_{1}((\varphi_{0})_{\tau_{0}}, (\varphi_{1})_{\tau_{1}}) = \inf_{\tau \in T}d_{1}((\varphi_{0})_{\tau}, \varphi_{1}). 
\end{align*}
\end{definition}

Similarly, we introduce the reduced $J$-functional. 
\begin{definition}
We define the reduced $J$-functional by 
\begin{align*}
J_{T}(\varphi) \coloneqq \inf_{\tau \in T}J(\varphi_{\tau}), 
\end{align*}
for any $\varphi \in \mathcal{H}_{0}^{S}$. 
\end{definition}

\begin{proposition}\label{equiv_L1_J}
The $J$-functional is equivalent to the $L^{1}$-distance in the following sense: there exists a $C > 0$ such that for any $\varphi \in \mathcal{H}_{0}^{S}$ 
\begin{align*}
\frac{1}{C}J(\varphi) - C \leq d_{1}(\varphi, 0) \leq CJ(\varphi) + C. 
\end{align*}
\end{proposition}

\begin{proposition}\label{equiv_redL1_redJ}
The reduced $J$-functional is equivalent to the reduced $L^{1}$-distance in the following sense: there exists a $C > 0$ such that for any $\varphi \in \mathcal{H}_{0}^{S}$ 
\begin{align*}
\frac{1}{C}J_{T}(\varphi) - C \leq d_{1, T}(\varphi, 0) \leq CJ_{T}(\varphi) + C. 
\end{align*}
\end{proposition}

By using the theory of Chen-Cheng and He \cite{CC17}, \cite{CC18a}, \cite{CC18b}, \cite{He18}, and the arguments in \cite{Li19}, we shall show the following theorem. 
\begin{theorem}\label{ext_coer_equiv}
For a polarized manifold $(X, L)$, the followings are equivalent. 
\begin{enumerate}[(1)]
\item $(X, L)$ admits an $S$-invariant extremal K\"ahler metric. 
\item There exists $\delta, C >0$ such that 
\begin{align*}
M_{V}(\varphi) \geq \delta d_{1, T}(\varphi, 0) - C 
\end{align*}
for any $\varphi \in \mathcal{H}_{0}^{S}$. 
\item There exists $\delta, C >0$ such that 
\begin{align*}
M_{V}(\varphi) \geq \delta J_{T}(\varphi) - C 
\end{align*}
for any $\varphi \in \mathcal{H}_{0}^{S}$. 
\end{enumerate}
\end{theorem}
We say that the relative K-energy is $T$-coercive if $M_{V}$ satisfies the condition (2) and (or) (3) in Theorem \ref{ext_coer_equiv}. 
\begin{proof}[Proof of Theorem \ref{ext_coer_equiv}]
The equivalence of (2) and (3) is ovbious from Proposition \ref{equiv_redL1_redJ}. 
We shall show the equivalence of (1) and (2). 
For this, we apply Darvas-Rubinstein's existence/properness principle \cite[Theorem 3.4]{DR17}. 
In their notation, we consider the following data 
\begin{align*}
\mathcal{R} = \mathcal{H}_{0}^{S},\quad d=d_{1},\quad F = M_{V},\quad G = T. 
\end{align*}
Note that the lower semicontinuity of $M_{V}$ is obtained in \cite[Corollary 2.2]{He18}. 
For (P1), existence of $d_{1}$-geodesic can be found in \cite[Theorem 3]{Ch00a} and \cite[Theorem 2]{Dar15}. 
Continuity and convexity of $M_{V}$ can be found in \cite[Theorem 3.4]{BB17} and \cite[Proposition 2.2]{He18}. 
The property (P2) is obtained by the compactness theorem in \cite[Theorem 2.17]{BBEGZ19} and the lower semicontinuity of $M_{V}$. 
The property (P3) is due to \cite[Theorem 3.6]{He18}. 
The property (P4) is already showed in Proposition \ref{T-action_isom}. 
The property (P6) can be omitted (see Remark 4.8 of \cite{Dar19}). 
The property (P7) is easily proved. 

We need to check the property (P5). 
Let $\mathfrak{aut}(X, V) = \{W \in \mathfrak{aut}(X) \mid L_{V}W = 0\}$. 
Then $\mathfrak{aut}(X, V)$ is a Lie subalgebra of $\mathfrak{aut}(X)$. 
We denote by $H \coloneqq \mathrm{Aut}(X, V)_{0}$ the connected Lie subgroup of $\mathrm{Aut}(X)_{0}$ corresponding to $\mathfrak{aut}(X, V)$. 
Let $\omega_{i}$, $i=1, 2$ be two $S$-invariant extremal K\"ahler metrics in $\mathcal{H}(X, L)^{S}$, and 
\begin{align*}
K_{i} = \mathrm{Isom}(X, \omega_{i})_{0}. 
\end{align*}
By a theorem of Calabi \cite{Ca85}, $\mathrm{Aut}(X, V)_{0}$ is a reductive and each $K_{i}$ is a maximal compact subgroup of $\mathrm{Aut}(X, V)_{0}$. 
Since $\omega_{i}$ is $S$-invariant, we have $S \subset K_{1} \cap K_{2}$ and that $S$ is a maximal compact torus of both $K_{1}$ and $K_{2}$. 
Hence by \cite[Proposition A.2]{Li19} $K_{2} = \tau K_{1} \tau^{-1}$ for some $\tau \in T$. 

By Berman-Berndttson's uniqueness theorem of extremal K\"ahler metrics, we can find $f \in H$ satisfying $\omega_{1} = f^{*}\omega_{2}$ (see the proof of \cite[Theorem 4.16]{BB17}). 
Then we have $K_{1} = f K_{2} f^{-1} = (f\tau)K_{1}(f\tau)^{-1}$ and hence $f\tau \in N_{H}(K_{1})$, where $N_{H}(K_{1})$ is the normalizer of $K_{1}$ in $H$. 
Furthermore, by \cite[Proposition A.1]{Li19} we have $f\tau \in C(H)_{0}K_{1}$, where $C(H)_{0}$ is the identity component of the center of $H$. 
Hence there exists $\tau' \in C(H)_{0} \subset T$ and $k_{1} \in K_{1}$ such that $f\tau = \tau'k_{1}$. 
By setting $\tau_{1} = \tau'\tau^{-1}$, we have $\tau_{1} \in T$, $f = k_{1}\tau_{1}$ and  
\begin{align*}
\omega_{1} = f^{*}\omega_{2} = (k_{1}\tau_{1})^{*}\omega_{2} = \tau_{1}^{*}k_{1}^{*}\omega_{2} = \tau_{1}^{*}\omega_{2}. 
\end{align*}
\end{proof}

\section{Test configurations and relative K-stability}
\subsection{Test configurations}

Let $(X, L)$ be an $n$-dimensional polarized manifold. 

\begin{definition}
A test configuration for $(X, L)$ of exponent $r$ consists of 
\begin{enumerate}
\item a scheme $\mathcal{X}$ with a $\mathbf{G}_{m}$-action, 
\item a $\mathbf{G}_{m}$-equivariant flat and proper morphism of schemes $\pi \colon \mathcal{X} \to \mathbf{A}^{1}$, where $\mathbf{G}_{m}$ acts on $\mathbf{A}^{1}$ by multiplication, 
\item a $\mathbf{G}_{m}$-linearlized $\pi$-very ample line bundle $\mathcal{L}$ over $\mathcal{X}$, and
\item an isomorphism $(\mathcal{X}_{1}, \mathcal{L}_{1}) \cong (X, L^{r})$. 
\end{enumerate}
Here $\mathcal{X}_{t}$ is the fiber of $\pi$ over $t \in \mathbf{A}^{1}$, and $\mathcal{L}_{t} = \mathcal{L}|_{\mathcal{X}_{t}}$. 
A test configuration $(\mathcal{X}, \mathcal{L})$ is called product if $\mathcal{X}$ is isomorphic to the fiber product of $X$ and $\mathbf{A}^{1}$ over $\mathbf{A}^{1}$, and trivial if in addition $\mathbf{G}_{m}$ acts only on the second factor. 
A test configuration $(\mathcal{X}, \mathcal{L})$ is called normal if $\mathcal{X}$ is a normal variety. 
For an algebraic subgroup $G$ of $\mathrm{Aut}(X, L)$, a $G$-equivariant test configuration is a test configuration $(\mathcal{X}, \mathcal{L})$ with a lifted $G$-action on $(\mathcal{X}, \mathcal{L})$ which preserves each fiber, commutes with the fiber $\mathbf{G}_{m}$-action of $(\mathcal{X}, \mathcal{L})$, and coincides with $G$-action when actiong on $(\mathcal{X}_{1}, \mathcal{L}_{1}) \cong (X, L^{r})$. 
\end{definition}
For our purpose, it is convenient to compactify test configurations. 
Let $(\mathcal{X}, \mathcal{L})$ be a test configuration for $(X, L)$ of exponent $r$. 
The $\mathbf{G}_{m}$-action on $(\mathcal{X}, \mathcal{L})$ and the isomorphism $(\mathcal{X}_{1}, \mathcal{L}_{1}) \cong (X, L^{r})$ induces an equivariant trivialization $(\mathcal{X}, \mathcal{L})|_{\mathbf{G}_{m}} \cong (X \times \mathbf{G}_{m}, p_{1}^{\ast}L^{r})$, where $p_{1} \colon X \times \mathbf{G}_{m} \to X$ is the projection to the first factor, and $\mathbf{G}_{m}$ acts only on the second factor. 
Then we can compactify $(\mathcal{X}, \mathcal{L})$ by gluing it with the product $(X \times (\mathbf{P}^{1}\setminus\{0\}), p_{1}^{\ast}L^{r})$ along their respective open sets $\mathcal{X} \setminus \mathcal{X}_{0}$ and $X \times (\mathbf{A}^{1} \setminus \{0\})$. 
The resulting $(\overline{\mathcal{X}}, \overline{\mathcal{L}})$ with $\overline{\pi} \colon \overline{\mathcal{X}} \to \mathbf{P}^{1}$ is a $\mathbf{G}_{m}$-equivariant flat family over $\mathbf{P}^{1}$ with fibers isomorphic to $(X, L^{r})$ except over 0. 
We call $(\overline{\mathcal{X}}, \overline{\mathcal{L}})$ \emph{the canonical compactification} of the test configuration $(\mathcal{X}, \mathcal{L})$. 
Note that the action of $\mathbf{G}_{m}$ on the $\infty$-fiber $(\overline{\mathcal{X}}_{\infty}, \overline{\mathcal{L}}_{\infty})$ is trivial. 
\begin{example}
Since $L$ is ample, there exists an embeddings 
\begin{align*}
\Phi_{m} \colon X \to \mathbf{P}(H^{0}(X, L^{m}))^{\vee} 
\end{align*}
for any sufficiently divisible $m \in \mathbf{Z}_{>0}$. 
By replacing $m$ if necessary, we obtain a faithful representation 
\begin{align*}
\theta_{m} \colon T \to \mathrm{GL}(H^{0}(X, L^{m})) 
\end{align*}
satisfying 
\begin{align*}
\theta_{m}(\tau) \circ \Phi_{m} = \Phi_{m}  \circ \tau
\end{align*}
for any $\tau \in T$. 
Then the infinitesimal representaion $\theta_{m*} \colon \mathfrak{t} \to \mathfrak{gl}(H^{0}(X, L^{m}))$ is given by 
\begin{align*}
-\frac{1}{4\pi}\theta_{m*}(JW) = \nabla_{W} + mu_{\omega}^{W}\mathrm{id}_{H^{0}(X, L^{m})}. 
\end{align*}
Let $C_{m}(T)$ be the centralizer of a subset $\theta_{m}(T)$ in $\mathrm{GL}(H^{0}(X, L^{m}))$. 
For any morphism $\lambda \colon \mathbf{G}_{m} \to C_{m}(T)$, 
we define $\mathcal{X}$ as the Zariski closure of the set 
\begin{align*}
\bigcup_{\tau \in \mathbf{G}_{m}}\lambda^{\vee}(\tau)(\Phi_{m}(X)) \times \{\tau\}
\end{align*}
in $\mathbf{P}(H^{0}(X, L^{m}))^{\vee} \times \mathbf{A}^{1}$, and set $\mathcal{L} = \mathcal{O}_{\mathcal{X}}(1)$. 
Then $(\mathcal{X}, \mathcal{L})$ is a test confiduration for $(X, L)$ of exponent $m$. 
The central fiber $\mathcal{X}_{0}$ is obtained as the flat limit of the image of $\Phi_{m}(X)$ under $\lambda^{\vee}$ as $\tau \to 0$. 
\end{example}

\begin{example}
The set of test configurations for $(X, L)$ admits two natural operations: translations and scalings. 
$(\mathcal{X}, \mathcal{L})$ be a test configuration for $(X, L)$. 
\begin{enumerate}
\item Let $c \in \mathbf{Q}$. 
A translation of $(\mathcal{X}, \mathcal{L})$ is defined to be $(\mathcal{X}, \mathcal{L} + c\mathcal{X}_{0}) \coloneqq (\mathcal{X}, \mathcal{L} \otimes_{\mathcal{X}}\mathcal{O}_{\mathcal{X}}(c\mathcal{X}_{0}))$. 
\item Let $d \in \mathbf{Z}_{>0}$. 
A scaling of $(\mathcal{X}, \mathcal{L})$ is defined to be the normalization of the base change of $(\mathcal{X}, \mathcal{L})$ by $\tau \mapsto \tau^{d}$. 
We denote such a test configuration by $(\mathcal{X}_{d}, \mathcal{L}_{d})$. 
\end{enumerate}
\end{example}

Associated to a test configuration, the Donaldson-Futaki invariant is defined as follows. 
Let $(\mathcal{X}, \mathcal{L})$ be a test configuration for $(X, L)$ of exponent $r$. 
By the flatness of $\pi \colon \mathcal{X} \to \mathbf{A}^{1}$ we have 
\begin{align*}
N_{rm} \coloneqq \dim_{\mathbf{C}}H^{0}(\mathcal{X}_{0}, \mathcal{L}_{0}^{m}) = \dim_{\mathbf{C}}H^{0}(X, L^{rm}) 
\end{align*}
for sufficiently large $m \in \mathbf{Z}_{>0}$. 
Moreover, the asymptotic Riemann-Roch theorem tells us that 
\begin{align*}
N_{m} = a_{0}m^{n} + a_{1}m^{n-1}+O(m^{n-2}), 
\end{align*}
where 
\begin{align*}
a_{0} = \frac{(L^{n})}{n!},\quad a_{1} = -\frac{1}{2}\frac{(K_{X} \cdot L^{n-1})}{(n-1)!}. 
\end{align*}
Note that 
\begin{align*}
\overline{s} = -n\frac{(K_{X}\cdot L^{n-1})}{(L^{n})} = -2\frac{a_{1}}{a_{0}}. 
\end{align*}
Since $0 \in \mathbf{A}^{1}$ is fixed by the $\mathbf{G}_{m}$-action, the $\mathbf{G}_{m}$-action on $(\mathcal{X}, \mathcal{L})$ induces a coaction $\lambda_{m} \colon H^{0}(\mathcal{X}_{0}, \mathcal{L}_{0}^{m}) \to H^{0}(\mathcal{X}_{0}, \mathcal{L}_{0}^{m}) \otimes \mathbf{C}[t, t^{-1}]$. 
Let 
\begin{align*}
H^{0}(\mathcal{X}_{0}, \mathcal{L}_{0}^{m}) = \bigoplus_{i=1}^{N_{rm}}H^{0}(\mathcal{X}_{0}, \mathcal{L}_{0}^{m})_{\lambda_{i}^{(rm)}} 
\end{align*}
be the weight decomposition. 
Here $H^{0}(\mathcal{X}_{0}, \mathcal{L}_{0}^{m})_{\lambda} \coloneqq \{s \in H^{0}(\mathcal{X}_{0}, \mathcal{L}_{0}^{m}) \mid \lambda_{m}(s) = s \otimes \tau^{\lambda}\}$ is the $\lambda$-weight space for this action. 
By \cite[Theorem 3.1]{BHJ17}, we have 
\begin{align*}
w_{m} \coloneqq \sum_{i=1}^{N_{m}}\lambda_{i}^{(m)} = b_{0}m^{n+1}+ b_{1}m^{n} + O(m^{n-1}) 
\end{align*}
for sufficiently large $m$ divisible by $r$. 
We then consider the coeffients 
\begin{align*}
\frac{w_{m}}{mN_{m}} = F_{0}(\mathcal{X}, \mathcal{L}) + F_{1}(\mathcal{X}, \mathcal{L})m^{-1} + O(m^{-2}), 
\end{align*}
\begin{align*}
F_{0}(\mathcal{X}, \mathcal{L}) = \frac{b_{0}}{a_{0}},\quad F_{1}(\mathcal{X}, \mathcal{L}) = \frac{a_{0}b_{1}-a_{1}b_{0}}{a_{0}^{2}}. 
\end{align*}
\begin{definition}
We call 
\begin{align*}
DF(\mathcal{X}, \mathcal{L}) \coloneqq -2F_{1}(\mathcal{X}, \mathcal{L}) = -2\frac{a_{0}b_{1}-a_{1}b_{0}}{a_{0}^{2}} 
\end{align*}
the Donaldson-Futaki invariant of a test configuration $(\mathcal{X}, \mathcal{L})$. 
\end{definition}

When the test configuration $(\mathcal{X}, \mathcal{L})$ is normal, we can describe the Donaldson-Futaki invariant via intersection numbers (\cite[Proposition 17]{W12}, \cite[Proposition 6]{LX14}, \cite[Proposition 3.12]{BHJ17}). 
Indeed, by using the canonical compactification and asymptotic Riemann-Roch theorem, we have 
\begin{align*}
w_{rm} 
&= \chi(\overline{\mathcal{X}}, \overline{\mathcal{L}}^{m}) - \chi(X, L^{rm}) \\
&= \frac{(\overline{\mathcal{L}}^{n+1})}{r^{n+1}(n+1)!}(rm)^{n+1} - \frac{(K_{\overline{\mathcal{X}}} \cdot \overline{\mathcal{L}}^{n})}{2r^{n}n!}(rm)^{n} + O(m^{n-1}) \\
&\quad - \frac{(L^{n})}{n!}(rm)^{n} + O(m^{n-1}) \\
&= \frac{(\overline{\mathcal{L}}^{n+1})}{r^{n+1}(n+1)!}(rm)^{n+1} - \frac{1}{2r^{n}n!}\left((K_{\overline{\mathcal{X}}} \cdot \overline{\mathcal{L}}^{n}) + 2r^{n}(L^{n})\right)(rm)^{n} \\
&\quad + O(m^{n-1})  
\end{align*}
for any sufficiently large $m \in \mathbf{Z}_{>0}$ (see the proof of \cite[Proposition 3.12]{BHJ17}). 
By setting $K_{\mathcal{X}/\mathbf{P}^{1}} \coloneqq K_{\overline{\mathcal{X}}} - \overline{\pi}^{*}K_{\mathbf{P}^{1}}$, we obtain 
\begin{align*}
2r^{n}(L^{n}) 
&= 2((L^{r})^{n}) 
= 2(\overline{\mathcal{L}}|_{\mathcal{\overline{X}_{1}}})^{n} 
= 2(\overline{\mathcal{X}}_{1} \cdot \overline{\mathcal{L}}_{1}^{n}) \\
&= (\overline{\pi}^{*}(2[1]) \cdot \overline{\mathcal{L}}_{1}^{n}) 
= -(\overline{\pi}^{*}K_{\mathbf{P}^{1}} \cdot \overline{\mathcal{L}}^{n})
\end{align*} 
and
\begin{align*}
b_{0} = \frac{(\overline{\mathcal{L}}^{n+1})}{r^{n+1}(n+1)!},\quad b_{1} = - \frac{1}{2r^{n}n!}(K_{\overline{\mathcal{X}}/\mathbf{P}^{1}} \cdot \overline{\mathcal{L}}^{n}). 
\end{align*}
Therefore, we obtain the intersection number formula 
\begin{align}\label{DF_intersection}
DF(\mathcal{X}, \mathcal{L}) = \overline{s}\frac{(\overline{\mathcal{L}}^{n+1})}{r(n+1)(L^{n})} + \frac{1}{(L^{n})}(K_{\overline{\mathcal{X}}/\mathbf{P}^{1}} \cdot \overline{\mathcal{L}}^{n}). 
\end{align}

\subsection{Duistermaat-Heckman measures}
Let $(\mathcal{X}, \mathcal{L})$ be a test configuration for $(X, L)$ of exponent $r$. 
Then the central fiber $(\mathcal{X}_{0}, \mathcal{L}_{0})$ is a polarized $\mathbf{G}_{m}$-scheme. 
For sufficiently divisible $m \in \mathbf{Z}_{>0}$, we can consider the induced $\mathbf{G}_{m}$-action on $H^{0}(\mathcal{X}_{0}, \mathcal{L}_{0}^{m})$. 
Then there is a weak limit of the measure 
\begin{align*}
\mathrm{DH}(\mathcal{X}, \mathcal{L}) = \lim_{m \to \infty}\frac{1}{N_{rm}}\sum_{\lambda \in \mathbf{Z}}\dim H^{0}(\mathcal{X}_{0}, \mathcal{L}_{0}^{m})_{\lambda}\delta_{\lambda/rm}
\end{align*}
on $\mathbf{R}$, where $\delta_{\lambda/rm}$ is the Dirac delta measure concentrated at $\lambda/rm$ (see \cite{BHJ17}). 
This is called \emph{the Duistermaat-Heckman measure} of the test configuratin $(\mathcal{X}, \mathcal{L})$. 

\subsection{Non-Archimedean functionals}
In this subsection we introduce non-Archimedean functionals. 
Let $(X, L)$ be an $n$-dimensional polarized manifold and $(\mathcal{X}, \mathcal{L})$ a test configuration for $(X, L)$ of exponent $r$. 
\begin{definition}[\cite{BHJ17}]
\begin{enumerate}[(1)]
\item We define the non-Archimedean Monge-Amp\`ere energy of $\varphi$ by 
\begin{align*}
E^{NA}(\mathcal{X}, \mathcal{L}) \coloneqq \frac{(\overline{\mathcal{L}}^{n+1})}{(n+1)(L^{n})} = \int_{\mathbf{R}}\lambda\,d\mathrm{DH}(\mathcal{X}, \mathcal{L}). 
\end{align*}
\item We define the non-Archimedean $J$-functional of $\varphi$ by 
\begin{align*}
J^{NA}(\mathcal{X}, \mathcal{L}) \coloneqq \sup \mathrm{supp}(\mathrm{DH}(\mathcal{X}, \mathcal{L}))-E^{NA}(\mathcal{X}, \mathcal{L}). 
\end{align*}
\end{enumerate}
\end{definition}
When $\mathcal{X}$ dominates the product $X \times \mathbf{P}^{1}$, namely, there exists $\rho \colon \mathcal{X} \to X \times \mathbf{P}^{1}$ which is equivariant with respect to the trivial action on the target, $J^{NA}$ can be written as 
\begin{align*}
J^{NA}(\mathcal{X}, \mathcal{L}) = \frac{1}{(L^{n})}(\overline{\mathcal{L}}\cdot\rho^{*}p_{1}^{*}L^{n}) - \frac{(\overline{\mathcal{L}}^{n+1})}{r(n+1)(L^{n})}. 
\end{align*}

Let 
\begin{align*}
&K_{\overline{\mathcal{X}}}^{\mathrm{log}} \coloneqq K_{\overline{\mathcal{X}}} + \overline{\mathcal{X}}_{0,\mathrm{red}} + \overline{\mathcal{X}}_{\infty, \mathrm{red}} = K_{\overline{\mathcal{X}}} + \overline{\mathcal{X}}_{0,\mathrm{red}} + \overline{\mathcal{X}}_{\infty}, \\
&K_{\mathbf{P}^{1}}^{\mathrm{log}} \coloneqq K_{\mathbf{P}^{1}} + [0] + [\infty], \\
&K_{\overline{\mathcal{X}}/\mathbf{P}^{1}}^{\mathrm{log}} \coloneqq K_{\overline{\mathcal{X}}}^{\mathrm{log}} - \overline{\pi}^{*}K_{\mathbf{P}^{1}}^{\mathrm{log}} = K_{\overline{\mathcal{X}}/\mathbf{P}^{1}} + (\overline{\mathcal{X}}_{0,\mathrm{red}} - \overline{\mathcal{X}}_{0}). 
\end{align*} 

\begin{definition}[\cite{BHJ17}]
Let $(\mathcal{X}, \mathcal{L})$ be a test configuration for $(X, L)$. 
\begin{enumerate}[(1)]
\item The non-Archimedean Ricci energy is defined by
\begin{align*}
R^{NA}(\mathcal{X}, \mathcal{L}) \coloneqq \frac{1}{(L^{n})}(\rho^{*}K_{X \times \mathbf{P}^{1}/\mathbf{P}^{1}}^{\mathrm{log}} \cdot \overline{\mathcal{L}}^{n}). 
\end{align*}
\item The non-Archimedean entropy is defined by
\begin{align*}
H^{NA}(\mathcal{X}, \mathcal{L}) \coloneqq \frac{1}{(L^{n})}(K_{\overline{\mathcal{X}}/\mathbf{P}^{1}}^{\mathrm{log}} \cdot \overline{\mathcal{L}}^{n})-\frac{1}{(L^{n})}(\rho^{*}K_{X \times \mathbf{P}^{1}/\mathbf{P}^{1}}^{\mathrm{log}} \cdot \overline{\mathcal{L}}^{n}). 
\end{align*}
\item The non-Archimedean K-energy is defined by 
\begin{align*}
M^{NA}(\varphi) 
&\coloneqq H^{NA}(\mathcal{X}, \mathcal{L}) + R^{NA}(\mathcal{X}, \mathcal{L}) + \overline{s}E^{NA}(\mathcal{X}, \mathcal{L}) \\
&= \overline{s}\frac{(\overline{\mathcal{L}}^{n+1})}{(n+1)(L^{n})} + \frac{1}{(L^{n})}(K_{\overline{\mathcal{X}}/\mathbf{P}^{1}}^{\mathrm{log}} \cdot \overline{\mathcal{L}}^{n}). 
\end{align*}
\end{enumerate}
\end{definition}

The non-Archimedean K-energy is closely related to the Donaldson-Futaki invariant. 
In fact, by comparing \eqref{DF_intersection} we have 
\begin{align*}
M^{NA}(\mathcal{X}, \mathcal{L}) 
&= DF(\mathcal{X}, \mathcal{L}) + \frac{1}{(L^{n})}(K_{\overline{\mathcal{X}}/\mathbf{P}^{1}}^{\mathrm{log}} \cdot \overline{\mathcal{L}}^{n}) - \frac{1}{(L^{n})}(K_{\overline{\mathcal{X}}/\mathbf{P}^{1}} \cdot \overline{\mathcal{L}}^{n}) \\
&= DF(\mathcal{X}, \mathcal{L}) + \frac{1}{(L^{n})}((\overline{\mathcal{X}}_{0,\mathrm{red}} - \overline{\mathcal{X}}_{0}) \cdot \overline{\mathcal{L}}^{n}). 
\end{align*}
Hence $M^{NA}(\mathcal{X}, \mathcal{L}) \leq DF(\mathcal{X}, \mathcal{L})$ for any test configuration $(\mathcal{X}, \mathcal{L})$ and equality holds if and only if $\mathcal{X}$ is regular in codimension one and $\mathcal{X}_{0}$ is generically reduced (\cite[Proposition 7.15]{BHJ17}, \cite[Proposition 2.8]{BHJ19}). 
An advantage of $M^{NA}$ is the following homogeneity with respect to base change, which fails for $DF$ when the central fiber is non-reduced. 
\begin{proposition}[{\cite[Proposition 7.14]{BHJ17}}]\label{homogenuity_NA_K-energy}
Let $(\mathcal{X}, \mathcal{L})$ be a test configuration for $(X, L)$. 
For each $d \in \mathbf{Z}_{>0}$, let $(\mathcal{X}_{d}, \mathcal{L}_{d})$ be the test configuration obtained by the base change $z \mapsto z^{d}$. 
Then $M^{NA}(\mathcal{X}_{d}, \mathcal{L}_{d}) = dM^{NA}(\mathcal{X}, \mathcal{L})$. 
\end{proposition}

On the other hand, Mumford's semistable reduction theorem (\cite[p.53, pp.100--101]{KKMS73}. See \cite[Theorem 3.8]{ADVLN12} and \cite[Lemma 5]{LX14}  for equivariant version) gives us the following proposition.   

\begin{proposition}[{\cite[Proposition 7.16]{BHJ17}}]\label{equiv_NA_K_DF}
For each test configuration $(\mathcal{X}, \mathcal{L})$ for $(X, L)$, there exists $d_{0} \in \mathbf{Z}_{>0}$ such that $DF(\mathcal{X}_{d}, \mathcal{L}_{d}) = M^{NA}(\mathcal{X}_{d}, \mathcal{L}_{d}) = dM^{NA}(\mathcal{X}, \mathcal{L})$ for all $d \in \mathbf{Z}_{>0}$ divisible by $d_{0}$. 
\end{proposition}

Now we consider the $T$-action on $(X, L)$. 
Let $(\mathcal{X}, \mathcal{L})$ be a $T$-equivariant test configuration for $(X, L)$. 
Let $V$ be the extremal K\"ahler vector field for $T$, and $T_{V}$ the one-dimensional algebraic torus generated by $V$. 
Then $T_{V} \subset T$ and induces a coaction 
\begin{align*}
\theta_{m} \colon H^{0}(\mathcal{X}_{0}, \mathcal{L}_{0}^{m}) \to H^{0}(\mathcal{X}_{0}, \mathcal{L}_{0}^{m}) \otimes \mathbf{C}[t, t^{-1}] 
\end{align*}
for each $m \in \mathbf{Z}_{> 0}$. 
We denote by $(\theta_{1}^{(rm)}, \ldots, \theta_{N_{rm}}^{(rm)})$ the weight of $\theta_{m}^{\vee}$. 
Let $\Theta_{rm} \coloneqq \theta_{m*}(1)$ be the generator of $\theta_{m}$ and $\Theta_{rm}^{\circ}$ the its traceless part. 

\begin{definition}
Let $(\mathcal{X}, \mathcal{L})$ a $T$-equivariant test configuration for $(X, L)$. 
\begin{enumerate}[(1)]
\item ${\displaystyle H_{V}^{NA}(\mathcal{X}, \mathcal{L}) \coloneqq \lim_{m \to \infty}\frac{1}{(rm)^{2}N_{rm}}\mathrm{tr}(\Lambda_{rm}^{\circ}\Theta_{rm}^{\circ})}$. 
\item We define the non-Archimedean relatice K-energy of $\varphi$ by 
\begin{align*}
M_{V}^{NA}(\mathcal{X}, \mathcal{L}) \coloneqq M^{NA}(\mathcal{X}, \mathcal{L}) + H_{V}^{NA}(\mathcal{X}, \mathcal{L}). 
\end{align*}
\item We define the relative Donaldson-Futaki invariant of $\varphi$ by 
\begin{align*}
DF_{V}(\mathcal{X}, \mathcal{L}) \coloneqq DF(\mathcal{X}, \mathcal{L}) + H_{V}^{NA}(\mathcal{X}, \mathcal{L}). 
\end{align*}
\end{enumerate}
\end{definition}

The following slope formula relates the Archimedean and non-Archimedean functionals. 
Let $(\mathcal{X}, \mathcal{L})$ be a $T$-equivariant test configuration for $(X, L)$ and $\lambda \colon \mathbf{G}_{m} \to \mathrm{Aut}^{0}(\mathcal{X}, \mathcal{L})$ the $\mathbf{G}_{m}$-action. 
We denote by $h_{\mathrm{FS}}$ the pull back of the Fubini-Study metric on $\mathbf{P}H^{0}(\mathcal{X}_{0}, \mathcal{L}_{0}^{m})^{\vee}$ by the composition 
\begin{align*}
\mathcal{X} \hookrightarrow \mathbf{P}H^{0}(\mathcal{X}_{0}, \mathcal{L}_{0}^{m})^{\vee} \times \mathbf{A}^{1} \to \mathbf{P}H^{0}(\mathcal{X}_{0}, \mathcal{L}_{0}^{m})^{\vee}, 
\end{align*}
and $h_{\mathrm{FS}}|_{\tau}$ the restriction of $h_{\mathrm{FS}}$ to $(\mathcal{X}_{\tau}, \mathcal{L}_{\tau})$ for any $\tau \in \mathbf{G}_{m}$. 
By setting 
\begin{align*}
h_{t} \coloneqq \lambda(e^{-t/2})^{*}h_{\mathrm{FS}}|_{e^{-t/2}},\quad \omega_{0} \coloneqq c_{1}(L, h_{0}) 
\end{align*}
for each $t \in [0, \infty)$, we have $\varphi^{t} \in \mathcal{H}(X, \omega_{0})_{0}^{S}$ by 
\begin{align*}
c_{1}(L, h_{t}) - \omega_{0} = dd^{c}\varphi^{t}. 
\end{align*}

\begin{theorem}[{\cite[Theorem 3.6]{BHJ19}, \cite[Theorem 12]{Y19}}]\label{slope_formula}
Let $(\mathcal{X}, \mathcal{L})$ and $\{\varphi^{t}\}_{t \in [0, \infty)}$ be as above, and $F$ be one of the following functionals: $E$, $J$, $M$, $H_{V}$, and $M_{V}$. 
Then 
\begin{align*}
\lim_{t \to \infty}\frac{F(\varphi^{t})}{t} = F^{NA}(\mathcal{X}, \mathcal{L}). 
\end{align*}
\end{theorem}

\subsection{Filtrations}
To introduce the non-Archimedean version of the reduced $J$-functional, it is convenient to use the notion of the \emph{filtration} of the section ring 
\begin{align*}
R = R(X, L) \coloneqq \bigoplus_{m=0}^{\infty}R_{m},\quad R_{m} = H^{0}(X, L^{m})
\end{align*}
of a polarized manifold $(X, L)$. 
Following \cite{BHJ17}, we give a brief review of filtrations. 

\begin{definition}
A filtration $\mathcal{F}$ of the graded $\mathbf{C}$-algebra $R$ consists of a family of subspaces $\{F^{\lambda}R_{m}\}$ of $R_{m}$ for each $m \in \mathbf{Z}_{\geq 0}$ satisfying: 
\begin{enumerate}[(1)]
\item decreasing: $F^{\lambda}R_{m} \subset F^{\lambda'}R_{m}$ if $\lambda \geq \lambda'$, 
\item left-continuous: $F^{\lambda}R_{m} = \bigcap_{\lambda' < \lambda}F^{\lambda'}R_{m}$, 
\item multiplicative: $F^{\lambda}R_{m} \cdot F^{\lambda'}R_{m'} \subset F^{\lambda + \lambda'}R_{m + m'}$, and 
\item linearly bounded: there exist $e_{-}, e_{+} \in \mathbf{Z}$ such that $F^{me_{-}}R_{m} = R_{m}$ and $F^{me_{+}}R_{m} = \{0\}$ for all $m \in \mathbf{Z}_{\geq 0}$. 
\end{enumerate}
We say that $\mathcal{F}$ is a $\mathbf{Z}$-filtration if $F^{\lambda}R_{m} = F^{\lceil \lambda \rceil}R_{m}$ for any $\lambda \in \mathbf{R}$ and $m \in \mathbf{Z}_{\geq 0}$. 
A filtration $FR$ is $T$-equivariant if $F^{\lambda}R_{m}$ is $T$-invariant subspace of $R_{m}$ for all $\lambda \in \mathbf{R}$. 
\end{definition}

\begin{definition}
Let $\mathcal{F} = \{F^{\lambda}R_{m}\}$ be a $\mathbf{Z}$-filtration  of $R$. 
\begin{enumerate}[(1)]
\item The Rees algebra of $\mathcal{F}$ is a $\mathbf{C}[t]$-algebra defined by 
\begin{align*}
\mathrm{Rees}(\mathcal{F}) \coloneqq \bigoplus_{m=0}^{\infty}\bigoplus_{\lambda \in \mathbf{Z}}t^{-\lambda}F^{\lambda}R_{m}. 
\end{align*}
We say that $\mathcal{F}$ is finitely generated if its Rees algebra is a finitely generated graded $\mathbf{C}[t]$-algebra. 
\item The associated graded algebra of $\mathcal{F}$ is a graded $\mathbf{C}$-algebra defined by 
\begin{align*}
\mathrm{Gr}^{\mathcal{F}}\mathcal{R} 
= \bigoplus_{m=0}^{\infty}\mathrm{Gr}_{m}^{\mathcal{F}}\mathcal{R} 
\coloneqq \mathcal{R}/t\mathcal{R}
= \bigoplus_{m=0}^{\infty}\bigoplus_{\lambda \in \mathbf{Z}}F^{\lambda}R_{m}/F^{\lambda+1}R_{m}. 
\end{align*}
\end{enumerate}
\end{definition}

Given a finitely generated $\mathbf{Z}$-filtration $\mathcal{F} = \{F^{\lambda}R_{m}\}$, one can construct a test configuration for $(X, L)$ as follows. 
Let $\mathcal{R} \coloneqq \mathrm{Rees}(\mathcal{F})$ be the Rees algebra of the filtration $FR$. 
Then $\mathcal{R}$ is a torsion free $\mathbf{C}[t]$-algebra. 
Since $\mathbf{C}[t]$ is a principal ideal domain, $\mathcal{R}$ is flat over $\mathbf{C}[t]$. 
Let 
\begin{align*}
\mathcal{X} \coloneqq \Proj\mathcal{R},\quad \mathcal{L} \coloneqq \mathcal{O}_{\mathcal{X}}(1), 
\end{align*}
and $\pi \colon \mathcal{X} \to \mathbf{A}^{1}$ be the structure morphism. 
Then $(\mathcal{X}, \mathcal{L})$ is a flat family over $\mathbf{A}^{1}$. 
Further, for any $a \in \mathbf{A}^{1}(\mathbf{C})$ we have 
\begin{align*}
\mathcal{X}_{a} 
&= \mathcal{X} \times_{\mathbf{A}^{1}}\Spec(\mathbf{C}[t]/(t-a)\mathbf{C}[t]) \\
&= \Proj (\mathcal{R}/(t-a)\mathcal{R}) \\
&= 
\begin{cases}
\Proj R & a \neq 0, \\
\Proj \mathrm{Gr}^{\mathcal{F}}\mathcal{R} & a = 0. 
\end{cases}
\end{align*}
Hence $(\mathcal{X}_{a}, \mathcal{L}_{a}) = (X, L)$ for any $a \in \mathbf{A}^{1} \setminus\{0\}$ and $(\mathcal{X}, \mathcal{L})$ is a test configuration for $(X, L)$. 

\begin{example}[Test configurations]
Let $(\mathcal{X}, \mathcal{L})$ be a test configuration for $(X, L)$ of exponent $1$. 
For each $\lambda \in \mathbf{Z}$, define $F^{\lambda}R_{m}$  to be the image of the restriction map 
\begin{align*}
H^{0}(\mathcal{X}, \mathcal{L}^{m})_{\lambda} \to H^{0}(\mathcal{X}_{1}, \mathcal{L}_{1}^{m}) = H^{0}(X, L^{m}). 
\end{align*}
By regarding $H^{0}(\mathcal{X} \setminus \mathcal{X}_{0}, \mathcal{L}^{m})$ as a $\mathbf{C}[t, t^{-1}]$ module, $F^{\lambda}R_{m}$ can be written as 
\begin{align*}
F^{\lambda}R_{m} = \{\sigma \in H^{0}(X, L^{m}) \mid t^{-\lambda}\overline{\sigma} \in H^{0}(\mathcal{X}, \mathcal{L}^{m})\}, 
\end{align*}
where $\overline{\sigma} \in H^{0}(\mathcal{X} \setminus \mathcal{X}_{0}, \mathcal{L}^{m})$ is the $\mathbf{G}_{m}$-equivariant section defined by $\sigma \in H^{0}(X, L^{m})$. 
By setting $F^{\lambda}R_{m} = F^{\lceil \lambda \rceil}R_{m}$ for each $\lambda \in \mathbf{R}$, we obtain a $\mathbf{Z}$-filtration of $R$.  
We denote the filtration by $\mathcal{F}(\mathcal{X}, \mathcal{L})$. 
\end{example}

\begin{example}[Translations and scalings]\label{Filt_induced_by_tc}
Let $\mathcal{F} = \{F^{\lambda}R_{m}\}$ be a filtration of $R$, and set $\mathcal{R} \coloneqq \mathrm{Rees}(\mathcal{F})$. 
\begin{enumerate}[(1)]
\item Let $c \in \mathbf{R}$. 
A translation $\mathcal{F}(c) = \{F(c)^{\lambda}R_{m}\}$ of $\mathcal{F}$ is defined by 
\begin{align*}
F(c)^{\lambda}R_{m} \coloneqq F^{\lambda - mc}R_{m}. 
\end{align*}
If $c \in \mathbf{Z}$ and $\mathcal{F}$ is a $\mathbf{Z}$-filtraion, then $\mathcal{F}(c)$ is also a $\mathbf{Z}$-filtraion. 
Suppose $c \in \mathbf{Z}$, and $\mathcal{F}$ is a finitely generated $\mathbf{Z}$-filtration. 
Let $(\mathcal{X}, \mathcal{L})$ be a test configuration associated to $\mathcal{F}$. 
Then $\mathcal{X} = \Proj \mathcal{R}$ and $\mathcal{L} = \mathcal{O}_{\mathcal{X}}(1) = \widetilde{\mathcal{R}(1)}$. 
By setting $\mathcal{I}_{\mathcal{X}_{0}}$ to be the ideal sheaf of $\mathcal{X}_{0}$ in $\mathcal{X}$, we have 
\begin{align*}
\mathcal{O}_{\mathcal{X}}(c\mathcal{X}_{0}) = \mathcal{I}_{\mathcal{X}_{0}}^{-c} = \widetilde{t^{-c}\mathcal{R}}. 
\end{align*}
Hence we obtain 
\begin{align*}
(\mathcal{L}+c\mathcal{X}_{0})^{m} 
&= \mathcal{L}^{m} \otimes_{\mathcal{O}_{\mathcal{X}}}\mathcal{I}_{\mathcal{X}_{0}}^{-mc} \\
&= (\mathcal{R}(1)^{m} \otimes_{\mathcal{R}}t^{-mc}\mathcal{R})^{\widetilde{}} \\
&= (t^{-mc}\mathcal{R}(m))^{\widetilde{}}
\end{align*}
and 
\begin{align*}
H^{0}(\mathcal{X}, (\mathcal{L}+c\mathcal{X}_{0})^{m}) 
&= (t^{-mc}\mathcal{R}(m))_{0} 
= t^{-mc}\mathcal{R}_{m} \\
&= t^{-mc}\bigoplus_{\lambda \in \mathbf{Z}}t^{-\lambda}F^{\lambda}R_{m} \\
&= \bigoplus_{\lambda \in \mathbf{Z}}t^{-(\lambda + mc)}F^{\lambda}R_{m} \\
&= \bigoplus_{\lambda \in \mathbf{Z}}t^{-\lambda}F^{\lambda - mc}R_{m} \\
&= \bigoplus_{\lambda \in \mathbf{Z}}t^{-\lambda}F(c)^{\lambda}R_{m}. 
\end{align*}
This shows that 
\begin{align*}
\bigoplus_{m=0}^{\infty}H^{0}(\mathcal{X}, (\mathcal{L}+c\mathcal{X}_{0})^{m}) 
= \bigoplus_{m=0}^{\infty}\bigoplus_{\lambda \in \mathbf{Z}}t^{-\lambda}F(c)^{\lambda}R_{m} 
= \mathrm{Rees}(\mathcal{F}(c))
\end{align*}
and that the translation $(\mathcal{X}, \mathcal{L}+c\mathcal{X}_{0})$ is the test configuration associated to $\mathcal{F}(c)$. 
Note that 
\begin{align*}
\mathcal{R}(c) 
\coloneqq \mathrm{Rees}(\mathcal{F}(c)) 
\cong \bigoplus_{m=0}^{\infty}\bigoplus_{\lambda \in \mathbf{Z}}t^{-\lambda}F^{\lambda}R_{m} 
= \mathcal{R} 
\end{align*}
as a graded ring. 
Hence we have $\Proj \mathcal{R}(c) = \Proj \mathcal{R} = \mathcal{X}$. 
\item Let $d \in \mathbf{Z}_{>0}$. 
A scaling $\mathcal{F}_{d} = \{F_{d}^{\lambda}R_{m}\}$ of $\mathcal{F}$ is defined by 
\begin{align*}
F_{d}^{\lambda}R_{m} \coloneqq F^{\lceil \lambda/d \rceil}R_{m}. 
\end{align*}
If $\mathcal{F}$ is a $\mathbf{Z}$-filtraion, then $\mathcal{F}_{d}$ is also a $\mathbf{Z}$-filtraion. 
Suppose $\mathcal{F}$ is a finitely generated $\mathbf{Z}$-filtration. 
Let $(\mathcal{X}, \mathcal{L})$ be a test configuration associated to $\mathcal{F}$. 
Then the base change $\mathcal{X}_{d}$ of $\mathcal{X}$ by the homomorphism $\mathbf{C}[t] \ni t \mapsto s^{d} \in \mathbf{C}[s]$ is given by 
\begin{align*}
&\mathcal{X}_{d} 
= \Proj\mathcal{R} \times_{\Spec \mathbf{C}[t]}\Spec\mathbf{C}[s] 
= \Proj(\mathcal{R} \otimes_{\mathbf{C}[t]}\mathbf{C}[s]), \\ 
&\mathcal{L}_{d} = p^{\ast}\mathcal{L} = p^{\ast}\mathcal{O}_{\mathcal{X}}(1) = p^{\ast}\widetilde{\mathcal{R}(1)}= ((\mathcal{R} \otimes_{\mathbf{C}[t]}\mathbf{C}[s])(1))\,\tilde{} = \mathcal{O}_{\mathcal{X}_{d}}(1) ,  
\end{align*}
where $p \colon \mathcal{X}_{d} \to \Proj\mathcal{R} = \mathcal{X}$ is the natural projection. 
As a $\mathbf{C}[t]$-algebra, we have
\begin{align*}
\mathcal{R} \otimes_{\mathbf{C}[t]}\mathbf{C}[s] 
&= \left(\bigoplus_{m=0}^{\infty}\bigoplus_{\lambda \in \mathbf{Z}}t^{-\lambda}F^{\lambda}R_{m}\right) \otimes_{\mathbf{C}[t]}\mathbf{C}[s] \\
&\cong \bigoplus_{m=0}^{\infty}\bigoplus_{\lambda \in \mathbf{Z}}s^{-\lambda}F^{\lceil \frac{\lambda}{d} \rceil}R_{m} \\
&= \mathrm{Rees}(\mathcal{F}_{d}), 
\end{align*}
which is also an isomprphism as $\mathbf{C}[s]$-algebras. 
Hence the scaling $(\mathcal{X}_{d}, \mathcal{L}_{d})$ is the test configuration associated to $\mathcal{F}_{d}$. 
\end{enumerate}
\end{example}

\subsection{Limit measures}
Let $(X, L)$ be a polarized manifold and $\mathcal{F}$ a filtration of $R = R(X, L)$. 
For each $m \in \mathbf{Z}_{> 0}$, we define a Borel probability measure $\nu_{m}$ on $\mathbf{R}$ by 
\begin{align*}
\nu_{m} \coloneqq -\frac{d}{d\lambda}\left(\frac{1}{N_{m}}\dim F^{m\lambda}R_{m}\right), 
\end{align*}
where the derivative is taken in the sense of distributions. 
We call it \emph{the normalized weight measure}. 
By the linearly boundedness of $\mathcal{F}$, $\nu_{m}$ has uniformly bounded support. 
To describe its limit, to each $\lambda \in \mathbf{R}$, we define a graded subalgebra $R^{(\lambda)}$ of $R(X, L)$ by 
\begin{align*}
R^{(\lambda)} \coloneqq \bigoplus_{m=0}^{\infty}F^{m\lambda}R_{m}. 
\end{align*}
\emph{The volume} of $R^{(\lambda)}$ is defined by 
\begin{align*}
\mathrm{vol}(R^{(\lambda)}) \coloneqq \limsup_{m \to \infty}\frac{1}{N_{m}}\dim F^{m\lambda}R_{m}. 
\end{align*}

\begin{theorem}[{\cite[Theorem 5.3]{BHJ17}}]
The sequence of normalized weight measures $\{\nu_{m}\}_{m = 0}^{\infty}$ converges weakly to the probability measure 
\begin{align*}
\mathrm{DH}(\mathcal{F}) \coloneqq -\frac{d}{d\lambda}\mathrm{vol}(R^{(\lambda)}). 
\end{align*}
\end{theorem}
We call $\mathrm{DH}(\mathcal{F})$ \emph{the Duistermaat-Heckman measure} of the filtration $\mathcal{F}$. 
\begin{example}[Test configurations]
Let $(\mathcal{X}, \mathcal{L})$ be a test configuration for $(X, L)$ of exponent $1$ and $\mathcal{F}$ the filtration of $R$ as in Example \ref{Filt_induced_by_tc}. 
By \cite[Lemma 8.5]{RW14}, the normalized weight measure is given by 
\begin{align*}
\nu_{m} 
&= -\frac{d}{d\lambda}\left(\frac{1}{N_{m}}\dim F^{m\lambda}R_{m}\right) \\
&= \frac{1}{N_{m}}\sum_{\lambda \in \mathbf{Z}}\dim H^{0}(\mathcal{X}_{0}, \mathcal{L}_{0}^{m})_{\lambda}\delta_{\lambda/m}
\end{align*}
for any $m \in \mathbf{Z}_{\geq 0}$. 
Hence the limit coincides with the Duistermaat-Heckman measure: 
\begin{align*}
\mathrm{DH}(\mathcal{F}) = \mathrm{DH}(\mathcal{X}, \mathcal{L}). 
\end{align*}
\end{example}

\subsection{Reduced non-Archimedean $J$-functional}
Let $(X, L)$ be a polarized manifold, and set $R = R(X, L)$. 
\begin{definition}
Let $\mathcal{F} =\{F^{\lambda}R_{m}\}$ be a filtration of $R$. 
\begin{enumerate}[(1)]
\item The non-Archimedean Monge-Amp\`ere energy of $\mathcal{F}$ is defined by 
\begin{align*}
E^{NA}(\mathcal{F}) \coloneqq \int_{\mathbf{R}}\lambda\,d\mathrm{DH}(\mathcal{F}). 
\end{align*}
\item The non-Archimedean J-functional of $\mathcal{F}$ is defined by 
\begin{align*}
J^{NA}(\mathcal{F}) \coloneqq \sup\mathrm{supp}(\mathrm{DH}(\mathcal{F})) - E^{NA}(\mathcal{F}). 
\end{align*}
\end{enumerate}
\end{definition}

\begin{definition}
Let $\mathcal{F}$ be a $T$-equivariant filtration of $R$. 
We denote the weight decomposition of $T$-action on each $F^{\lambda}R_{m}$ by 
\begin{align*}
F^{\lambda}R_{m} = \bigoplus_{\alpha \in M}(F^{\lambda}R_{m})_{\alpha}. 
\end{align*}
Here $(F^{\lambda}R_{m})_{\alpha}$ is the $\alpha$-weight subspace of $F^{\lambda}R_{m}$. 
For each $\xi \in N_{\mathbf{R}}$, we define the $\xi$-twist $\mathcal{F}_{\xi} = \{F_{\xi}^{\lambda}R_{m}\}$ of $\mathcal{F}$ by 
\begin{align*}
F_{\xi}^{\lambda}R_{m} \coloneqq \bigoplus_{\alpha \in M}(F_{\xi}^{\lambda}R_{m})_{\alpha},\quad (F_{\xi}^{\lambda}R_{m})_{\alpha} = (F^{\lambda - \langle \alpha, \xi \rangle}R_{m})_{\alpha} 
\end{align*}
for each $\lambda \in \mathbf{R}$ and $m \in \mathbf{Z}_{\geq 0}$. 
\end{definition}

Note that the $\xi$-twist of a filtration $\mathcal{F}$ may be an $\mathbf{R}$-filtration even if $\mathcal{F}$ is a $\mathbf{Z}$-filtration. 

\begin{definition}
Let $\mathcal{F}$ be a $T$-equivariant filtration of $R$. 
The reduced non-Archimedean $J$-functional of $\mathcal{F}$ is defined by 
\begin{align*}
J_{T}^{NA}(\mathcal{F}) \coloneqq \inf_{\xi \in N_{\mathbf{R}}}J^{NA}(\mathcal{F}_{\xi}). 
\end{align*}
If $(\mathcal{X}, \mathcal{L})$ is a test configuration for $(X, L)$, the reduced non-Archimedean $J$-functional of $(\mathcal{X}, \mathcal{L})$ is defined by $J_{T}^{NA}(\mathcal{X}, \mathcal{L}) \coloneqq J_{T}^{NA}(\mathcal{F}(\mathcal{X}, \mathcal{L}))$. 
\end{definition}

\begin{theorem}[{\cite[Theorem B]{Hisa16}, \cite[Theorem 3.14]{Li19}}]\label{slope_J_T}
Let $(\mathcal{X}, \mathcal{L})$ be a $T$-equivariant test configuration for $(X, L)$. 
Then 
\begin{align*}
\lim_{t \to \infty}\frac{J_{T}(\varphi^{t})}{t} = J_{T}^{NA}(\mathcal{X}, \mathcal{L}). 
\end{align*}
\end{theorem}

\begin{definition}[see also \cite{BHJ17}]
Let $p \in [1, \infty]$. 
\begin{enumerate}[(1)] 
\item Let $\mathcal{F}$ be a filtration of $R$. 
The $L^{p}$-norm $\|\mathcal{F}\|_{p}$ of $\mathcal{F}$ is defined as the $L^{p}$-norm of $\lambda - \overline{\lambda}$ with respect to $\mathrm{DH}(\mathcal{F})$, where 
\begin{align*}
\overline{\lambda} \coloneqq \int_{\mathbf{R}}\lambda\,d\mathrm{DH}(\mathcal{F}) 
\end{align*}
is the barycenter of $\mathrm{DH}(\mathcal{F})$. 
If $(\mathcal{X}, \mathcal{L})$ is a test configuration for $(X, L)$, the $L^{p}$-norm of $(\mathcal{X}, \mathcal{L})$ is defined by $\|(\mathcal{X}, \mathcal{L})\|_{p} \coloneqq \|\mathcal{F}(\mathcal{X}, \mathcal{L})\|_{p}$. 
\item Let $\mathcal{F}$ be a $T$-equivariant filtration of $R$. 
Let $\mathcal{F}$ be a filtration of $R$. 
The reduced $L^{p}$-norm $\|\mathcal{F}\|_{p}$ of $\mathcal{F}$ is defined by 
\begin{align*}
\|\mathcal{F}\|_{p, T} \coloneqq \inf_{\xi \in N_{\mathbf{R}}}\|\mathcal{F}_{\xi}\|_{p}. 
\end{align*}
If $(\mathcal{X}, \mathcal{L})$ is a $T$-equivariant test configuration for $(X, L)$, the reduced $L^{p}$-norm of $(\mathcal{X}, \mathcal{L})$ is defined by $\|(\mathcal{X}, \mathcal{L})\|_{p, T} \coloneqq \|\mathcal{F}(\mathcal{X}, \mathcal{L})\|_{p, T}$. 
\end{enumerate}
\end{definition}

According to \cite[Lemma 7.10]{BHJ17}, we have the following equivalence between the (reduced) $L^{1}$-norm and the (reduced) non-Archimedean $J$-functional. 

\begin{theorem}[{see also \cite[Theorem 7.9]{BHJ17}}]
Let $c_{n} \coloneqq 2n^{2}/(n+1)^{n+1}$. 
\begin{enumerate}[(1)]
\item For each filtration $\mathcal{F}$, we have 
\begin{align*}
c_{n}J^{NA}(\mathcal{F}) \leq \|\mathcal{F}\|_{1} \leq 2J^{NA}(\mathcal{F}). 
\end{align*}
\item For each $T$-equivariant filtration $\mathcal{F}$, we have 
\begin{align*}
c_{n}J_{T}^{NA}(\mathcal{F}) \leq \|\mathcal{F}\|_{1, T} \leq 2J_{T}^{NA}(\mathcal{F}). 
\end{align*}
\end{enumerate}
\end{theorem}

\begin{corollary}
Let $(\mathcal{X}, \mathcal{L})$ be a normal test configuration for $(X, L)$. 
\begin{enumerate}[(1)]
\item (\cite[Theorem 6.8 and Theorem 7.9]{BHJ17}) The following conditions are equivalent: 
\begin{enumerate}[(i)]
\item $(\mathcal{X}, \mathcal{L})$ is trivial. 
\item $\|(\mathcal{X}, \mathcal{L})\|_{p} = 0$ for some $p \in [1, \infty]$. 
\item $J^{NA}(\mathcal{X}, \mathcal{L}) = 0$. 
\end{enumerate}
\item (\cite[Theorem B]{Hisa17}) Suppose $(\mathcal{X}, \mathcal{L})$ is $T$-equivariant. 
Then the following conditions are equivalent: 
\begin{enumerate}[(i)]
\item $(\mathcal{X}, \mathcal{L})$ is product. 
\item $\|(\mathcal{X}, \mathcal{L})\|_{p, T} = 0$ for some $p \in [1, \infty]$. 
\item $J_{T}^{NA}(\mathcal{X}, \mathcal{L}) = 0$. 
\end{enumerate}
\end{enumerate}
\end{corollary}

\subsection{Uniform relative K-stability}
\begin{definition}
Let $(X, L)$ be a polarized algebraic manifold and $T$ a maximal algebraic torus of $\mathrm{Aut}^{0}(X)$. 
\begin{enumerate}[(1)]
\item $(X, L)$ is relatively K-semistable if $M_{V}^{NA}(\mathcal{X}, \mathcal{L}) \geq 0$ for any $T$-equivariant normal test configuration $(\mathcal{X}, \mathcal{L})$ for $(X, L)$. 
\item $(X, L)$ is relatively K-polystable if $(X, L)$ is relatively K-semistable and $M_{V}^{NA}(\mathcal{X}, \mathcal{L}) = 0$ if and only if $(\mathcal{X}, \mathcal{L})$ is product. 
\item $(X, L)$ is relatively K-unstable if $(X, L)$ is not relatively K-semistable. 
\item $(X, L)$ is uniformly relatively K-polystable if there exists a $\delta > 0$ such that 
\begin{align*}
M_{V}^{NA}(\mathcal{X}, \mathcal{L}) \geq \delta J_{T}^{NA}(\mathcal{X}, \mathcal{L}) 
\end{align*}
for any $T$-equivariant normal test configuration $(\mathcal{X}, \mathcal{L})$ for $(X, L)$. 
\end{enumerate}
\end{definition}

By the slope formula, we have the following 
\begin{theorem}\label{rel_K_coer}
A polarized manifold $(X, L)$ is uniformly relatively K-polystable if the relative K-energy is $T$-coercive. 
\end{theorem}
\begin{proof}
Assume that the relative K-energy is $T$-coercive. 
Then there exists $\delta, C >0$ such that 
\begin{align*}
M_{V}(\varphi) \geq \delta J_{T}(\varphi) - C 
\end{align*}
for any $\varphi \in \mathcal{H}_{0}^{S}$. 
Let $\varphi = (\mathcal{X}, \mathcal{L})$ be a test configuration for $(X, L)$ of exponent $r$, and $\{\varphi^{t}\}_{t \in [0, +\infty)}$ the subgeodesic ray corresponding to $\varphi$. 
Then, Theorem \ref{slope_formula} and Theorem \ref{slope_J_T} allow us to conclude 
\begin{align*}
M_{V}^{NA}(\mathcal{X}, \mathcal{L}) \geq \delta J_{T}^{NA}(\mathcal{X}, \mathcal{L}), 
\end{align*}
as required. 
\end{proof}

Combining Theorems \ref{ext_coer_equiv} and \ref{rel_K_coer}, we further obtain the following corollary. 
\begin{corollary}
A polarized manifold $(X, L)$ is uniformly relatively K-polystable if $(X, L)$ admits an $S$-invariant extremal K\"ahler metric. 
\end{corollary}

\section{Convex Preliminaries}\label{sec:CvxPre}
\subsection{Polarized toric manifolds}\label{sec:PolToric}
In this subsection, we briefly review the definition and construction of polarized toric manifolds. 
See \cite{G94a} for details. 

A $2n$-dimensional symplectic toric manifold is a $2n$-dimensional compact connected symplectic manifold $(X, \omega)$ endowed with an effective Hamiltonian action of an $n$-dimensional compact torus $S = (S^{1})^{n}$ and the moment map $\mu  \colon X \to M_{\mathbf{R}}$. 
Let $(X, \omega)$ be a $2n$-dimensional symplectic toric manifold. 
According to the convexity theorem of Atiyah \cite{At82} and Guillemin-Sternberg \cite{GS82a}, the image of the moment map $P \coloneqq \mu(X) \subset M_{\mathbf{R}}$ is a convex polytope obtained as the convex hull of the image of the fixed points of the $S$-action. 
We call $P$ the \emph{moment polytope}. 
Further, since $(X, \omega)$ is toric, $P$ satisfies the following special conditions, usually called the \emph{Delzant condition}: 
\begin{enumerate}[(1)]
\item There are exactly $n$ edges meeting at each vertex. 
\item Each edge meeting at the vertex $v \in P$ is of the form $v + tu_{i}$, $t \geq 0$, where $u_{i} \in N$ is primitive. 
\item For each vertex, the corresponding $u_{1}, \ldots, u_{n}$ generate the lattice of $N_{\mathbf{R}}$ over $\mathbf{Z}$. 
\end{enumerate}
A convex polytope satisfiying the conditions above is called a \emph{Delzant polytope}. 
In \cite{Del88}, Delzant showed that there is a one to one correspondence between isomorphism classes of symplectic toric manifolds and Delzant polytopes. 
More precisely, for each Delzant polytope $P$, one can construct in a canonical way a symplectic toric manifold in such a way that $P$ may be identified with the moment polytope. 
This construction is known as the \emph{Delzant construction}. 
Let $P$ be a Delzant polytope and $(X_{P}, \omega_{P})$ the corresponding symplectic toric manifold with the moment map $\mu_{P} \colon X_{P} \to M_{\mathbf{R}}$. 
By the Delzant construction, there is an $S$-invariant complex structure $J_{P}$ compatible with $\omega_{P}$. 
In particular, $(X_{P}, \omega_{P}, J_{P})$ is a K\"ahler manifold. 
Since the $S$-action preserves the complex structure $J_{P}$, it can be canonically extended to an action of $T \coloneqq S^{\mathbf{C}} = (\mathbf{G}_{m})^{n}$ preserving $J_{P}$ (\cite[Theorem 4.4]{GS82b}). 
Then $X_{P}$ contains $T$ as an open dense orbit, and hence $X_{P}$ has a structure of a nonsingular toric variety with the fan associated to the polytope $P$ (\cite[Lemma 9.2]{LT97}). 

Let 
\begin{align}\label{facet_Rep}
P = \{x \in M_{\mathbf{R}} \mid \langle \lambda_{j}, x \rangle + d_{j} \geq 0\ (j=1, \ldots, r)\} 
\end{align}
be the facet representaion of $P$. 
Here $\langle \cdot, \cdot \rangle$ is the natural pairing between $N_{\mathbf{R}}$ and $M_{\mathbf{R}}$, $r$ is the number of facets of $P$, $\lambda_{j} \in N$, $d_{j} \in \mathbf{R}$, and each $\lambda_{j}$ is primitive. 
Let $P^{\circ}$ denote the interior of $P$. 
Then $X_{P}^{\circ} \coloneqq \mu_{P}^{-1}(P^{\circ})$ is a dense open subset of $X_{P}$ where the $S$-action is free. 
This coincides with the open dense orbit of the $T$-action described above. 
Also, for any $k$-dimensional face $F$ of $P$, $\mu_{P}^{-1}(F)$ is an $T$-invariant $k$-dimensional connected complex submanifold of $X_{P}$. 
In particular, if $F$ is a facet of $P$ then $\mu_{P}^{-1}(F)$ is a $T$-invariant prime divisor of $X_{P}$. 
For each $j = 1, \ldots, r$, let $F_{j}$ denote the facet of $P$ defined by 
\begin{align*}
F_{j} = \{x \in P \mid \langle \lambda_{j}, x \rangle + d_{j} = 0\},  
\end{align*}
and $D_{j} \coloneqq \mu_{P}^{-1}(F_{j})$. 
By setting $c_{j}$ the Poincare dual of $D_{j}$, we have 
\begin{align*}
[\omega_{P}] = \sum_{j=1}^{r}d_{j}c_{j} 
\end{align*}
as a de Rham cohomology class. 
If the Delzant polytope $P$ is \emph{integral}, that is, each vertex of $P$ belongs to $\mathbf{Z}^{n}$, then $d_{j} \in \mathbf{Z}$ ($j = 1, \ldots, r$) and hence $[\omega_{P}] \in H^{2}(X, \mathbf{Z})$. 
Hence, in this case there exists a $T$-equivariant holomorphic line bundle $L_{P}$ over $X_{P}$ which satisfies $c_{1}(L_{P}) = [\omega_{P}]$. 
By the Kodaira embedding theorem, $L_{P}$ is an ample line bundle over $X_{P}$. 
We call $(X_{P}, L_{P})$ a \emph{polarized toric manifold associated to $P$}. 

\subsection{Symplectic potentials}\label{sec:SympPot}
In this subsection, we quickly review differential-geometric aspects of toric K\"ahler manifolds. 
For details, see \cite{Ab03} and \cite{Ab98}, \cite{G94a}, \cite{G94b}. 

Let $P \subset M_{\mathbf{R}}$ be an $n$-dimensional integral Delzant polytope, and $(X_{P}, L_{P})$ the polarized toric manifold corresponding to $P$. 
We choose an $S$-invariant K\"ahler metric $\omega \in c_{1}(L_{P})$. 
Since $X_{P}^{\circ}$ is holomorphically isomorphic to $T$, there is an $S$-invariant smooth function $\phi \colon T \to \mathbf{R}$ so that 
\begin{align*}
\omega = dd^{c}\phi
\end{align*}
on $X_{P}^{\circ}$ (\cite[Theorem 4.3]{G94b}). 
Through the identification $X_{P}^{\circ} \cong (\mathbf{C}^{\times})^{n} \cong \mathbf{R}^{n} \times S$, we can regard $\phi$ as a smooth function defined on $\mathbf{R}^{n}$. 

\begin{remark}
Let $\tau = (\tau^{1}, \ldots, \tau^{n})$, $\tau^{i} = \exp((-1/2)y^{i} + 2\pi \sqrt{-1}\theta^{i})$ ($i=1, \ldots, n$) denote the coordinate of $T$. 
Then, for each $i=1, \ldots, n$ we have 
\begin{align*}
\frac{d\tau^{i}}{\tau^{i}} = -\frac{1}{2}dy^{i} + 2\pi\sqrt{-1}d\theta^{i},\quad \frac{d\bar{\tau}^{i}}{\bar{\tau}^{i}} = -\frac{1}{2}dy^{i} - 2\pi\sqrt{-1}d\theta^{i}
\end{align*}
and 
\begin{align*}
\frac{\sqrt{-1}}{2\pi}\frac{d\tau^{i} \wedge d\bar{\tau}^{i}}{|\tau^{i}|^{2}} = dy^{i} \wedge d\theta^{i}. 
\end{align*}
Further, since 
\begin{align*}
\tau^{i}\frac{\partial}{\partial \tau^{i}} = -\frac{\partial}{\partial y^{i}} -\frac{\sqrt{-1}}{4\pi}\frac{\partial}{\partial \theta^{i}},\quad \bar{\tau}^{i}\frac{\partial}{\partial \bar{\tau}^{i}} = -\frac{\partial}{\partial y^{i}} +\frac{\sqrt{-1}}{4\pi}\frac{\partial}{\partial \theta^{i}}, 
\end{align*}
we obtain 
\begin{align*}
dd^{c}\phi 
= \frac{\sqrt{-1}}{2\pi}\frac{\partial^{2}\phi}{\partial \tau^{i}\partial \bar{\tau}^{j}}d\tau^{i} \wedge d\bar{\tau}^{i} 
= \frac{\sqrt{-1}}{2\pi}\frac{\partial^{2}\phi}{\partial y^{i}\partial y^{j}}\frac{d\tau^{i}}{\tau^{i}} \wedge \frac{d\bar{\tau}^{i}}{\bar{\tau}^{j}} 
\end{align*}
for any $\phi \in C^{\infty}(T, \mathbf{R})^{S}$. 
In particular, $dd^{c}\phi$ is a K\"ahler form on $T$ if and only if $\phi$ is a strictly convex function on $\mathbf{R}^{n}$. 
\end{remark}

By using the moment map, one can describe $\phi$ in terms of a convex function on $P$, which is called a symplectic potential. 
Indeed, up to an additive constant, the derivative $\nabla \phi$ coincides with the moment map for the K\"ahler metric $\omega$. 
Hence we may assume that $\nabla \phi$ precisely coincides with the moment map by adding a suitable linear function. 
Since $\phi$ is strictly convex, the moment map $\nabla \phi$ is a diffeomorphism from $\mathbf{R}^{n}$ onto $P^{\circ}$. 
Now we introduce the coordinate
\begin{align*}
x = (\nabla \phi)(y), 
\end{align*}
and define the function $u \colon P^{\circ} \to \mathbf{R}$ as the Legendre dual of $\phi$, i.e., 
\begin{align*}
u(x) + \phi(y) = \langle x, y \rangle. 
\end{align*}
$u$ is called the \emph{symplectic potential} of the K\"ahler metric $\omega$. 

\begin{example}[\cite{G94b}]
Let $\ell_{j}(x) \coloneqq \langle \lambda_{j}, x \rangle + d_{j}$ ($j=1, \ldots, r$) and $\ell_{\infty} \coloneqq \sum_{j=1}^{r}\ell_{j}$. 
Then, by setting $\phi_{P} \coloneqq \mu_{P}^{\ast}(\ell_{\infty} - \sum_{j=1}^{r}d_{j}\ell_{j})$, we have
\begin{align*}
\omega_{P} = dd^{c}\phi_{P}
\end{align*}
on $X_{P}^{\circ}$. 
The symplectic potential of the K\"ahler metric $\omega_{P}$ is given by 
\begin{align*}
u_{P} \coloneqq \sum_{j=1}^{r}\ell_{j}\log \ell_{j}, 
\end{align*}
usually called the Guillemin potential. 
\end{example}

By the theory of Guillemin \cite{G94b} and Abreu \cite{Ab03}, there is the following correspondence for $S$-invariant K\"ahler metrics in $c_{1}(L_{P})$ and symplectic potentials. 
Let $u$ be a strictly convex function on $P^{\circ}$. 
The function $u$ is said to satisfy the {\it Guillemin's boundary condition} if it satisfies the following conditions: 
\begin{enumerate}[\upshape(1)]
\item $u-u_{P} \in C^{\infty}(P)$; 
\item For any face $F$ of $P$, $u|_{F^{\circ}}$ is a strictly convex function on $F^{\circ}$. 
\end{enumerate}
Let $\mathcal{S}$ denote the set of all strictly convex functions on $P^{\circ}$ which satisfy the Guillemin's boundary condition. 

\begin{theorem}[\cite{Ab03}, \cite{Apo19}, \cite{Don05}, \cite{G94b}]\label{GuiAb}
Every $S$-invariant K\"ahler metric in the class $c_{1}(L_{P})$ has a symplectic potential $u$ in $\mathcal{S}$. 
Conversely, every function belongs to $\mathcal{S}$ is a symplectic potential for some $S$-invariant K\"ahler metric in $c_{1}(L_{P})$. 
\end{theorem}
For later use, for each $u \in \mathcal{S}$ we denote $\varphi_{u} \in \mathcal{H}(X_{P}, \omega_{P})^{S}$ by the K\"ahler potential corresponding to $u$, which is determined by $\varphi_{u}|_{X_{P}^{\circ}} = \phi - \phi_{P}$, where $\phi$ is the Legendre dual of $u$. 
By using the correspodence in Theorem \ref{GuiAb}, the scalar curvature of an $S$-invariant K\"ahler metric $\omega \in c_{1}(L_{P})$ is given by derivatives of the corresponding symplectic potential $u$ with respect to the symplectic coordinates $(x_{1}, \ldots, x_{n})$. 

\begin{theorem}[Abreu \cite{Ab98}]\label{AbScal}
Let $\omega$ be an $S$-invariant K\"ahler metric in the class $c_{1}(L_{P})$ and $u$ the symplectic potential of $\omega$. 
Then the scalar curvature $s(\omega)$ of $\omega$ is given by 
\begin{align*}
s(\omega) = s(u) \coloneqq -\sum_{i,j=1}^{n} \frac{\partial^2 u^{ij}}{\partial x_{i} \partial x_{j}}, 
\end{align*}
where $u^{ij}$ is the inverse of the Hessian $(u_{ij})$ of $u$: 
\begin{align*}
(u_{ij})=\left( \frac{\partial ^2 u}{\partial x_i \partial x_j }    \right). 
\end{align*}
\end{theorem}
Furthermore, by using Donaldson's integration by parts formula (\cite{Don02}), the average of the scalar curvature $\bar{s}$ is obtained as follows. 
Let us recall the Borel measure $\sigma$ on $\partial P$ defined by 
\begin{align*}
\sigma(A) = \sum_{j=1}^{r}\frac{|A \cap F_{j}|_{n-1}}{|\lambda_{j}|}, 
\end{align*}
where $|\cdot|_{n-1}$ be the $(n-1)$-dimensional Lebesgue measure. 
This measure arises naturally as the subleading coefficient of the Ehrhart polynomial of an integral polytope 
\begin{align*}
E_{P}(m) \coloneqq \#(mP \cap M) = \mathrm{vol}(P)m^{n} + \frac{\sigma(\partial P)}{2}m^{n-1} + O(m^{n-2}). 
\end{align*}

Let $\mathcal{C}_{\infty}$ denote the set of all continuous convex functions on $P$ which are smooth in the interior. 

\begin{theorem}[{\cite[Lemma 3.3.5]{Don02}}]\label{integration_by_parts}
Let $u \in \mathcal{S}$ and $f \in \mathcal{C}_{\infty}$. 
Then $u^{ij}f_{ij}$ is integrable on $P$ and 
\begin{align*}
\int_{P}u^{ij}f_{ij}\,dx 
&= \int_{\partial P}fd\sigma + \int_{P}f(u^{ij})_{ij}\,dx \\
&= \int_{\partial P}fd\sigma - \int_{P}fs(u)\,dx. 
\end{align*}
\end{theorem}

\begin{corollary}
The average of the scalar curvature $\overline{s}$ is given by 
\begin{align*}
\overline{s}=\frac{\sigma(\partial P)}{\mathrm{vol}(P)}. 
\end{align*}
\end{corollary}

\subsection{Toric test configurations}\label{sec:ToricTc}
In this subsection we explain the construction of toric test configurations. 
For this, we first recall the algebraic construction of toric varieties following \cite[Chapter 7]{CLS11} (see also \cite{Zie95}). 
Let $P \subset \mathbf{R}^{n}$ be an $n$-dimensional integral (possibly unbounded) polyhedron. 
Then $P$ has a unique facet representaion 
\begin{align*}
P =\{x \in M_{\mathbf{R}} \mid \langle\lambda_{j}, x\rangle + d_{j} \geq 0\ (j=1, \ldots, r) \}, 
\end{align*}
where $\lambda_{j} \in N$, $d_{j} \in \mathbf{Z}$ and $\lambda_{j}$ is primitive. 
Also, it can be written as a Minkowski sum 
\begin{align*}
P = Q + C, 
\end{align*} 
where  $Q$ is an integral polytope and $C$ is a strongly convex rational polyhedral cone (\cite[Theorem 1.2]{Zie95}). 
The cone part of $P$ is determined by 
\begin{align*}
C = \{x \in M_{\mathbf{R}} \mid \langle\lambda_{j}, x\rangle \geq 0\ (j=1, \ldots, r) \}. 
\end{align*}
(\cite[Proposition 1.12]{Zie95}). 
Following \cite{Zie95}, we call $C$ the \emph{recession cone} of $P$. 
Let $C(P)$ denote the \emph{cone of} $P$ defined by 
\begin{align*}
C(P) = \{(x, \rho) \in M_{\mathbf{R}} \times \mathbf{R} \mid \rho \geq 0, \langle \lambda_{j}, x \rangle + \rho d_{j} \geq 0\ (j=1, \ldots, r)\}. 
\end{align*}
Then it is easy to check that 
\begin{align*}
C(P) \cap (M_{\mathbf{R}} \times \{0\}) = C \times \{0\}, 
\end{align*}
and $\rho P$ is the slice of $C(P)$ at height $\rho$ for any $\rho > 0$. 
Let $S_{P} \coloneqq C(P)\cap (M \times \mathbf{Z})$. 
Also, we let $\sigma \coloneqq C \times \{0\}$ and $S_{\sigma} \coloneqq \sigma \cap (M \times \mathbf{Z})$. 
From Gordan's lemma, $S_{P}$ and $S_{\sigma}$ are affine semigroups. 
Let $R \coloneqq \mathbf{C}[S_{P}]$ be the semigroup ring of $S_{P}$. 
Then $R$ is a finitely generated $\mathbf{C}$-algebra. 
The character associated to $(\alpha, m) \in M \times \mathbf{Z}$ is written $\chi^{\alpha}s^{m}$, and $R$ is graded by height, i.e. $\deg(\chi^{\alpha}s^{m}) = m$. 
Consequently, we obtain a graded $\mathbf{C}$-algebra 
\begin{align*}
R = \bigoplus_{m=0}^{\infty}R_{m},\quad R_{0} = \mathbf{C}[S_{\sigma}]. 
\end{align*}
The inclusion $R_{0} \hookrightarrow R$ gives a structure of finitely generated $R_{0}$-algebra on $R$. 
Let $U_{P} \coloneqq \Spec R_{0}$, $X_{P} \coloneqq \Proj R$, and $\pi_{P} \colon X_{P} \to U_{P}$ be the structure morphism. 
Then $U_{P}$ is an affine toric variety, $X_{P}$ is a projective toric variety, and $\pi_{P}$ is a projective toric morphism. 
For each $j \in \{1, \ldots, r\}$, let $F_{j}$ denote the facet of $P$ defined by $F_{j} = \{x \in P \mid \langle \lambda_{j}, x \rangle + d_{j} = 0\}$. 
Associated to $F_{j}$, we have a $T$-invariant prime divisor $D_{j}$ of $X_{P}$. 
Then, 
\begin{align*}
D_{P} \coloneqq \sum_{j=1}^{r}d_{j}D_{j}
\end{align*}
is a $\pi_{P}$-ample Cartier divisor of $X_{P}$ (\cite[Example 7.2.8]{CLS11}). 
An important fact is that if $P$ is an integral Delzant polytope then $(X_{P}, L_{P}) \cong (X_{P}, \mathcal{O}_{X_{P}}(D_{P}))$.  

Now let $P \subset M_{\mathbf{R}}$ be an integral Delzant polytope, and identify $(X_{P}, L_{P})$ and $(X_{P}, \mathcal{O}_{X_{P}}(D_{P}))$ constructed above. 
Recall that a function $f \colon P \to \mathbf{R}$ is called \emph{rational piecewise affine convex} if $f$ is a convex function of the form 
\begin{align}\label{PLconv}
f = \max\{\ell_{1}, \ldots, \ell_{m}\} 
\end{align}
with each $\ell_{j}$ an affine function having rational coefficients. 
Given a rational piecewise affine convex function $f$, we can construct a $T$-equivariant test configuration $(\mathcal{X}_{f}, \mathcal{L}_{f})$ for $(X_{P}, L_{P})$ as follows. 
Choose an integer $L$ so that $L > \max_{P}f$, we define a rational convex polyhedron $\mathcal{P} \subset M_{\mathbf{R}} \times \mathbf{R}$ by 
\begin{align*}
\mathcal{P} = \{(x, y) \in M_{\mathbf{R}} \times \mathbf{R} \mid x \in P,\ f(x)-L \leq y\}. 
\end{align*}
By replacing $f$, $L$, $\mathcal{P}$ by $kf$, $kL$, $k\mathcal{P}$ for suitable $k \in \mathbf{Z}_{>0}$ if necessary, we may assume that each $\ell_{j}$ has integral coefficients and $\mathcal{P}$ is an integral polyhedron. 
Then the recession cone of $\mathcal{P}$ is given by $\mathcal{C} \coloneqq \{0\} \times \mathbf{R}_{\geq 0} \subset M_{\mathbf{R}} \times \mathbf{R}$, and the cone of $\mathcal{P}$ is given by 
\begin{align*}
C(\mathcal{P}) = (\mathcal{C} \times \{0\}) \cup \left(\bigcup_{\rho > 0}\rho\mathcal{P} \times \{\rho\}\right). 
\end{align*}
Let $S_{\mathcal{P}} \coloneqq C(\mathcal{P}) \cap (M \times \mathbf{Z}^{2})$, $\sigma \coloneqq \mathcal{C} \times \{0\}$, and $S_{\sigma} \coloneqq (\mathcal{C} \times \{0\}) \cap (M \times \mathbf{Z}^{2})$. 
Then $S_{\mathcal{P}}$ and $S_{\sigma}$ are affine semigroups. 
Let $\mathcal{R} \coloneqq \mathbf{C}[S_{\mathcal{P}}]$ be the semigroup ring of $S_{\mathcal{P}}$. 
Then $\mathcal{R}$ is a finitely generated $\mathbf{C}$-algebra. 
The character associated to $(\alpha, \lambda, m) \in M \times \mathbf{Z}^{2}$ is written $\chi^{\alpha}t^{\lambda}s^{m}$, and the grading of $R_{\mathcal{P}}$ is given by $\deg(\chi^{\alpha}t^{\lambda}s^{m}) = m$. 
Consequently, we have the graded $\mathbf{C}$-algebra 
\begin{align*}
&\mathcal{R} = \bigoplus_{m=0}^{\infty}\mathcal{R}_{m}, \\
&\mathcal{R}_{m} = 
\begin{cases}
\mathbf{C}[S_{\sigma}] = \mathbf{C}[t] & m = 0, \\
\mathrm{Vect}_{\mathbf{C}}\left\{\chi^{\alpha}t^{\lambda}s^{m} \left| 
\begin{array}{l}
\alpha \in mP \cap M, \lambda \in \mathbf{Z}, \\
m(f(\alpha/m)-L) \leq \lambda 
\end{array}\right.\right\}
& m > 0. 
\end{cases}
\end{align*}
The inclusion $\mathbf{C}[t] = \mathcal{R}_{0} \hookrightarrow \mathcal{R}$ gives a structure of $\mathbf{C}[t]$-algebra on $\mathcal{R}$. 
Note also that $\mathcal{R}$ is finitely generated as a $\mathbf{C}[t]$-algebra. 
By setting $\mathcal{D}_{\mathcal{P}}$ the divisor of $\Proj \mathcal{R}$ defined above, we have a polarized scheme  
\begin{align*}
(\mathcal{X}_{f}, \mathcal{L}_{f}) = (\Proj \mathcal{R}, \mathcal{O}_{\Proj \mathcal{R}}(\mathcal{D}_{\mathcal{P}})). 
\end{align*}
Then one can check that $\mathcal{O}_{\Proj \mathcal{R}}(\mathcal{D}_{\mathcal{P}}) = \mathcal{O}_{\Proj \mathcal{R}}(1)$. 
Let $\pi_{f} \colon \mathcal{X}_{f} \to \mathbf{A}^{1}$ be the structure morphism. 
Then $\pi_{f}$ is a projective morphism and hence proper. 
By construction, $\mathcal{X}_{f}$ is a normal toric variety and $\mathcal{L}_{f}$ is $\pi_{f}$-ample line bundle over $\mathcal{X}_{f}$. 
There is a $\mathbf{G}^{m+1}$-action on $\mathcal{R}$, which has weight $(\alpha, \lambda) \in M \times \mathbf{Z}$ for any $\chi^{\alpha}t^{\lambda} \in S_{\mathcal{P}}$. 
This induces a $T_{f} \coloneqq T \times \mathbf{G}_{m}$-action on $\mathcal{X}_{f}$ over $\mathbf{A}^{1}$. 
We claim $(\mathcal{X}_{f}, \mathcal{L}_{f})$ is a test configuration for $(X_{P}, L_{P})$. 
To see this, it is sufficient to show that $\mathcal{R}$ is the Rees algebra of a finitely generated $\mathbf{Z}$-filtration of $R$ (see Section 3.4). 
Let $\mathcal{F} = \{F^{\lambda}R_{m}\}$ be a $\mathbf{Z}$-filtration  of $R$ defined by 
\begin{align*}
F^{\lambda}R_{0} = 
\begin{cases}
\mathbf{C} & \lambda \leq 0, \\
\{0\}  & \lambda > 0
\end{cases}
\end{align*}
and 
\begin{align*}
F^{\lambda}R_{m} = \mathrm{Vect}_{\mathbf{C}}\{\chi^{\alpha}s^{m} \mid \alpha \in mP \cap M,\ \lambda \leq m(L-f(\alpha/m)) \} 
\end{align*}
for any $\lambda \in \mathbf{R}$ and $m \in \mathbf{Z}_{>0}$. 
Then the Rees algebra of $\mathcal{F}$ is given by
\begin{align*}
\mathrm{Rees}(\mathcal{F}) 
&= \bigoplus_{m=0}^{\infty}\bigoplus_{\lambda \in \mathbf{Z}}t^{-\lambda}F^{\lambda}R_{m} \\
&= \bigoplus_{m=0}^{\infty}\bigoplus_{\lambda \in \mathbf{Z}}\mathrm{Vect}_{\mathbf{C}}\left\{\chi^{\alpha}t^{-\lambda}s^{m} \left|\begin{array}{c}
\alpha \in mP \cap M, \\
\lambda \leq m(L-f(\alpha/m)) 
\end{array}\right. \right\} \\
&= \bigoplus_{m=0}^{\infty}\bigoplus_{\lambda \in \mathbf{Z}}\mathrm{Vect}_{\mathbf{C}}\left\{\chi^{\alpha}t^{\lambda}s^{m} \left|\begin{array}{c}
\alpha \in mP \cap M, \\
m(f(\alpha/m)-L) \leq \lambda 
\end{array}\right. \right\} \\
&= \bigoplus_{m=0}^{\infty}\mathcal{R}_{m} = \mathcal{R}. 
\end{align*}
Furthermore, $\mathcal{F}$ is finitely generated since $\mathcal{R}$ is a finitely generated $\mathbf{C}[t]$-algebra. 
This shows the claim. 
We call $(\mathcal{X}_{f}, \mathcal{L}_{f})$ the \emph{toric test configuration} associated to $f$. 
The induced $\mathbf{G}_{m}$-coaction $\mu_{m}$ on the central fiber $H^{0}((\mathcal{X}_{f})_{0}, (\mathcal{L}_{f}^{m})_{0}) = \mathrm{Gr}_{m}^{\mathcal{F}}\mathcal{R}$ is given by
\begin{align}\label{dual_action}
\mu_{m}(\chi^{\alpha}s^{m}) = t^{m(L-f(\alpha/m))} \otimes \chi^{\alpha}s^{m}. 
\end{align}

Easy computation shows that the canonical compactification of $(\mathcal{X}_{f}, \mathcal{L}_{f})$ is given as a polarized toric variety $(\mathcal{X}_{\mathcal{Q}}, \mathcal{L}_{\mathcal{Q}})$ corresponding to a rational convex polytope 
\begin{align}\label{compactification_polytope}
\mathcal{Q} = \{(x, y) \in M_{\mathbf{R}} \times \mathbf{R} \mid x \in P,\ f(x)-L \leq y \leq 0\}. 
\end{align}

\begin{example}[Translations and scalings: toric cases]\label{toric_tc_scaling_base_change}
Let $f$ be a rational piecewise affine convex function on an integral Delzant polytope $P$. 
Let us describe translations and scalings of $(\mathcal{X}_{f}, \mathcal{L}_{f})$ in terms of the convex function $f$. 
for simplicity, we assume $f$ is integral, i.e. $f$ is of the form \eqref{PLconv} with each $\ell_{j}$ an affine function having integral coefficients. 
\begin{enumerate}[(1)]
\item Let $c \in \mathbf{Z}$, and choose $L \in \mathbf{Z}_{>0}$ so that $\max\{f, f-c\} < L$ on $P$. 
Then the filtration $\mathcal{F}_{c}$ corresponding to the translation $(\mathcal{X}_{f}, (\mathcal{L}_{f})_{c})$ is given by 
\begin{align*}
F_{c}^{\lambda}R_{m} = F^{\lambda - mc}R_{m}. 
\end{align*}
Claerly we have $F_{c}^{\lambda}R_{0} = F^{\lambda}R_{0}$. 
If $m > 0$, then we have 
\begin{align*}
F_{c}^{\lambda}R_{m} 
&= \mathrm{Vect}_{\mathbf{C}}\{\chi^{\alpha}s^{m} \mid \alpha \in mP \cap M,\ \lambda - mc \leq m(L-f(\alpha/m)) \} \\
&= \mathrm{Vect}_{\mathbf{C}}\{\chi^{\alpha}s^{m} \mid \alpha \in mP \cap M,\ \lambda  \leq m(L-f(\alpha/m)+c) \}. 
\end{align*} 
Hence $\mathcal{F}_{c}$ coincides with the filtration corresponding to the toric test configuration $(\mathcal{X}_{f-c}, \mathcal{L}_{f-c})$. 
\item Let $d \in \mathbf{Z}_{>0}$. 
Then the filtration $\mathcal{F}_{d}$ corresponding to the scaling $((\mathcal{X}_{f})_{d}, (\mathcal{L}_{f})_{d})$ is given by 
\begin{align*}
F_{d}^{\lambda}R_{m} = F^{\lceil \frac{\lambda}{d} \rceil}R_{m}. 
\end{align*}
Claerly we have $F_{d}^{\lambda}R_{0} = F^{\lambda}R_{0}$. 
If $m > 0$, then we have 
\begin{align*}
F_{d}^{\lambda}R_{m} 
&= \mathrm{Vect}_{\mathbf{C}}\{\chi^{\alpha}s^{m} \mid \alpha \in mP \cap M,\ \lceil \lambda/d \rceil \leq m(L-f(\alpha/m)) \} \\
&= \mathrm{Vect}_{\mathbf{C}}\{\chi^{\alpha}s^{m} \mid \alpha \in mP \cap M,\ \lambda/d \leq m(L-f(\alpha/m)) \} \\
&= \mathrm{Vect}_{\mathbf{C}}\{\chi^{\alpha}s^{m} \mid \alpha \in mP \cap M,\ \lambda \leq m(dL-df(\alpha/m)) \}. 
\end{align*}
Hence $\mathcal{F}_{d}$ coincides with the filtration corresponding to the toric test configuration $(\mathcal{X}_{df}, \mathcal{L}_{df})$. 
\end{enumerate}
\end{example}

\subsection{Energy functionals on polarized toric manifolds}\label{sec:ToricEnergy}
By using Duistermaat-Heckman theorem and Legendre dual functions, we can write down energy functionals in terms of convex functions. 

\begin{lemma}
Let $\{u^{t}\} \subset \mathcal{S}$ be a smooth path, and $\{\phi^{t}\}$ be their Legendre duals. 
Then 
\begin{align*}
\dot{u}^{t}(x) = -\dot{\phi}^{t}(\nabla u^{t}(x))
\end{align*}
 for any $x \in P^{\circ}$. 
\end{lemma}
\begin{proof}
Since $\phi^{s}$ is a strictly convex function, $\nabla u^{s}(x)$ is a unique minimizer of $y \mapsto \phi^{s}(y) - \langle y, x\rangle$. 
By setting $y^{s} = \nabla u^{s}(x)$, we have 
\begin{align*}
0 
= \left. \frac{d}{ds}\right|_{s=t}(\phi^{t}(y^{s})-\langle y^{s}, x\rangle)
= \left. \frac{d}{ds}\right|_{s=t}\nabla u^{s}(x). 
\end{align*}
Hence we obtain 
\begin{align*}
\dot{u}^{t}(x) 
&= \left. \frac{d}{ds}\right|_{s=t}u^{s}(x) \\
&= \left. \frac{d}{ds}\right|_{s=t}(\langle y^{s}, x\rangle - \phi^{s}(y^{s})) \\
&= \left. \frac{d}{ds}\right|_{s=t}(\langle y^{s}, x\rangle - \phi^{t}(y^{s})) + \left. \frac{d}{ds}\right|_{s=t}(\phi^{t}(y^{s}) - \phi^{s}(y^{s})) \\
&= \langle \nabla \phi^{t},\dot{y}^{t} \rangle - \dot{\phi}^{t}(y^{t}) -\langle \nabla \phi^{t},\dot{y}^{t} \rangle \\
&= - \dot{\phi}^{t}(y^{t}), 
\end{align*}
as required. 
\end{proof}

\begin{proposition}\label{MA-energy_Legendre_dual}
Let $\varphi \in \mathcal{H}(X_{P}, \omega_{P})^{S}$, and $u \in \mathcal{S}$ be the symplectic potential of $\omega_{\varphi}$. 
Then 
\begin{align*}
E(\varphi) = -\dashint_{P}(u - u_{P})\,dx. 
\end{align*}
\end{proposition}
\begin{proof}
Let $\{u^{t}\}_{t \in [0, 1]} \subset \mathcal{S}$ be a smooth path so that $u^{0} = u_{P}$ and $u^{1} = u$. 
Let $\varphi^{t} \coloneqq \varphi_{u^{t}} \in \mathcal{H}(X_{P}, \omega_{P})^{S}$ for each $t \in [0, 1]$. 
Note that $\varphi^{0} = 0$ and $\varphi^{1} = \varphi$. 
Then we have 
\begin{align*}
E(\varphi) 
&= \int_{0}^{1}dt \dashint_{X_{P}}\dot{\varphi}^{t}\omega_{\varphi^{t}}^{n} \\
&= \int_{0}^{1}dt \dashint_{X_{P}^{\circ}}\dot{\phi}^{t}(dd^{c}\phi^{t})^{n} \\
&= -\int_{0}^{1}dt \dashint_{P}\dot{u}^{t}\,dx \\
&= -\dashint_{P}(u-u_{P})\,dx. 
\end{align*} 
\end{proof}
We define a subset $\mathcal{S}_{0}$ of $\mathcal{S}$ by 
\begin{align*}
\mathcal{S}_{0} \coloneqq \left\{ u \in \mathcal{S} \left| -\dashint_{P}(u-u_{P})\,dx = 0 \right. \right\}
\end{align*}

\begin{proposition}\label{K-energy_Legendre_dual}
Let $\varphi \in \mathcal{H}(X_{P}, \omega_{P})^{S}$ and $u \in \mathcal{S}$ be the symplectic potential of $\omega_{\varphi}$. 
\begin{enumerate}[(1)]
\item 
Let $\phi$ be the Legendere dual of $u$. 
Then 
\begin{align*}
J(\varphi) = \dashint_{X_{P}^{\circ}}(\phi-\phi_{P})\omega_{P}^{n} + \dashint_{P}(u - u_{P})\,dx. 
\end{align*} 
\item Let ${\displaystyle F(u) \coloneqq -\dashint_{P}\log\det(u_{ij})\,dx + \frac{1}{\mathrm{vol}(P)}\int_{\partial P}u\,d\sigma - \overline{s}\dashint_{P}u\,dx}$. 
Then 
\begin{align*}
M(\varphi) = F(u) -F(u_{P}). 
\end{align*} 
\end{enumerate}
\end{proposition}
\begin{proof}
For (1), since $\varphi = \phi - \phi_{P}$ on $X_{P}^{\circ}$ we have
\begin{align*}
\dashint_{X_{P}}\varphi\omega_{P}^{n} = \dashint_{X_{P}^{\circ}}(\phi-\phi_{P})\,\omega_{P}^{n}. 
\end{align*}
Hence we have 
\begin{align*}
J(\varphi) 
&= \dashint_{X_{P}}\varphi\omega_{P}^{n} - E(\varphi) \\
&= \dashint_{X_{P}^{\circ}}(\phi-\phi_{P})\omega_{P}^{n} + \dashint_{P}(u - u_{P})\,dx 
\end{align*}
by Proposition \ref{MA-energy_Legendre_dual}. 

Next consider (2). 
Let $\{\varphi^{t}\}_{t \in [0, 1]}$ be a smooth path in $\mathcal{H}_{0}^{S}$ satisfying $\varphi^{0} = 0$ and $\varphi^{1} = \varphi$, 
and $\{u^{t}\}_{t \in [0, 1]}$ be the coressponding symplectic potentials. 
Note that $u^{0} = u_{P}$ and $u^{1} = u$. 
Then 
\begin{align*}
\frac{d}{dt}M(\varphi^{t}) 
&= -\dashint_{X_{P}}\dot{\varphi}^{t}(s(\varphi^{t})-\overline{s})\omega_{\varphi^{t}}^{n} \\
&= -\dashint_{X_{P}}(-\dot{u}^{t})(-{(u^{t})^{ij}}_{ij}-\overline{s})\,dx \\
&= -\dashint_{X_{P}}(\dot{u}^{t}){(u^{t})^{ij}}_{ij}\,dx -\overline{s}\dashint_{P}\dot{u}^{t}\,dx. 
\end{align*}
By Theorem \ref{integration_by_parts}, we obtain 
\begin{align*}
-\dashint_{X_{P}}(\dot{u}^{t}){(u^{t})^{ij}}_{ij}\,dx 
&= \frac{1}{\mathrm{vol}(P)}\int_{\partial P}\dot{u}^{t}\,d\sigma - \dashint_{P}(u^{t})^{ij}(\dot{u}^{t})_{ij}\,dx \\
&= \frac{1}{\mathrm{vol}(P)}\int_{\partial P}\dot{u}^{t}\,d\sigma - \dashint_{P}\frac{d}{dt}\log\det((u^{t})_{ij})\,dx 
\end{align*}
and hence 
\begin{align*}
\frac{d}{dt}M(\varphi^{t}) 
&= \frac{1}{\mathrm{vol}(P)}\int_{\partial P}\dot{u}^{t}\,d\sigma - \dashint_{P}\frac{d}{dt}\log\det((u^{t})_{ij})\,dx - \overline{s}\dashint_{P}\dot{u}^{t}\,dx \\
&= \frac{d}{dt}\left(-\dashint_{P}\log\det((u^{t})_{ij})\,dx + \frac{1}{\mathrm{vol}(P)}\int_{\partial P}u^{t}\,d\sigma  - \overline{s}\dashint_{P}u^{t}\,dx\right), 
\end{align*}
as required. 
\end{proof}

\begin{proposition}\label{rel_K-energy_Legendre_dual}
Let $\varphi \in \mathcal{H}(X_{P}, \omega_{P})^{S}$, and $u \in \mathcal{S}$ be the symplectic potential of $\omega_{\varphi}$. 
\begin{enumerate}[(1)]
\item ${\displaystyle H_{V}(\varphi) = -\dashint_{P}(u-u_{P})V\,dx}$. 
\item Let ${\displaystyle F_{V}(u) \coloneqq F(u)+H_{V}(\varphi)}$. 
Then $M_{V}(\varphi) = F_{V}(u) -F_{V}(u_{P})$.  
\end{enumerate}
\end{proposition}
\begin{proof}
For (1), let $\{\varphi^{t}\}_{t \in [0, 1]}$ be a smooth path in $\mathcal{H}_{0}^{S}$ satisfying $\varphi^{0} = 0$ and $\varphi^{1} = \varphi$, and $\{u^{t}\}_{t \in [0, 1]}$ be the corresponding symplectic potentials. 
Note that $u^{0} = u_{P}$ and $u^{1} = u$. 
Then 
\begin{align*}
\frac{d}{dt}H_{V}(\varphi^{t}) 
&= \dashint_{X_{P}}\dot{\varphi}^{t}\theta_{\varphi^{t}}\omega_{\varphi^{t}}^{n} \\
&= \dashint_{X_{P}}\dot{\varphi}^{t}(-\left\langle \mu_{\varphi^{t}}, -\frac{1}{2}\theta \right\rangle)\omega_{\varphi^{t}}^{n} \\
&= -\dashint_{P}\dot{u}^{t}\frac{1}{2}\theta\,dx \\
&= \frac{d}{dt}\left( -\dashint_{P}u^{t}\frac{1}{2}\theta\,dx \right), 
\end{align*}
as required. 

The claim (2) is easily proved from (1) and Proposition \ref{K-energy_Legendre_dual}.
\end{proof}

As shown by Guan in \cite{Gua00}, any pair of K\"ahler potentials in $\mathcal{H}_{0}^{S}$ can be joined by a unique geodesic segment, which are obtained as a line segment of symplectic potentials. 
Thus, we have the following; 

\begin{proposition}
Let $\varphi \in \mathcal{H}_{0}^{S}$, and $u \in \mathcal{S}$ be the symplectic potential of $\omega_{\varphi}$. 
Then 
\begin{align*}
d_{1}(\varphi, 0) &= \dashint_{P}|u-u_{P}|\,dx. 
\end{align*}
\end{proposition}
\begin{proof}
Let $u^{t} = (1-t)u_{P} + tu$ ($t \in [0, 1]$), and $\varphi^{t} \coloneqq \varphi_{u^{t}}$. 
Then $\{\varphi^{t}\}_{t \in[0, 1]}$ is a unique geodesic segment joining $0$ and $\varphi$. 
Hence we have 
\begin{align*}
d_{1}(\varphi, 0) 
&= \ell_{1}(\{\varphi^{t}\}_{t \in[0, 1]}) 
= \int_{0}^{1}dt\dashint_{X}|\dot{\varphi}^{t}|\omega_{\varphi^{t}}^{n} \\ 
&= \int_{0}^{1}dt\dashint_{P}|-\dot{u}^{t}|\,dx 
= \dashint_{P}|u-u_{P}|\,dx. 
\end{align*}
\end{proof}

Consider the action of $T$ on $\mathcal{S}_{0}$. 
Let $u \in \mathcal{S}_{0}$ and $\tau \in T$. 
Then $(\varphi_{u})_{\tau} \in \mathcal{H}_{0}^{S}$ can be written as 
\begin{align*}
\tau^{*}\omega_{\varphi_{u}} = \omega_{P} + dd^{c}(\varphi_{u})_{\tau}. 
\end{align*}
On $X_{P}^{\circ} \cong T$, we can further express 
\begin{align*}
\tau^{*}\omega_{\varphi_{u}} = \tau^{*}dd^{c}u^{\vee} = dd^{c}(\tau^{*}u^{\vee}). 
\end{align*}
If we write $\tau = \exp(-\xi/2)$ for some $\xi \in N_{\mathbf{R}}$, we have 
\begin{align*}
\tau \cdot \exp(-y/2+2\pi\sqrt{-1}\theta) = \exp(-(y+\xi)/2+2\pi\sqrt{-1}\theta)
\end{align*}
and 
\begin{align*}
\tau^{*}u^{\vee}(y) 
&= u^{\vee}(y+\xi) \\
&= \sup_{x \in P^{\circ}}(\langle x, y+\xi \rangle - u) \\
&= (u-\xi)^{\vee}(y). 
\end{align*}
Therefore we obtain 
\begin{align*}
\tau^{*}\omega_{u} = dd^{c}(\tau^{*}u^{\vee}) = dd^{c}(u-\xi) = \omega_{u-\xi}
\end{align*}
and hence 
\begin{align*}
dd^{c}(\varphi_{u})_{\tau} 
= \tau^{*}\omega_{u} - \omega_{P}
= \omega_{u-\xi} - \omega_{P} 
= dd^{c}\varphi_{u-\xi}
\end{align*}
on $X_{P}^{\circ}$. 
By replacing the coordinate $(x_{1}, \ldots, x_{n})$ on $M_{\mathbf{R}}$ so that 
\begin{align*}
\int_{P}x_{i}\,dx = 0,\quad i=1, \ldots, n, 
\end{align*}
we obtain $(\varphi_{u})_{\tau} = \varphi_{u-\xi}$. 
This shows the following 
\begin{proposition}
Let $\varphi \in \mathcal{H}_{0}^{S}$, and $u \in \mathcal{S}$ be the symplectic potential of $\omega_{\varphi}$. 
Then 
\begin{align*}
d_{1, T}(\varphi, 0) &= \inf_{\xi \in N_{\mathbf{R}}}\dashint_{P}|(u-\xi)-u_{P}|\,dx. 
\end{align*}
\end{proposition}

We define the \emph{reduced $L^{1}$-norm} by 
\begin{align*}
\|u\|_{1, T} \coloneqq \inf_{\xi \in N_{\mathbf{R}}}\dashint_{P}|(u-\xi) - (\overline{u-\xi})|\,dx,  
\end{align*}
where
\begin{align*}
(\overline{u-\xi}) \coloneqq \dashint_{P}(u-\xi)\,dx.  
\end{align*}
Then 
\begin{corollary}
Let $\varphi \in \mathcal{H}_{0}^{S}$, and $u \in \mathcal{S}$ be the symplectic potential of $\omega_{\varphi}$. 
Then 
\begin{align*}
\|u\|_{1, T} - \|u_{P}\|_{1} \leq d_{1, T}(\varphi, 0) \leq \|u\|_{1, T} + \|u_{P}\|_{1}. 
\end{align*}
\end{corollary}

\begin{corollary}
For a polarized toric manifold $(X_{P}, L_{P})$, the followings are equivalent. 
\begin{enumerate}[(1)]
\item The relative K-energy is $T$-coercive. 
\item There exists $\delta, C >0$ such that 
\begin{align*}
F_{V}(u) \geq \delta \|u\|_{1, T} - C 
\end{align*}
for any $u \in \mathcal{S}_{0}$. 
\item There exists $\delta, C >0$ such that 
\begin{align*}
F_{V}(u) \geq \delta J_{T}(u) - C 
\end{align*}
for any $u \in \mathcal{S}_{0}$. 
\end{enumerate}
\end{corollary}

\subsection{Non-Archimedean functionals on polarized toric manifolds}\label{sec:NAToricEnergy}
In this subsection, we give convex analytic expression of non-Archimedean functionals for toric test configurations. 
Let $f \in \mathcal{C}_{PL}^{\mathbf{Q}}$. 
For simplicity we assume $f$ is integral. 
Choose $L \in \mathbf{Z}_{>0}$ so that $L - f > 0$ on $P$. 
Then the corresponding $(\mathcal{X}_{f}, \mathcal{L}_{f})$ is a toric test configuration of exponent 1. 

\begin{proposition}\label{ToricDH}
The Duistermaat-Heckmann measureof $(\mathcal{X}_{f}, \mathcal{L}_{f})$ s given by 
\begin{align*}
\mathrm{DH}(\mathcal{X}_{f}, \mathcal{L}_{f}) = (L-f)_{\#}\frac{dx}{\mathrm{vol}(P)}. 
\end{align*} 
In particular, $\mathrm{supp}(\mathrm{DH}(\mathcal{X}_{f}, \mathcal{L}_{f})) = [L-\max_{P}f, L-\min_{P}f]$. 
\end{proposition}
\begin{proof}
Let $\rho$ be a bounded continuous function on $\mathbf{R}$. 
Then we have 
\begin{align*}
\int_{\mathbf{R}}\rho\,d\mathrm{DH}(\mathcal{X}_{f}, \mathcal{L}_{f}) 
&= \lim_{m \to \infty}\frac{1}{E_{P}(m)}\sum_{\alpha \in mP \cap M}\rho\left(L-f\left(\frac{\alpha}{m}\right)\right) \\
&= \dashint_{P}\rho(L-f)\,dx. 
\end{align*} 
\end{proof}

\begin{proposition}\label{ToricNA_E}
For each $f \in \mathcal{C}_{PL}^{\mathbf{Q}}$ we have 
\begin{align*}
E^{NA}(\mathcal{X}_{f}, \mathcal{L}_{f}) = L -\dashint_{P}f\,dx. 
\end{align*} 
\end{proposition}
\begin{proof}
Let $f \in \mathcal{C}_{PL}^{\mathbf{Q}}$. 
Then we have 
\begin{align*}
E^{NA}(\mathcal{X}_{f}, \mathcal{L}_{f}) 
&= \int_{\mathbf{R}}\lambda\,d\mathrm{DH}(\mathcal{X}_{f}, \mathcal{L}_{f}) \\
&= \lim_{m \to \infty}\frac{1}{E_{P}(m)}\sum_{\alpha \in mP \cap M}\left(L-f\left(\frac{\alpha}{m}\right)\right) \\
&= \dashint_{P}(L-f)\,dx \\
&= L -\dashint_{P}f\,dx, 
\end{align*}
as required. 
\end{proof}

\begin{proposition}\label{ToricNA_J}
For each $f \in \mathcal{C}_{PL}^{\mathbf{Q}}$ we have 
\begin{align*}
&J^{NA}(\mathcal{X}_{f}, \mathcal{L}_{f}) = \|f\|_{J} \coloneqq \dashint_{P}f\,dx - \min_{P}f, \\
&J_{T}^{NA}(\mathcal{X}_{f}, \mathcal{L}_{f}) = \|f\|_{J} \coloneqq \inf_{\text{$\xi \colon$affine}}\left(\dashint_{P}(f+\xi)\,dx - \min_{P}(f+\xi)\right). 
\end{align*} 
\end{proposition}
\begin{proof}
The formula for $J^{NA}(\mathcal{X}_{f}, \mathcal{L}_{f})$ is a direct consequence of Propositions \ref{ToricDH} and \ref{ToricNA_E}. 
Let us consider $J_{T}^{NA}(\mathcal{X}_{f}, \mathcal{L}_{f})$. 
For the graded ring $R(X_{P}, L_{P})$, $R_{m}$ can be identified with 
\begin{align*}
R_{m} = H^{0}(X_{P}, L_{P}^{m}) \cong \mathrm{Vect}_{\mathbf{C}}\{\chi^{\alpha}t^{m} \mid \alpha \in mP \cap M\}.  
\end{align*}
The induced $T$-equivariant filtration $\mathcal{F}(\mathcal{X}_{f}, \mathcal{L}_{f})$ is then given by 
\begin{align*}
&F^{\lambda}R_{m} = \mathrm{Vect}_{\mathbf{C}}\{\chi^{\alpha}t^{m} \mid \alpha \in mP \cap M, m(f(\alpha/m)-L) \leq \lambda \}, \\
&(F^{\lambda}R_{m})_{\alpha} = \mathrm{Vect}_{\mathbf{C}}\{\chi^{\alpha}t^{m} \mid m(f(\alpha/m)-L) \leq \lambda \}. 
\end{align*}
Let $\xi \in N_{\mathbf{R}}$. 
If necessary, replace $L \in \mathbf{Z}_{>0}$ so that $L-f-\xi > 0$. 
Then the $\xi$-twist $\mathcal{F}_{\xi}$ of $\mathcal{F}(\mathcal{X}_{f}, \mathcal{L}_{f})$ is given by 
\begin{align*}
(F_{\xi}^{\lambda}R_{m})_{\alpha} 
&= (F^{\lambda - \langle \alpha, \xi \rangle}R)_{\alpha} \\
&= \mathrm{Vect}_{\mathbf{C}}\{\chi^{\alpha}t^{m} \mid m(f(\alpha/m)-L) \leq \lambda - \langle \alpha, \xi \rangle \} \\
&= \mathrm{Vect}_{\mathbf{C}}\{\chi^{\alpha}t^{m} \mid m(f(\alpha/m)+\xi(\alpha/m)-L) \leq \lambda\}, \\
F_{\xi}^{\lambda}R_{m} 
&= \bigoplus_{\alpha \in M}(F_{\xi}^{\lambda}R_{m})_{\alpha}. 
\end{align*}
Hence we have 
\begin{align*}
\nu_{m} 
&= -\frac{d}{d\lambda}\frac{1}{E_{P}(m)}\dim F_{\xi}^{m\lambda}H^{0}(X, L^{m}) \\
&= \frac{1}{E_{P}(m)}\sum_{\alpha \in mP \cap M}\delta_{L-f(\alpha/m)-\xi(\alpha/m)}
\end{align*}
and 
\begin{align*}
\int_{\mathbf{R}}\rho\,d\nu  
&= \lim_{m \to \infty}\frac{1}{E_{P}(m)}\sum_{\alpha \in mP \cap M}\rho(L-f(\alpha/m)-\xi(\alpha/m)) \\
&= \dashint_{P}\rho(L-f-\xi)\,dx 
\end{align*}
for any bounded continuous function $\rho$ on $\mathbf{R}$. 
Therefore,  for each Borel measurable set $A$ in $P$ we have 
\begin{align*}
\nu(A) = \frac{1}{\mathrm{vol}(P)}\mathrm{vol}(P)\{x \in P \mid L-f-\xi \in A\}
\end{align*}
and 
\begin{align*}
&\int_{\mathbf{R}}\lambda\,d\nu = L-\dashint_{P}(f+\xi)\,dx, \\
&\mathrm{supp}(\nu) = [L-\max_{P}(f+\xi), L-\min_{P}(f+\xi)]. 
\end{align*}
Hence, we obtain 
\begin{align*}
J^{NA}(\mathcal{F}_{\xi}) 
&= \sup \mathrm{supp}(\nu) - \int_{\mathbf{R}}\lambda\,\nu = \dashint_{P}(f+\xi)\,dx - \min_{P}(f+\xi), \\
J_{T}^{NA}(\mathcal{X}_{f}, \mathcal{L}_{f}) 
&= \inf_{\text{$\xi \colon$affine}}\left(\dashint_{P}(f+\xi)\,dx - \min_{P}(f+\xi)\right). 
\end{align*}
\end{proof}

\begin{proposition}\label{ToricNA_M}
For each $f \in \mathcal{C}_{PL}^{\mathbf{Q}}$ we have 
\begin{align*}
M^{NA}(\mathcal{X}_{f}, \mathcal{L}_{f}) 
&= \frac{1}{\mathrm{vol}(P)}L(f) \\
&\coloneqq \frac{1}{\mathrm{vol}(P)}\left(\int_{\partial P}f\,d\sigma - \overline{s}\int_{P}f\,dx\right). 
\end{align*} 
\end{proposition}
\begin{proof}
Let $f \in \mathcal{C}_{PL}^{\mathbf{Q}}$. 
First we assume the central fiber of $(\mathcal{X}_{f}, \mathcal{L}_{f})$ is generically reduced. 
Then $M^{NA}(\mathcal{X}_{f}, \mathcal{L}_{f})$ coincides with $DF(\mathcal{X}_{f}, \mathcal{L}_{f})$. 
Further, by \eqref{dual_action} and \cite[Lemma 3.3]{ZZ08} we have 
\begin{align*}
w_{m} 
&= \sum_{\alpha \in mP \cap M}m\left(L-f\left(\frac{\alpha}{m}\right)\right) \\
&= m\left(m^{n}\int_{P}(L-f)\,dx + \frac{m^{n-1}}{2}\int_{\partial P}(L-f)\,d\sigma +O(m^{n-2})\right) \\
&= m^{n+1}\int_{P}(L-f)\,dx + \frac{m^{n}}{2}\int_{\partial P}(L-f)\,d\sigma +O(m^{n-1}) 
\end{align*}
for any $m \in \mathbf{Z}_{>0}$. 
By noting $\mathrm{vol}(P) = (L^{n})/n!$ and $\overline{s} = \sigma(\partial P)/\mathrm{vol}(P)$, we obtain 
\begin{align*}
M^{NA}(\mathcal{X}_{f}, \mathcal{L}_{f}) 
&= DF(\mathcal{X}_{f}, \mathcal{L}_{f}) \\
&= \frac{n!}{(L^{n})}\left(-\int_{\partial P}(L-f)\,d\sigma+\overline{s}\int_{P}(L-f)\,dx \right) \\
&= \frac{1}{\mathrm{vol}(P)}\left(\int_{\partial P}f\,d\sigma-\overline{s}\int_{P}f\,dx \right). 
\end{align*}

For general cases, we consider $df$ for $d \in \mathbf{Z}_{>0}$. 
This corresponds to replacing $(\mathcal{X}_{f}, \mathcal{L}_{f})$ with its base change $(\mathcal{X}_{df}, \mathcal{L}_{df})$ (see Example \ref{toric_tc_scaling_base_change} (2)). 
Then, by Proposition \ref{equiv_NA_K_DF} there exists $d_{0} \in \mathbf{Z}_{>0}$ such that $DF(\mathcal{X}_{df}, \mathcal{L}_{df}) = M^{NA}(\mathcal{X}_{df}, \mathcal{L}_{df}) = dM^{NA}(\mathcal{X}_{f}, \mathcal{L}_{f})$ for all $d \in \mathbf{Z}_{>0}$ divisible by $d_{0}$. 
Further, by Proposition \ref{homogenuity_NA_K-energy} we have 
\begin{align*}
M^{NA}(\mathcal{X}_{f}, \mathcal{L}_{f}) 
&= \frac{1}{d}M^{NA}(\mathcal{X}_{df}, \mathcal{L}_{df}) \\
&= \frac{1}{d}\left[\frac{1}{\mathrm{vol}(P)}\left(\int_{\partial P}df\,d\sigma-\overline{s}\int_{P}df\,dx \right)\right] \\
&= \frac{1}{\mathrm{vol}(P)}\left(\int_{\partial P}f\,d\sigma-\overline{s}\int_{P}f\,dx \right), 
\end{align*}
as required. 
\end{proof}

\begin{proposition}\label{ToricNA_MV}
For each $f \in \mathcal{C}_{PL}^{\mathbf{Q}}$ we have 
\begin{align*}
H_{V}^{NA}(\mathcal{X}_{f}, \mathcal{L}_{f}) 
&= -\dashint_{P}fV\,dx, \\
M_{V}^{NA}(\mathcal{X}_{f}, \mathcal{L}_{f}) 
&\coloneqq \frac{1}{\mathrm{vol}(P)}\left(\int_{\partial P}f\,d\sigma - \int_{P}(\overline{s} + V)f\,dx\right). 
\end{align*}
\end{proposition}
\begin{proof}
It is enough to show the formula for $H_{V}^{NA}$. 
Since $\int_{P}V\,dx = 0$, we have 
\begin{align*}
H_{V}^{NA}(\mathcal{X}_{f}, \mathcal{L}_{f}) 
&= \lim_{m \to \infty}\frac{1}{m^{2}}\left\{\frac{1}{E_{P}(m)}\sum_{\alpha \in mP \cap M}m\left(L-f\left(\frac{\alpha}{m}\right)\right)V(\alpha) \right. \\
&\quad \left. -\frac{1}{(E_{P}(m))^{2}}\sum_{\alpha \in mP \cap M}m\left(L-f\left(\frac{\alpha}{m}\right)\right)\sum_{\alpha \in mP \cap M}V(\alpha) \right\} \\
&= \lim_{m \to \infty}\frac{1}{E_{P}(m)}\sum_{\alpha \in mP \cap M}\left(L-f\left(\frac{\alpha}{m}\right)\right)V\left(\frac{\alpha}{m}\right) \\
&= \frac{1}{\mathrm{vol}(P)}\int_{P}(L-f)V\,dx \\
&= -\dashint_{P}fV\,dx, 
\end{align*}
as required. 
\end{proof}

By Proposition \ref{ToricNA_MV}, we have the following 

\begin{corollary}
Let $(X_{P}, L_{P})$ be a polarized toric manifold. 
If $(X_{P}, L_{P})$ is uniformly relatively K-polystable, then there exists $\delta > 0$ such that 
\begin{align}\label{Toric_relK}
L_{V}(f) \geq \delta \|f\|_{J} 
\end{align}
for any $f \in \mathcal{C}_{PL}^{\mathbf{Q}}$. 
\end{corollary}

Our main theorem guarantees that the converse is also true, that is, a polarized toric manifold $(X_{P}, L_{P})$ satisfying \eqref{Toric_relK} for any $f \in \mathcal{C}_{PL}^{\mathbf{Q}}$ is uniformly relatively K-polystable. 

\section{Proof of Theorem \ref{equiv_strengthenings}}
\subsection{Spaces of convex functions}\label{sec:CvxFct}
Let $P \subset M_{\mathbf{R}}$ be an $n$-dimensional integral Delzant polytope defined by 
\begin{align*}
P = \{x \in M_{\mathbf{R}} \mid \langle \lambda_{j}, x \rangle + d_{j} \geq 0\ (j=1, \ldots, r)\}, 
\end{align*}
where $\lambda_{j} \in N$, $d_{j} \in \mathbf{Z}$, $r$ is the number of the facets, and each $\lambda_{j}$ is primitive. 
Throughout this section, we assume the origin 0 lies in the interior of $P$. 
Then $d_{j} > 0$ for any $j = 1, \ldots, r$. 
Let $\mathrm{CVX}(P)$ denote the set of all lower semicontinuous convex functions $f \colon P \to (-\infty, +\infty]$ such that $f \not \equiv+\infty$. 
Note that every function in $\mathrm{CVX}(P)$ is bounded from below since $P$ is compact. 
For each subset $\mathcal{F}$ of $\mathrm{CVX}(P)$, set 
\begin{align*}
\widetilde{\mathcal{F}} \coloneqq \{f \in \mathcal{F} \mid \inf_{P}f = f(0) = 0\}. 
\end{align*}
If $\mathcal{F}$ is closed under addition of affine functions, then for any function $f \in \mathcal{F}$ there exists an affine function $\ell$ such that $f + \ell \in \widetilde{\mathcal{F}}$ by the supporting hyperplane theorem. 

Let $\mathcal{E}_{1} \coloneqq \mathrm{CVX}(P) \cap L^{1}(P)$. 
By the toric pluripotential theory \cite[Proposition 3.2, Theorem 3.6, Proposition 3.9]{CGSZ19}, $\mathcal{E}_{1}$ can be identified with the metric completion of the space $(\mathcal{S}, d_{1})$. 
For each $j = 1, \ldots, r$, let $F_{j}$ denote the facet of $P$ defined by 
\begin{align*}
F_{j} = \{x \in P \mid \langle \lambda_{j}, x \rangle + d_{j} = 0\},  
\end{align*}
and define a subset $P^{\ast}$ of $P$ by 
\begin{align*}
P^{\ast} = P^{\circ} \cup \left(\bigcup_{j=1}^{r}F_{j}^{\circ}\right). 
\end{align*}
Here $F_{j}^{\circ}$ denotes the relative interior of $F_{j}$. 
Let $\mathcal{C}_{\ast}$ be the set of all elements of $\mathrm{CVX}(P)$ which are continuous on $P^{\ast}$ and integrable on $\partial P$. 
Note that the integral of any function in $\mathcal{C}_{\ast}$ over the boundary makes sense since the measure $\sigma$ is supported on the facets of $P$. 
A relation between $\mathcal{C}_{\ast}$ and $\mathcal{E}_{1}$ is given by the following proposition. 
\begin{proposition}\label{E1_bdry_int}
$\mathcal{C}_{\ast} = \mathcal{E}_{1} \cap L^{1}(\partial P)$. 
\end{proposition}

We explain a proof of Proposition \ref{E1_bdry_int} below. 
The inclusion $\mathcal{C}_{\ast} \subset \mathcal{E}_{1} \cap L^{1}(\partial P)$ is obtained from the following proposition. 

\begin{proposition}\label{integral_bound_C*}
Let $d = \max\{d_{1}, \ldots, d_{r}\}$. 
If $f$ is a nonnegative function in $\mathcal{C}_{\ast}$, then 
\begin{align*}
\int_{P}f\,dx \leq \frac{d}{n+1}\left(\frac{\sigma(\partial P)}{n}f(0) + \int_{\partial P}f\,d\sigma\right). 
\end{align*}
In particular, for any $f \in \widetilde{\mathcal{C}}_{\ast}$ we have 
\begin{align*}
\int_{P}f\,dx \leq \frac{d}{n+1}\int_{\partial P}f\,d\sigma. 
\end{align*}
\end{proposition}
\begin{proof}
For each $j=1, \ldots, r$, let $C(F_{j}) = \{t\zeta \mid t \in [0, 1],\ \zeta \in F_{j}\}$. 
Then we have 
\begin{align*}
P = \bigcup_{j=1}^{r}C(F_{j}). 
\end{align*}
Also, by convexity of $f$ we have 
\begin{align*}
\int_{C(F_{j})}f\,dx 
&\leq d_{j}\int_{0}^{1}t^{n-1}\,dt\int_{F_{j}}((1-t)f(0) + tf)\,d\sigma \\
&= d_{j}\int_{0}^{1}(1-t)t^{n-1}\,dt\int_{F_{j}}f(0)\,d\sigma 
+ d_{j}\int_{0}^{1}t^{n}\,dt\int_{F_{j}}f\,d\sigma \\
&= \frac{d_{j}}{n+1}\left(\frac{\sigma(F_{j})}{n}f(0) + \int_{F_{j}}f\,d\sigma\right). 
\end{align*}
Hence we obtain 
\begin{align*}
\int_{P}f\,dx 
&= \sum_{j=1}^{r}\int_{C(F_{j})}f\,dx \\
&\leq \sum_{j=1}^{r}\frac{d_{j}}{n+1}\left(\frac{\sigma(F_{j})}{n}f(0) + \int_{F_{j}}f\,d\sigma\right) \\
&\leq \frac{d}{n+1}\left(\frac{\sigma(\partial P)}{n}f(0) + \int_{\partial P}f\,d\sigma\right). 
\end{align*}
The latter claim is now obvious. 
\end{proof}
\begin{corollary}\label{integrable_C*}
Every function in $\mathcal{C}_{\ast}$ is integrable on $P$. 
In particular, we have $\mathcal{C}_{\ast} \subset \mathcal{E}_{1} \cap L^{1}(\partial P)$. 
\end{corollary}
\begin{proof}
Let $f \in \mathcal{C}_{\ast}$. 
Then there is an affine function $\ell$ so that $\tilde{f} \coloneqq f + \ell \in \widetilde{\mathcal{C}}_{*}$. 
By Proposition \ref{integral_bound_C*}, we have 
\begin{align*}
\int_{P}\tilde{f}\,dx \leq \frac{d}{n+1}\int_{\partial P}\tilde{f}\,d\sigma < \infty
\end{align*}
and hence $\tilde{f}$ is integrable on $P$. 
Since $\ell$ is integrable on $P$, $f = \tilde{f} - \ell$ is also integrable on $P$. 
\end{proof}

Let us show the converse inclusion. 
For the purpose, we use the following facts from the Donldson's work \cite{Don02}. 
Let $\Omega$ be a bounded convex open set of $\mathbf{R}^{n}$, and $f \colon \Omega \to \mathbf{R}$ be a convex function. 
For each $x \in \Omega$, set 
\begin{align*}
D_{x}(f) \coloneqq \sup\{|\lambda| \mid f(y) \geq \langle \lambda, x\rangle + f(y)\ (y \in \Omega)\}. 
\end{align*}
\begin{proposition}[{\cite[Lemma 5.2.3]{Don02}}]\label{upper_bdd_subdiff}
There is a constant $\kappa > 0$ such that if a nonnegative convex function $f$ on $\Omega$ is integrable, then 
\begin{align*}
D_{x}(f) \leq \kappa d_{x}^{-(n+1)}\int_{\Omega}f\,dy
\end{align*}
for any $x \in \Omega$. 
Here $d_{x}$ is the distance from $x$ to the boundary of $\Omega$. 
\end{proposition}
\begin{proposition}[{\cite[Lemma 5.2.4]{Don02}}]\label{Lip_bdd_subdiff}
For any convex function $f \colon \Omega \to \mathbf{R}$ and any two points $x, y \in P$, 
\begin{align*}
|f(x) - f(y)| \leq \max\{D_{x}(f), D_{y}(f)\}|x-y|. 
\end{align*}
\end{proposition}

Let $f \in \mathcal{E}_{1} \cap L^{1}(\partial P)$. 
By adding a suitable affine function, we may assume $f \in \widetilde{\mathcal{E}_{1}}$. 
Since $f$ is integrable over $P$, $f < +\infty$ on $P^{\circ}$. 
Continuity of $f$ on $P^{\circ}$ is obvious by the convexity. 
Let $F \in \{F_{1}, \ldots, F_{r}\}$. 

\begin{lemma}\label{Rad_lim}
For each $\zeta \in F$ we have ${\displaystyle f(\zeta) = \lim_{t \to 1}f(t\zeta)}$. 
\end{lemma}
\begin{proof}
Let $\zeta \in F$ and $s, t \in [0, 1]$. 
If $s < t$, then we have 
\begin{align*}
\frac{f(t\zeta) - f(s\zeta)}{t-s} \geq \frac{f(t\zeta) - f(0)}{t - 0} \geq 0 
\end{align*}
by the convexity of $f$. 
Hence the function $[0, 1] \ni t \mapsto f(t\zeta) \in \mathbf{R}$ is monotonically non-decreasing. 
Further, since 
\begin{align*}
f(t\zeta) \leq (1-t)f(0) + tf(\zeta) = tf(\zeta)
\end{align*}
for any $t \in [0, 1]$, we obtain 
\begin{align*}
\lim_{t \to 1}f(t\zeta) \leq f(\zeta) \leq \liminf_{\xi \to \zeta}f(\xi)
\end{align*}
by the lower semicontinuity of $f$. 
We claim that $\liminf_{\xi \to \zeta}f(\xi) \leq \lim_{t \to 1}f(t\zeta)$. 
Let $\delta > 0$ and $t \in [0, 1]$, and suppose $|t-1|< \delta/|\zeta|$. 
Then $t\zeta \in B(\zeta, \delta)$ and 
\begin{align*}
\inf_{B(\zeta, \delta) \cap P}f \leq f(t\zeta). 
\end{align*}
Hence we have 
\begin{align*}
\inf_{B(\zeta, \delta) \cap P}f \leq \lim_{t \to 1}f(t\zeta) 
\end{align*}
and 
\begin{align*}
\liminf_{\xi \to \zeta}f(\xi) = \sup_{\delta > 0}\inf_{B(\zeta, \delta) \cap P}f \leq \lim_{t \to 1}f(t\zeta), 
\end{align*}
as required. 
\end{proof}

For each $\eta \in (0, 1]$, define a Borel measure $\sigma_{\eta}$ on $\eta(\partial P)$ by 
\begin{align*}
\sigma_{\eta}(A) \coloneqq \eta^{n-1}\sigma(\eta^{-1}A). 
\end{align*}
By setting $C \coloneqq \int_{\partial P}f(\xi)\,d\sigma$, we have 
\begin{align*}
\int_{\eta(\partial P)}f(\xi)\,d\sigma_{\eta} = \eta^{n-1}\int_{\partial P}f(\eta\xi)\,d\sigma \leq \eta^{n}\int_{\partial P}f(\xi)\,d\sigma = \eta^{n}C. 
\end{align*}
Let $K$ be a compact subset of $F^{\circ}$. 
We write $F$ as 
\begin{align*}
F = \{x \in P \mid \langle \lambda, x \rangle + d = 0\}. 
\end{align*}
Regarding $f$ as a nonnegative convex function on $\eta F^{\circ}$, by Proposition \ref{upper_bdd_subdiff} there is a constant $\kappa > 0$ such that 
\begin{align*}
D_{\zeta}(f) \leq |\lambda|\kappa d_{\zeta}^{-n}\int_{\eta F}f(\xi)\,d\sigma_{\eta} \leq |\lambda|\kappa d_{\zeta}^{-n}\eta^{n}C 
\end{align*}
for any $\zeta \in \eta K$. 
Here $d_{\zeta}$ is the distance from $\zeta$ to the boundary $\partial(\eta F^{\circ})$. 
Also, by setting $d_{\eta K}$ with the distance between $\eta K$ and $\partial(\eta F^{\circ})$, we have 
\begin{align*}
D_{\zeta}(f) \leq  |\lambda|\kappa d_{\zeta}^{-n}\eta^{n}C \leq |\lambda|\kappa d_{\eta K}^{-n}\eta^{n}C = |\lambda|\kappa d_{K}^{-n}C. 
\end{align*}
Hence, by Proposition \ref{Lip_bdd_subdiff} we obtain 
\begin{align*}
|f(\zeta') - f(\zeta'')| \leq \max\{D_{\zeta'}(f), D_{\zeta''}(f)\}|\zeta'-\zeta''| \leq  |\lambda|\kappa d_{K}^{-n}C|\zeta'-\zeta''| 
\end{align*}
for any $\zeta', \zeta'' \in \eta K$. 
Let $\zeta \in F^{\circ}$, and 
\begin{align*}
K \coloneqq \overline{B(\zeta, \delta|\zeta|)} \cap \partial P 
\end{align*}
for each $\delta > 0$. 
By choosing $\delta$ to be sufficiently small, we may assume that $K \subset F^{\circ}$. 
Set $C(K) \coloneqq \{t\xi \mid t \geq 0, \xi \in K\}$. 
Then $C(K)$ is a closed neighborhood of $\zeta$ in $M_{\mathbf{R}}$. 
Let $x \in C(K) \cap P^{\circ}$, and $H_{x}$ be the hyperplane of $M_{\mathbf{R}}$ which contains $x$ and is parallel to the hyperplane containing $F$. 
Then $H_{x}$ intersects the line $\mathbf{R}\zeta$ at a single point $\{y\}$. 
Define $t(x) \in [0, 1]$ by $y = t(x)\zeta$. 
Then the function $x \mapsto t(x)$ is continuous and $t(x) \to 1$ (as $x \to \zeta$). 
Also, since $x, y \in t(x)K$ we have 
\begin{align*}
|f(x) - f(y)| \leq |\lambda|\kappa d_{K}^{-n}C|x - y|. 
\end{align*}
Hence we obtain 
\begin{align*}
|f(\zeta) - f(x)| 
&\leq |f(\zeta) - f(y)| + |f(y) - f(x)| \\
&\leq |f(\zeta) - f(t(x)\zeta)| + |\lambda|\kappa d_{K}^{-n}C|y - x| \\
&\leq |f(\zeta) - f(t(x)\zeta)| + |\lambda|\kappa d_{K}^{-n}C|t(x)\zeta - x| \\
&\to 0\quad \text{(as $x \to \zeta$)} 
\end{align*}
by Lemma \ref{Rad_lim}. 
This shows that $f$ is continuous at $\zeta$, and completes the proof of Proposition \ref{E1_bdry_int}. 

\subsection{Approximations and compactness of convex functions}
We collect various approximation and compactness results for convex functions, which is crucial for our proof of Theorem \ref{equiv_strengthenings}. 

\begin{proposition}\label{convergence_L1_unif} 
Let $\{f_{i}\}_{i=1}^{\infty}$ be a sequence of nonnegative functions in $\mathcal{C}_{*}$, and $f \in \mathcal{C}_{*}$. 
Suppose that $\{f_{i}\}_{i=1}^{\infty}$ and $f$ satisfy the following conditions: 
\begin{enumerate}[\upshape(i)]
\item $\{f_{i}\}_{i=1}^{\infty}$ converges locally uniformly to $f$ on $P^{\circ}$. 
\item ${\displaystyle \sup_{i \in \mathbf{Z}_{> 0}}\int_{\partial P}f_{i}\,d\sigma < \infty}$. 
\end{enumerate}
Then ${\displaystyle \lim_{i \to \infty}\|f - f_{i}\|_{L^{1}(P)} = 0}$. 
\end{proposition}
\begin{proof}
By assumption, there is a constant $C > 0$ such that 
\begin{align}\label{lem_1}
\begin{split}
f(0) + \sup_{i \in \mathbf{Z}_{>0}}f_{i}(0) + \int_{\partial P}f\,d\sigma + \sup_{i \in \mathbf{Z}_{>0}}\int_{\partial P}f_{i}\,d\sigma \leq C. 
\end{split}
\end{align}
Let $\eta \in (0, 1)$. 
Then for each $i \in \mathbf{Z}_{>0}$ we have 
\begin{align*}
\|f-f_{i}\|_{L^{1}(P)} 
= \int_{P \setminus \eta P}|f-f_{i}|\,dx + \int_{\eta P}|f-f_{i}|\,dx. 
\end{align*}
By the convexity of $f_{i}$ and $f$, we have 
\begin{align*}
\int_{P \setminus \eta P}|f-f_{i}|\,dx 
&\leq \int_{P \setminus \eta P}(f+f_{i})\,dx \\
&\leq \sum_{j=1}^{r}d_{j}\int_{\eta}^{1}(1-t)t^{n-1}\,dt\int_{F_{j}}(f(0) + f_{i}(0))\,d\sigma \\
&\quad + \sum_{j=1}^{r}d_{j}\int_{\eta}^{1}t^{n}\,dt\int_{F_{j}}(f+ f_{i})\,d\sigma \\
&\leq 2Cd\sigma(\partial P)\int_{\eta}^{1}(1-t)t^{n-1}\,dt + 2Cd\int_{\eta}^{1}t^{n}\,dt \\
&\leq 2Cd(\sigma(\partial P) + 1)\int_{\eta}^{1}t^{n-1}\,dt \\
&= 2Cd(\sigma(\partial P) + 1)\frac{1-\eta^{n}}{n}. 
\end{align*}
For each $\varepsilon > 0$, choose $\eta \in (0, 1)$ so that 
\begin{align*}
2Cd(\sigma(\partial P) + 1)\frac{1-\eta^{n}}{n} < \frac{\varepsilon}{2}. 
\end{align*}
Since $\{f_{i}\}_{i=1}^{\infty}$ converges uniformly to $f$ on $\eta P$, there exists $N \in \mathbf{Z}_{>0}$ such that for any integer $i \geq N$ 
\begin{align*}
\int_{\eta P}|f-f_{i}|\,dx < \frac{\varepsilon}{2}. 
\end{align*}
Hence, for any integer $i \geq N$ we have 
\begin{align*}
\|f-f_{i}\|_{L^{1}(P)}
&= \int_{P \setminus \eta P}|f-f_{i}|\,dx + \int_{\eta P}|f-f_{i}|\,dx < \varepsilon, 
\end{align*}
as required. 
\end{proof}

By using Proposition \ref{convergence_L1_unif}, we obtain the following improvements of \cite[Lemma 3.1]{CLS14} and \cite[Proposition 5.2.6]{Don02}. 

\begin{proposition}[{cf. \cite[Lemma 3.1]{CLS14}}]\label{appr_smooth}
Let $f \in \widetilde{\mathcal{C}}_{\ast}$. 
Then there exists a sequence $\{f_{i}\}_{i=1}^{\infty}$ in $\mathcal{C}_{\infty} \cap C^{\infty}(P)$ such that 
\begin{enumerate}[\upshape(i)]
\item $0 \leq f_{i}$ for any $i \in \mathbf{Z}_{>0}$, 
\item $\{f_{i}\}_{i =1}^{\infty}$ converges locally uniformly to $f$ on $P^{\circ}$, and 
\item ${\displaystyle  \int_{P}|f - f_{i}|\,dx + \int_{\partial P}|f - f_{i}|\,d\sigma \to 0}$ (as $i \to \infty$). 
\end{enumerate}
In particular, we have ${\displaystyle \lim_{i \to \infty}L_{V}(f_{i}) = L_{V}(f)}$. 
\end{proposition}
\begin{proof}
By \cite[Lemma 3.1]{CLS14}, there exists a sequence $\{f_{i}\}_{i=1}^{\infty}$ in $\mathcal{C}_{\infty}$ which satisfies (i), (ii), and 
\begin{align}\label{bdrconv_L1}
\lim_{i \to \infty}\int_{\partial P}|f - f_{i}|\,d\sigma = 0. 
\end{align}
Moreover, a careful reading of the proof in \cite{CLS14} shows that we can choose $\{f_{i}\}_{i=1}^{\infty}$ so that $f_{i} \in \mathcal{C}_{\infty} \cap C^{\infty}(P)$ for any $i \in \mathbf{Z}_{>0}$. 
Then the sequence $\{f_{i}\}_{i=1}^{\infty}$ satisfies the conditions (i), (ii) of Proposition \ref{convergence_L1_unif}, and hence it converges to $f$ in $L^{1}$. 
\end{proof}

\begin{proposition}[{cf. \cite[Proposition 5.2.6]{Don02}}]\label{cptness}
Let $\{f_{i}\}_{i =1}^{\infty}$ be a sequence in $\widetilde{\mathcal{C}}_{*}$ with 
\begin{align*}
\sup_{i \in \mathbf{Z}_{>0}}\int_{\partial P}f_{i}\,d\sigma < \infty. 
\end{align*}
Then there exists a subsequence $\{f_{i_{k}}\}_{k =1}^{\infty}$ of $\{f_{i}\}_{i =1}^{\infty}$ and $f \in \widetilde{\mathcal{C}}_{*}$ such that 
\begin{enumerate}[\upshape(i)]
\item $\{f_{i_{k}}\}_{k =1}^{\infty}$ converges locally uniformly to $f$ on $P^{\circ}$, and 
\item ${\displaystyle \int_{P}|f - f_{i_{k}}|\,dx \to 0\ (\text{as $k \to \infty$})}$, \\
${\displaystyle \int_{\partial P}f\,d\sigma \leq \liminf_{k \to \infty}\int_{\partial P}f_{i_{k}}\,d\sigma}$. 
\end{enumerate}
In particular, we have ${\displaystyle L_{V}(f) \leq \liminf_{k \to \infty}L_{V}(f_{i_{k}})}$. 
\end{proposition}
\begin{proof}
By \cite[Proposition 5.2.6]{Don02}, there exist a subsequence $\{f_{i_{k}}\}_{k =1}^{\infty}$ of $\{f_{i}\}_{i =1}^{\infty}$ and $f \in \widetilde{\mathcal{C}}_{*}$ which satisfy (i) and 
\begin{align*}
\int_{\partial P}f\,d\sigma \leq \liminf_{k \to \infty}\int_{\partial P}f_{i_{k}}\,d\sigma. 
\end{align*}
Then Proposition \ref{convergence_L1_unif} shows that the sequence $\{f_{i_{k}}\}_{k=1}^{\infty}$ converges to $f$ in $L^{1}$. 
The proposition is proved. 
\end{proof}

We also use the following approximation result due to Donaldson \cite{Don02}. 

\begin{proposition}[{\cite[Proposition 5.2.8]{Don02}}]\label{appr_PL}
Let $f$ be a nonnegative function in $\mathcal{C}_{\ast}$. 
Then there exists a sequence $\{f_{i}\}_{i=1}^{\infty}$ in $\mathcal{C}_{PL}$ such that 
\begin{enumerate}[\upshape(i)]
\item $0 \leq f_{i} \leq f$ for any $i \in \mathbf{Z}_{>0}$, 
\item $\{f_{i}\}_{i =1}^{\infty}$ converges locally uniformly to $f$ on $P^{\circ}$, and 
\item ${\displaystyle \int_{P}|f-f_{i}|\,dx + \int_{\partial P}|f-f_{i}|\,d\sigma \to 0}$ (\text{as $i \to \infty$}). 
\end{enumerate}
In particular, we have ${\displaystyle \lim_{i \to \infty}L_{V}(f_{i}) = L_{V}(f)}$. 
\end{proposition}

The following proposition is easily proved from the density of $\mathbf{Q}$ in $\mathbf{R}$ and convexity. 

\begin{proposition}\label{apprQPL}
Let $f \in \mathcal{C}_{PL}$ (possibly irrational). 
Then there exists a sequence $\{f_{i}\}_{i=1}^{\infty}$ in $\mathcal{C}_{PL}^{\mathbf{Q}}$ which converges uniformly to $f$ in $P$. 
\end{proposition}
\begin{proof}
Choose affine functions $\ell_{1}, \ldots, \ell_{r}$ so that 
\begin{align*}
f(x) = \max\{\ell_{1}(x), \ldots, \ell_{m}(x)\} 
\end{align*}
for any $x \in P$, and we regard $f$ as a function defined on $\mathbf{R}^{n}$. 
For any $j\in\{1, ,\ldots, m\}$, choose a sequence $\{\ell_{j}^{(i)}\}_{i=1}^{\infty}$ of rational affine functions which converges pointwise to $\ell_{j}$ on $\mathbf{R}^{n}$, and define a rational piecewise affine function $f^{(i)}$ by 
\begin{align*}
f^{(i)}(x) = \max\{\ell_{1}^{(i)}(x), \ldots, \ell_{m}^{(i)}(x)\}. 
\end{align*}
Then $\{f^{(i)}\}_{i=1}^{\infty}$ converges locally uniformly to $f$ on $\mathbf{R}^{n}$. 
In particular, $\{f^{(i)}\}_{i=1}^{\infty}$ converges uniformly to $f$ on $P$. 
\end{proof}

\subsection{The non-Archimedean K-energy}\label{sec:NAK}
We can define $L(f)$ and $L_{V}(f)$ for any $f \in \mathcal{E}_{1}$, taking value in $(-\infty, \infty]$.  
By using these, we can get the following characterization of $\mathcal{C}_{\ast}$. 

\begin{proposition}\label{C*_NA_K_energy}
\begin{align*}
\mathcal{C}_{\ast} 
= \{f \in \mathcal{E}_{1} \mid L(f) < \infty\} 
= \{f \in \mathcal{E}_{1} \mid L_{V}(f) < \infty\}. 
\end{align*}
\end{proposition}
\begin{proof}
By Proposition \ref{E1_bdry_int}, for each $f \in \mathcal{E}_{1}$ we have 
\begin{align*}
f \in \mathcal{C}_{\ast} 
\iff \int_{\partial P}f\,d\sigma < \infty 
\iff L(f) < \infty
\iff L_{V}(f) < \infty, 
\end{align*}
as required. 
\end{proof}

\subsection{The J-norm}\label{sec:RedJ}
For any convex function $f$ in $\mathcal{E}_{1}$, define 
\begin{align*}
\|f\|_{J} &= \inf_{\text{$\xi \colon$affine}}\left(\dashint_{P}(f+\xi)\,dx - \inf_{P}(f+\xi)\right). 
\end{align*}
We call $\|f\|_{J}$ the \emph{J-norm} of $f$. 
The followings are the basic properties on the J-norm. 

\begin{proposition}\label{red_J_estimate}
\begin{enumerate}[\upshape(1)]
\item $\|f + \xi\|_{J} = \|f\|_{J}$ for any $f \in \mathcal{E}_{1}$ and $\xi \in N_{\mathbf{R}}$. 
\item $\|f\|_{J} \geq 0$ for any $f \in \mathcal{E}_{1}$, and $\|f\|_{J} = 0$ if and only if $f$ is affine. 
\item There exists a constant $C_{1} > 0$ such that 
\begin{align*}
\|f\|_{J} \leq \|f\|_{L^{1}(P)} \leq C_{1}\|f\|_{J}
\end{align*}
for any $f \in \widetilde{\mathcal{E}}_{1}$. 
\item If a sequence $\{f_{i}\}_{i=1}^{\infty}$ in $\mathcal{E}_{1}$ converges uniformly to $f \in \mathcal{E}_{1}$, then 
\begin{align*}
\lim_{i \to \infty}\|f_{i}\|_{J} = \|f\|_{J}. 
\end{align*}
\end{enumerate}
\end{proposition}
\begin{proof}
The claim (1) is clear, and (2) is easily deduced from (1) and (3). 
Let us show the claim (3). 
The first inequality is obvious. 
For the latter inequality, it is sufficient to show that there exists a constant $\kappa > 0$ such that 
\begin{align*}
\|f\|_{J} \geq \kappa
\end{align*}
for any $f \in \widetilde{\mathcal{E}}_{1}$ with $\|f\|_{L^{1}(P)} = 1$. 
For contradiction, suppose that the conclusion does not hold. 
Then there is a sequence $\{f_{i}\}_{i=1}^{\infty}$ in $\widetilde{\mathcal{E}}_{1}$ such that 
\begin{enumerate}[\upshape(i)]
\item $\|f_{i}\|_{L^{1}(P)} = 1$ for any $i \in \mathbf{Z}_{> 0}$, and 
\item ${\displaystyle \lim_{i \to \infty}\|f_{i}\|_{J} = 0}$. 
\end{enumerate}
By condition (ii), there is a sequence of affine functions $\{\ell_{i}\}_{i=1}^{\infty}$ such that 
\begin{enumerate}[\upshape(i)]
\setcounter{enumi}{2}
\item ${\displaystyle \inf_{P}(f_{i}+\ell_{i}) = 0}$ for any $i \in \mathbf{Z}_{> 0}$, and 
\item ${\displaystyle \int_{P}(f_{i}+\ell_{i})\,dx = 0}$. 
\end{enumerate}
By conditions (i), (ii), (iii), and (iv) and \cite[Corollary 5.2.5]{Don02}, we may assume both $\{f_{i}\}_{i=1}^{\infty}$ and $\{f_{i}+\ell_{i}\}_{i=1}^{\infty}$ converge locally uniformly on $P^{\circ}$. 
Then the sequence $\{\ell_{i}\}_{i=1}^{\infty}$ converges locally uniformly on $P^{\circ}$. 
Moreover, since every $\ell_{i}$ is affine, the limit $\ell$ is also affine and the convergence $\ell_{i} \to \ell$ is uniform. 
Let 
\begin{align*}
f(x) \coloneqq \lim_{i \to \infty}f_{i}(x),\quad \ell(x) \coloneqq \lim_{i \to \infty}\ell_{i}(x)
\end{align*}
for each $x \in P^{\circ}$. 
By (iv) above, $f + \ell = 0$ on $P^{\circ}$ and hence $f$ is affine on $P^{\circ}$. 
Moreover, since $\inf_{P^{\circ}}f = f(0) = 0$ we obtain 
\begin{align*}
f = -\ell = 0. 
\end{align*}
Therefore, we have 
\begin{align*}
1 = \|f_{i}\|_{L^{1}(P)} \leq \|f_{i} + \ell_{i}\|_{L^{1}(P)} + \|-\ell_{i}\|_{L^{1}(P)} \to 0\quad (\text{as $i \to \infty$}), 
\end{align*}
which is contradiction. 

Finally, we show the claim (4). 
For each $\varepsilon > 0$, choose $N \in \mathbf{Z}_{> 0}$ so that 
\begin{align*}
\sup_{P}|f - f_{i}| < \frac{\varepsilon}{4\mathrm{vol}(P)}
\end{align*}
for any $i \in \mathbf{Z}_{\geq N}$. 
Then for any $i \in \mathbf{Z}_{\geq N}$ and an affine function $\ell$ we have
\begin{align*}
&\left|\left(\int_{P}(f + \ell)\,dx - \mathrm{vol}(P)\inf_{P}(f+\ell)\right) - \left(\int_{P}(f_{i} + \ell)\,dx - \mathrm{vol}(P)\inf_{P}(f_{i}+\ell)\right)\right| \\
&\leq \left|\int_{P}(f + \ell)\,dx - \int_{P}(f + \ell)\,dx\right| + \mathrm{vol}(P)\left|\inf_{P}(f+\ell) - \inf_{P}(f_{i}+\ell)\Big)\right| \\
&\leq \left|\int_{P}(f -f_{i})\,dx\right| + \mathrm{vol}(P)\sup_{P}|f-f_{i}| \\
&< \frac{\varepsilon}{4} + \frac{\varepsilon}{4} = \frac{\varepsilon}{2}. 
\end{align*}
By choosing $\ell$ so that 
\begin{align*}
\int_{P}(f + \ell)\,dx - \mathrm{vol}(P)\inf_{P}(f+\ell) - \|f\|_{J} < \frac{\varepsilon}{2}, 
\end{align*}
we obtain 
\begin{align}\label{red_J_est_1}
\begin{split}
\|f_{i}\|_{J} 
&\leq \int_{P}(f_{i} + \ell)\,dx - \mathrm{vol}(P)\inf_{P}(f_{i}+\ell) \\
&< \int_{P}(f + \ell)\,dx - \mathrm{vol}(P)\inf_{P}(f+\ell) + \frac{\varepsilon}{2} \\
&< \|f\|_{J} + \varepsilon. 
\end{split}
\end{align}
Similarly, we have 
\begin{align}\label{red_J_est_2}
\|f\|_{J} 
< \|f_{i}\|_{J} + \varepsilon. 
\end{align}
By \eqref{red_J_est_1} and \eqref{red_J_est_2}, we obtain $|\|f\|_{J} - \|f_{i}\|_{J}| < \varepsilon$, as required. 
\end{proof}

Similar to the J-norm, \emph{the reduced $L^{1}$-norm} of a convex function $f \in \mathcal{E}_{1}$ is defined by 
\begin{align*}
\|f\|_{1, T} \coloneqq \inf_{\text{$\xi \colon$affine}}\int_{P}|(f+ \xi) - (\overline{f+\xi})|\,dx. 
\end{align*}
The similar claims as (1), (2), (3) of Proposition \ref{red_J_estimate} are valid for the $J$-norm. 
In particular, the reduced $L^{1}$-norm and the $J$-norm are equivalent in the following sense. 
\begin{proposition}\label{equiv_L1J}
There are constants $C_{2}, C_{3} > 0$ such that 
\begin{align*}
C_{2}\|f\|_{1, T} \leq \|f\|_{J} \leq C_{3}\|f\|_{1, T}
\end{align*}
for any $f \in \mathcal{E}_{1}$. 
\end{proposition}

An advantage of the reduced $L^{1}$-norm is the following $L^{1}$-continuity. 
\begin{proposition}\label{L^1conti}
If a sequence $\{f_{i}\}_{i=1}^{\infty}$ in $\mathcal{E}_{1}$ converges to $f \in \mathcal{E}_{1}$ in $L^{1}$, then 
\begin{align*}
\lim_{i \to \infty}\|f_{i}\|_{1, T} = \|f\|_{1, T}. 
\end{align*}
\end{proposition}
\begin{proof}
For each $\varepsilon > 0$, choose $N \in \mathbf{Z}_{> 0}$ so that 
\begin{align*}
\|f - f_{i}\|_{L^{1}(P)} < \frac{\varepsilon}{2}
\end{align*}
for any $i \in \mathbf{Z}_{\geq N}$. 
Then for any $i \in \mathbf{Z}_{\geq N}$ and an affine function $\ell$ we have
\begin{align*}
|\|f + \ell\|_{L^{1}(P)} - \|f_{i} + \ell\|_{L^{1}(P)}|
\leq \|f - f_{i}\|_{L^{1}(P)} 
< \frac{\varepsilon}{2}. 
\end{align*}
By choosing $\ell$ so that 
\begin{align*}
|\|f + \ell\|_{L^{1}(P)} - \|f\|_{1, T} < \frac{\varepsilon}{2}, 
\end{align*}
we obtain 
\begin{align}\label{red_L^1_est_1}
\|f_{i}\|_{1, T} 
\leq \|f_{i} + \ell\|_{L^{1}(P)} 
< \|f + \ell\|_{L^{1}(P)} + \frac{\varepsilon}{2} 
< \|f\|_{1, T} + \varepsilon. 
\end{align}
Similarly, we have 
\begin{align}\label{red_L^1_est_2}
\|f\|_{1, T} 
< \|f_{i}\|_{1, T} + \varepsilon. 
\end{align}
By \eqref{red_L^1_est_1} and \eqref{red_L^1_est_2}, we obtain $|\|f\|_{T, 1} - \|f_{i}\|_{T, 1}| < \varepsilon$, as required. 
\end{proof}
It is not known whether the similar result as Proposition \ref{L^1conti} are valid for the $J$-norm or not. 

\subsection{Proof of Theorem \ref{equiv_strengthenings}}
In this subsection, we give a proof of Theorem \ref{equiv_strengthenings}. 
For $(b)_{\mathcal{F}} \Rightarrow (J)_{\mathcal{F}}$, let $\delta > 0$ satisfying 
\begin{align*}
L_{V}(f) \geq \delta\int_{\partial P}f\,d\sigma
\end{align*}
for any $f \in \widetilde{\mathcal{F}}$. 
Let $f \in \mathcal{F}$, and choose an affine function $\ell$ so that $f + \ell \in \widetilde{\mathcal{F}}$. 
Then, by Propositions \ref{integral_bound_C*} and \ref{red_J_estimate} we have 
\begin{align*}
L_{V}(f) 
&= L_{V}(f + \ell) 
\geq \delta \int_{\partial P}(f + \ell)\,d\sigma \\
&\geq \left(\frac{n+1}{d}\right)\delta\int_{P}(f + \ell)\,dx \\
&\geq \left(\frac{n+1}{2d}\right)\delta\|f + \ell\|_{J} \\ 
&= \left(\frac{n+1}{2d}\right)\delta\|f\|_{J}, 
\end{align*}
as required. 

For $(\mathrm{K})_{\mathcal{C}_{\ast}} \Rightarrow (b)_{\mathcal{F}}$, first note that 
\begin{align*}
L_{V}(f) \geq 0 
\end{align*}
for any $f \in \widetilde{\mathcal{C}}_{*}$. 
For contradiction, suppose that $P$ does not satisfy $(b)_{\mathcal{F}}$. 
Then there exists a sequence $\{f_{i}\}_{i=1}^{\infty}$ in $\widetilde{\mathcal{F}}$ such that 
\begin{enumerate}[\upshape(a)]
\item ${\displaystyle \int_{\partial P}f_{i}\,d\sigma = 1}$ for any $i \in \mathbf{Z}_{>0}$, and  
\item ${\displaystyle \lim_{i \to \infty}L_{V}(f_{i}) = 0}$. 
\end{enumerate}
By the condition (a), we may further assume that there exists $f \in \widetilde{\mathcal{C}}_{*}$ which satisfies the conditions (i), (ii) of Proposition \ref{cptness}. 
Since 
\begin{align*}
0 
\leq L_{V}(f) 
\leq \liminf_{i \to \infty}L_{V}(f_{i}) = 0, 
\end{align*}
we have $L_{V}(f) = 0$ and that $f$ is affine. 
Since $f \in \widetilde{\mathcal{C}}_{*}$, $f$ must be $0$, which contradicts the choice of $f$. 

The equivalence $(\mathrm{K})_{\mathcal{C}_{\ast}} \Leftrightarrow (\mathrm{K})_{\mathcal{E}_{1}}$ can be obtained as follows. 
The implication $(\mathrm{K})_{\mathcal{E}_{1}} \Rightarrow (\mathrm{K})_{\mathcal{C}_{\ast}}$ is clear. 
Let us assume $(\mathrm{K})_{\mathcal{C}_{\ast}}$ and $f \in \mathcal{E}_{1}$. 
If $f \in \mathcal{C}_{\ast}$, then $L_{V}(f) \geq 0$ and $L_{V}(f) = 0$ if and only if $f$ is affine by the assumption $(\mathrm{K})_{\mathcal{C}_{\ast}}$. 
If $f \in \mathcal{E}_{1} \setminus \mathcal{C}_{\ast}$, then $L_{V}(f) = +\infty > 0$ by Proposition \ref{C*_NA_K_energy}. 

Finally, let us show $(J)_{\mathcal{F}} \Rightarrow (\mathrm{K})_{\mathcal{C}_{\ast}}$ in each case. 

Case 1: For the case $\mathcal{F} = \mathcal{E}_{1}$ or $\mathcal{F} = \mathcal{C}_{\ast}$, it is obvious from Propositions \ref{C*_NA_K_energy} and \ref{red_J_estimate}. 

Case 2: Let $\mathcal{F} = \mathcal{C}_{PL}$. 
It is sufficient to show that $L_{V}(f) \geq 0$ for any $f \in \widetilde{\mathcal{C}}_{*}$ and $L_{V}(f) = 0$ if and only if $f$ is affine. 
Choose a constant $\delta > 0$ so that 
\begin{align*}
L_{V}(f) \geq \delta\|f\|_{J}
\end{align*}
for any $f \in \mathcal{C}_{PL}$. 
Then by Proposition \ref{equiv_L1J} we have 
\begin{align*}
L_{V}(f) \geq \delta C_{2}\|f\|_{1, T}. 
\end{align*}
Let $f \in \widetilde{\mathcal{C}}_{*}$, and choose a sequence $\{f_{i}\}_{i=1}^{\infty}$ in $\mathcal{C}_{PL}$ satisfying the conditions (i), (ii), (iii) of Proposition \ref{appr_PL}. 
By conditions (i) and (iii), we have $f_{i} \in \widetilde{\mathcal{C}}_{PL}$ for any $i \in \mathbf{Z}_{> 0}$ and ${\displaystyle \lim_{i \to \infty}L_{V}(f_{i}) = L_{V}(f)}$. 
Since 
\begin{align*}
L_{V}(f_{i}) \geq \delta C_{2}\|f_{i}\|_{1, T} \geq 0
\end{align*}
for any $i \in \mathbf{Z}_{>0}$, we have
\begin{align*}
L_{V}(f) \geq \delta C_{2}\|f\|_{1, T} \geq 0 
\end{align*}
by Proposition \ref{L^1conti}. 
Suppose $L_{V}(f) = 0$. 
Then $\|f\|_{1, T} = 0$ and hence $f$ is affine. 

Case 3: Let $\mathcal{F} = \mathcal{C}_{PL}^{\mathbf{Q}}$. 
Then then we can reduce to the Case 2 as follows. 
Choose a constant $\delta > 0$ so that 
\begin{align*}
L_{V}(f) \geq \delta\|f\|_{J}
\end{align*}
for any $f \in \mathcal{C}_{PL}^{\mathbf{Q}}$. 
Let $f \in \mathcal{C}_{PL}$. 
By Proposition \ref{apprQPL}, there is a sequence $\{f_{i}\}_{i=1}^{\infty}$ in $\mathcal{C}_{PL}^{\mathbf{Q}}$ which converges uniformly to $f$ on $P$. 
Since 
\begin{align*}
L_{V}(f_{i}) \geq \delta \|f_{i}\|_{J}
\end{align*}
for any $i \in \mathbf{Z}_{>0}$, we have
\begin{align*}
L_{V}(f) \geq \delta\|f\|_{J}
\end{align*}
by Proposition \ref{red_J_estimate}. 
Hence $P$ satisfies the condition $(J)_{\mathcal{C}_{PL}}$. 

Case 4: Let $\mathcal{F} = \mathcal{C}_{\infty}$. 
Then, by using Proposition \ref{appr_smooth} in place of Proposition \ref{appr_PL}, we can also prove $(J)_{\mathcal{C}_{\infty}} \Rightarrow (\mathrm{K})_{\mathcal{C}_{\ast}}$ by the similar argument as Case 2. 

Case 5: Finally, let $\mathcal{F} = \mathcal{S}$. 
As well as Case 2, it is sufficient to show that $L_{V}(f) \geq 0$ for any $f \in \widetilde{\mathcal{C}}_{*}$ and $L_{V}(f) = 0$ if and only if $f$ is affine. 
Choose $\delta > 0$ so that 
\begin{align*}
L_{V}(u) \geq \delta\|u\|_{J} 
\end{align*}
for any $u \in \mathcal{S}$. 
Then by Proposition \ref{equiv_L1J} we have 
\begin{align*}
L_{V}(u) \geq \delta C_{2}\|u\|_{1, T}. 
\end{align*}
Let $f \in \widetilde{\mathcal{C}}_{*}$, and choose a sequence $\{f_{i}\}_{i=1}^{\infty}$ in $\mathcal{C}_{\infty} \cap C^{\infty}(P)$ satisfying the conditions (i), (ii), (iii) of Proposition \ref{appr_smooth}. 
Fix $u_{0} \in \widetilde{\mathcal{S}}$ arbitrarily. 
Then, for each $i \in \mathbf{Z}_{> 0}$ and $t \in (0, \infty)$, $u_{0}+tf_{i} \in \mathcal{S}$ and 
\begin{align*}
L_{V}\left(\frac{u_{0}+tf_{i}}{t}\right) 
\geq \delta C_{2}\left\|\frac{u_{0}+tf_{i}}{t}\right\|_{1, T} 
\geq 0. 
\end{align*}
By taking the limit as $t \to \infty$ and $i \to \infty$, we obtain 
\begin{align}\label{section3_DF_upbd}
L_{V}(f) \geq \delta C_{2}\|f\|_{1, T} \geq 0. 
\end{align}
Now suppose $L_{V}(f) = 0$. 
Then $\|f\|_{1, T} = 0$ and hence $f$ is affine. 

\begin{remark}
Let $\mathcal{B}$ be the set of all bounded lower semicontinuous convex functions defined on the whole of $P$. 
Then $\mathcal{B}$ is contained in $\mathcal{C}_{\ast}$ and consequently $(\mathrm{K})_{\mathcal{C}_{\ast}}$ implies $(\mathrm{K})_{\mathcal{B}}$. 
The latter condtion $(\mathrm{K})_{\mathcal{B}}$ can be found in \cite{LLS}.
\end{remark}

\section{Proof of Theorem \ref{suff}}
Finally in this section, we prove Theorem \ref{suff}. 
By the implication $(b)_{\mathcal{C}_{PL}^{\mathbf{Q}}} \Rightarrow (J)_{\mathcal{C}_{PL}^{\mathbf{Q}}}$ in Theorem \ref{equiv_strengthenings}, it is enough to show that $P$ satisfies the condition $(b)_{\mathcal{C}_{PL}^{\mathbf{Q}}}$. 
First suppose that $P$ satisfies 
\begin{align*}
\overline{s} + \max_{P}V \leq \frac{n+1}{d}. 
\end{align*}
Define a constant $\delta\ge 0$ by
\begin{align*}
\delta = 1 - \frac{d}{n+1}(\overline{s}+\max_{P}V). 
\end{align*}
Then, by Proposition \ref{integral_bound_C*} we have 
\begin{align*}
L_{V}(f) 
&= \int_{\partial P}f\,d\sigma - \int_{P}(\overline{s}+V)f\,dx \\
&\geq \int_{\partial P}f\,d\sigma - (\overline{s}+\max_{P}V)\int_{P}f\,dx \\
&\geq \int_{\partial P}f\,d\sigma - \frac{d}{n+1}(\overline{s}+\max_{P}V)\int_{\partial P}f\,d\sigma \\
&= \delta \int_{\partial P}f\,d\sigma
\end{align*}
 for any $f \in \widetilde{\mathcal{C}}_{PL}^{\mathbf{Q}}$. 
Hence, if $P$ satisfies \eqref{unif.stab} then $\delta > 0$ and $P$ satisfies the condition $(b)_{\mathcal{C}_{PL}^{\mathbf{Q}}}$. 
Now suppose that $P$ satisfies \eqref{unif.rel.stab}. 
In this case $\delta$ might be $0$, but at least we have 
\begin{align*}
L_{V}(f) \geq \delta \int_{\partial P}f\,d\sigma \geq 0 
\end{align*}
for any $f \in \widetilde{\mathcal{C}}_{PL}^{\mathbf{Q}}$. 
For contradiction, suppose that our conclusion does not hold. 
Then there exists a sequence $\{f_{i}\}_{i=1}^{\infty}$ in $\widetilde{\mathcal{C}}_{PL}^{\mathbf{Q}}$ such that 
\begin{align}\label{unitbdrL1}
\int_{\partial P}f_{i}\,d\sigma = 1
\end{align}
and 
\begin{align}\label{conv0}
\lim_{i \to \infty}L_{V}(f_{i}) = 0. 
\end{align}
Furthermore, we can find a subsequence $\{f_{i_{k}}\}_{k=1}^{\infty}$ of $\{f_{i}\}_{i=1}^{\infty}$ and $f \in \widetilde{\mathcal{C}}_{*}$ satisfying the conditions (i), (ii), (iii) of Proposition \ref{cptness}. 
Hence we have
\begin{align}\label{estDF}
L_{V}(f) 
\leq \liminf_{k \to \infty}L_{V}(f_{i_{k}}) 
= \lim_{k \to \infty}L_{V}(f_{i_{k}}) 
= 0. 
\end{align}
Now suppose $f \not \equiv 0$ on $P^{*}$. 
Then $U = \{x \in P^{*} \mid f(x) > 0\}$ is a nonempty open set of $P^{*}$. 
Since $V$ is a nonconstant affine function, it holds that $V < \max_{P}V$ almost everywhere on $P$. 
It follows that 
\begin{align*}
L_{V}(f) 
&= \int_{\partial P}f\,d\sigma - \int_{P}(\overline{s}+V)f\,dx \\
&> \int_{\partial P}f\,d\sigma - (\overline{s}+\max_{P}V)\int_{U}f\,dx \\
&\geq \delta \int_{\partial P}f\,d\sigma \\
&\geq 0, 
\end{align*}
which contradicts to \eqref{estDF}. 
Hence we have $f \equiv 0$ on $P^{*}$, and 
\begin{align}\label{convDF}
\lim_{k \to \infty}\int_{P}(\overline{s} + V)f_{i_{k}}\,dx = 0. 
\end{align}
On the other hand, by \eqref{unitbdrL1} and \eqref{conv0} we have 
\begin{align*}
\lim_{k \to \infty}\int_{P}(\overline{s} + V)f_{i_{k}}\,dx = 
\lim_{k \to \infty}\left(\int_{\partial P}f_{i_{k}}\,d\sigma - L_{V}(f_{i_{k}})\right) = 1, 
\end{align*}
which contradicts to \eqref{convDF}. 


\end{document}